\documentclass[12pt]{amsart}
\usepackage[utf8]{inputenc}
\usepackage[T1]{fontenc}
\usepackage{amsmath,amsfonts,amssymb}

\usepackage[all]{xy}
\usepackage{tikz}  
\usepackage{url}
\usepackage{enumerate}
\usepackage{vmargin}

\usepackage{pstricks}
\usepackage{graphicx}

\DeclareMathOperator{\codim}{codim}
\DeclareMathOperator{\OG}{OG}
\DeclareMathOperator{\Gr}{Gr}
\DeclareMathOperator{\IG}{IG}
\DeclareMathOperator{\QH}{QH}
\DeclareMathOperator{\Hr}{H}
\DeclareMathOperator{\Sp}{Sp}
\DeclareMathOperator{\SL}{SL}
\DeclareMathOperator{\GL}{GL}
\DeclareMathOperator{\Hom}{Hom}
\DeclareMathOperator{\Aut}{Aut}
\DeclareMathOperator{\Pic}{Pic}
\DeclareMathOperator{\Span}{Span}
\DeclareMathOperator{\Ker}{Ker}
\DeclareMathOperator{\rk}{rk}

\DeclareMathOperator{\Ext}{Ext}
\DeclareMathOperator{\End}{End}
\DeclareMathOperator{\XX}{X}

\DeclareMathOperator\Ri{R}
\DeclareMathOperator{\Le}{L}
\DeclareMathOperator{\Sing}{Sing}
\DeclareMathOperator{\ev}{ev}
\DeclareMathOperator{\pr}{pr}
\DeclareMathOperator{\Cone}{Cone}
\DeclareMathOperator{\RHom}{\mathop{\mathrm RHom}}

\newcommand{\g}{\mathfrak g}
\newcommand{\gh}{\mathfrak h}
\newcommand{\Xf}{\mathfrak{X}}
\newcommand{\Vf}{\mathfrak{V}}

\newcommand{\id}{\mathrm{id}}

\newcommand{\Qs}{{\sf Q}}

\newcommand{\cF}{{\mathcal F}}
\renewcommand{\L}{{\mathcal L}}
\newcommand{\Ll}{{\mathcal L}}
\newcommand{\Kk}{{\mathcal K}}
\newcommand{\Uu}{{\mathcal U}}
\newcommand{\Oo}{{\mathcal O}}
\newcommand{\Hh}{\mathcal H}
\newcommand{\Ss}{\mathcal S}
\newcommand{\Ee}{{\mathcal E}}
\newcommand{\Cc}{{\mathcal C}}
\newcommand{\Dd}{{\mathcal D}}
\newcommand{\cU}{{\mathcal U}}
\newcommand{\cO}{{\mathcal O}}

\newcommand{\A}{{\mathcal A}}
\newcommand{\B}{{\mathcal B}}
\newcommand{\Gg}{{\mathcal G}}
\newcommand{\bp}{{\mathbf{p}}}

\renewcommand{\P}{{\mathbb P}}
\newcommand{\Pp}{{\mathbb P}}
\newcommand{\G}{{\mathbb{G}}}
\newcommand{\Z}{{\mathbb Z}}
\newcommand{\C}{{\mathbb C}}
\newcommand{\R}{{\mathbb R}}
\newcommand{\Sb}{{\mathbb S}}
\newcommand{\Ub}{{\mathbb U}}
\newcommand{\Eb}{{\mathbb E}}

\newcommand{\SX}{\mathcal{S}_X}
\newcommand{\SXp}{\mathcal{S}_{X^+}}
\newcommand{\QX}{\mathcal{Q}_X}
\newcommand{\QXp}{\mathcal{Q}_{X^+}}

\newcommand{\uh}{{\widehat{u}}}
\newcommand{\vh}{{\widehat{v}}}
\newcommand{\Ssh}{{\widehat{\Ss}}}
\newcommand{\Sbh}{{\widehat{\Sb}}}
\newcommand{\Uuh}{{\widehat{\Uu}}}

\newcommand{\Ft}{\widetilde{F}}

\newcommand{\Xt}{{\widetilde{X}}}
\newcommand{\ut}{{\widetilde{u}}}
\newcommand{\vt}{{\widetilde{v}}}
\newcommand{\ft}{{\widetilde{f}}}
\newcommand{\tildeX}{\widetilde{X}}
\newcommand{\tildeY}{\widetilde{Y}}
\newcommand{\tildeZ}{\widetilde{Z}}
\newcommand{\tildeT}{\widetilde{T}}
\newcommand{\tildeXp}{\widetilde{X}^+}
\newcommand{\tildeV}{\widetilde{V}}
\newcommand{\tildeVp}{\widetilde{V}^+}

\newcommand{\tildeFp}{\widetilde{F}^+}

\newcommand{\tildes}{\tilde{\sigma}}
\newcommand{\tildet}{\tilde{\tau}}
\newcommand{\tildeSp}{\widetilde{\Sigma}^+}
\newcommand{\tildej}{\tilde{\jmath}}
\newcommand{\Sst}{\widetilde \Ss}
\newcommand{\Sbt}{\widetilde \Sb}
\newcommand{\At}{\widetilde \A}
\newcommand{\Uut}{\widetilde \Uu}

\renewcommand{\sb}{{\overline{\s}}}
\newcommand{\taub}{{\overline{\tau}}}

\newcommand{\Lb}{\mathbf{L}}
\newcommand{\Rb}{\mathbf{R}}

\newcommand{\s}{\sigma}
\renewcommand{\a}\alpha
\newcommand{\oZ}{\omega_Z}
\newcommand{\oY}{\omega_Y}
\newcommand{\oa}{\omega_\alpha}
\newcommand{\ob}{\omega_\beta}
\newcommand{\Pa}{P_\alpha}
\newcommand{\pa}{\pi_\alpha}
\newcommand{\Pb}{P_\beta}
\newcommand{\pb}{\pi_\beta}
\newcommand{\Gad}{G_2^\textrm{ad}}

\newcommand{\scal}[1]{\langle #1 \rangle}
\newcommand{\Mod}[3]{M_{#1,#2}(#3)}
\newcommand{\Modc}[3]{M_{#1,#2}^\circ(#3)}

\newtheorem{thm}{Theorem}[section]
\newtheorem{lemma}[thm]{Lemma}
\newtheorem{prop}[thm]{Proposition}
\newtheorem{cor}[thm]{Corollary}

\theoremstyle{definition}
\newtheorem{remark}[thm]{Remark}
\newtheorem{notation}[thm]{Notation}
\newtheorem{fact}[thm]{Fact}

\newtheorem{defn}[thm]{Definition}
\newtheorem{question}[thm]{Question}

\newtheorem*{thm*}{Theorem}
\newtheorem*{cor*}{Corollary}
\newtheorem*{conj*}{Conjecture}

\long\def\comment#1{}

\begin{document}

\setcounter{tocdepth}{1}

\title{{Geometry of horospherical varieties \\
    of Picard rank one}}

\date{18.05.2021}

\author{R.~Gonzales}
\address{Department of Sciences, Pontificia Universidad Cat\'olica del Per\'u, San Miguel, Lima 32, Lima, Per\'u}
\email{rgonzalesv@pucp.edu.pe}

\author{C.~Pech}
\address{School of Mathematics, Statistics and Actuarial Science, University of Kent, Canterbury CT2 7NF, United Kingdom}
\email{c.m.a.pech@kent.ac.uk}

\author{N.~Perrin}
\address{Laboratoire de Math\'{e}matiques de Versailles, UVSQ, CNRS, Universit\'{e} Paris-Saclay,
78035 Versailles, France}
\email{nicolas.perrin@uvsq.fr}

\author{A.~Samokhin}
\address{Institute for Information Transmission Problems of the Russian Academy of Sciences (Kharkevich Institute), B.Karetny per. 19, Moscow 127051, Russia}
\email{alexander.samokhin@gmail.com}



\begin{abstract}
	We study the geometry of smooth non--homogeneous horospherical varieties of Picard rank one. These have been classified by Pasquier and include the well--known odd symplectic Grassmannians. We focus our study on quantum cohomology, with a view towards Dubrovin's conjecture. We start with 
describing the cohomology groups of smooth horospherical varieties of Picard rank one. We show a Chevalley formula for these and establish that many Gromov--Witten invariants are enumerative. This enables us to prove that in many cases the quantum cohomology is semisimple. We give a presentation of the quantum cohomology ring for odd symplectic Grassmannians. In the last sections, we turn to derived categories of coherent sheaves. We first discuss a general construction of exceptional bundles on horospherical varieties. We work out in detail the case of the horospherical variety associated to the exceptional group $G_2$ and construct a full rectangular Lefschetz exceptional collection in the derived category.
\end{abstract}

\maketitle
\tableofcontents

\markboth{R.~GONZALES, C.~PECH, N.~PERRIN, A.~SAMOKHIN}
{\uppercase{{Horospherical varieties of Picard rank one}}}

\section*{Introduction}

Schubert calculus and enumerative geometry are very classical. One of their many aspects is the study of intersection theory of projective rational homogeneous spaces $G/P$ with $G$ a reductive group and $P$ a parabolic subgroup. We now have a very good understanding of the cohomology rings of $G/P$ as well as a good understanding of their quantum cohomology and even of their quantum $K$-theory.

One of the simplest classes of quasi--homogeneous spaces is toric varieties for which many results are known. Rational projective homogeneous spaces and toric varieties are both subsumed by a larger class of spherical varieties. Recall that a normal variety $X$ is called \emph{$G$--spherical}, $G$ being a reductive group, if $X$ has a dense $B$--orbit, where $B \subset G$ is a Borel subgroup. Spherical varieties form such a large class of varieties that there is no yet a complete description of their cohomology rings. In this paper, we concentrate on a special class of spherical varieties: \emph{horospherical varieties}, see Definition~\ref{defi-horo}. More specifically, we consider smooth projective horospherical varieties of Picard rank one. These varieties have been classified by Pasquier \cite{pa:h}, who proved that they are either homogeneous or can be constructed in a uniform way via a triple $(\textrm{Type}(G),\oY,\oZ)$ of representation--theoretic data. Here $\textrm{Type}(G)$ is the semisimple Lie type of a reductive group $G$ and $\oY,\oZ$ are fundamental weights, see Section~\ref{sect-real}. Pasquier has classified the possible triples in five classes:

\begin{enumerate}[(1)]
	\item $(B_n,\omega_{n-1},\omega_n)$ with $n \geq 3$;
	\item $(B_3,\omega_{1},\omega_3)$; 
	\item $(C_n,\omega_{m},\omega_{m-1})$ with $n \geq 2$ and $m \in [2,n]$;
	\item $(F_4,\omega_{2},\omega_3)$;
	\item $(G_2,\omega_{1},\omega_2)$.
\end{enumerate}

Following \cite{BP}, we denote by $X^1(n), X^2, X^3(n,m), X^4$ and $X^5$ the corresponding varieties, the superscript indicating the class that a variety belongs to. We will refer to this list of varieties as Pasquier's list.

The variety $X^3(n,m)$ from the third class has already attracted some attention since in this case it is isomorphic to the so--called \emph{odd symplectic Grassmannian} $\IG(m,2n+1)$ (see Subsection~\ref{ss:special} and Section \ref{odd}). They were studied, for instance, in \cite{m:o,pech:t,pech:q} and very recently in \cite{odd:mih1,odd:mih2}.
  
In this paper, we start a systematic study of the geometry of non--homogeneous projective smooth horospherical varieties $X$ with $\Pic(X) = \Z$, with a view towards Dubrovin's conjecture. Our first results are concerned with the cohomology ring $\Hr^*(X,\Z)$. We describe two \emph{Schubert--like} cohomology basis $(\sigma'_u,\tau_v)_{u,v}$ and $(\sigma_u,\tau'_v)_{u,v}$ and prove that these bases are Poincar\'e dual to each other, see Proposition \ref{prop:poincare}. We also give a Chevalley formula, \emph{i.e.}, a formula for multiplying any of the above classes with the hyperplane class $h$ generating the Picard group, see Proposition \ref{prop-chev}.

We then turn to the study of rational curves, in particular to the study of $\Mod{d}{k}{X}$, the Kontsevich moduli space of degree $d$, genus zero, stable maps with $k$ marked points. In particular, we prove the following result (Theorem \ref{thm:irr}):

\begin{thm*}
    Let $X$ be a non--homogeneous projective smooth horospherical variety with $\Pic(X) = \Z$. Then
    \begin{enumerate}
    	\item $\Mod{1}{k}{X}$ is irreducible of dimension $\dim X + c_1(X) + k - 3$.
    	\item $\Mod{d}{k}{X}$ is irreducible of dimension $\dim X + dc_1(X) + k - 3$ for all $d$ if $X=X^2$ or $X=X^3(n,m)$.
    	
    	\item $\Mod{2}{k}{X}$ is irreducible of dimension $\dim X + 2c_1(X) + k - 3$ for $X=X^1(3)$ and $X=X^4$, and has two irreducible components of dimension $\dim X + 2 c_1(X) + k - 3$ if $X=X^1(4)$ or $X=X^5$.
    	
    \end{enumerate}
\end{thm*}

In all cases above, the dimension of the moduli space $\Mod{d}{k}{X}$ is equal to the expected dimension $\dim X + d c_1(X) + k-3$. This implies that the Gromov--Witten invariants are enumerative and proves, in particular, a conjecture stated in the second author's PhD thesis (see Corollary \ref{cor:enum}). It is worth noting that the moduli space $\Mod{d}{k}{X}$ is not irreducible and/or has non--expected dimension whenever $c_1(Z) > c_1(X)$ and $d$ is large enough. Here $Z$ is the unique closed $Aut(X)$-orbit in $X$(see Corollary 2.3 and the proof of Theorem \ref{thm:irr}).

\begin{cor*}
    If $X=X^2$ or $X=X^3(n,m)$, then the Gromov--Witten invariants of $X$ are enumerative.  
\end{cor*}

Using these enumerativity results, we are able to compute many Gromov--Witten invariants. In particular, we prove the following result, see Sections \ref{subsec-chev}  and \ref{example}:

\begin{thm*}
    If $X$ is one of the varieties $X^1(3), X^2, X^3(3,3)$ and $X^5$, then
    
    \begin{enumerate}
    	\item There is an explicit quantum Chevalley formula.
    	\item The small quantum cohomology ring $\QH(X)$ is semisimple.
    \end{enumerate}
    
\end{thm*}

Note that for the varieties $X^3(n,m)$, \emph{i.e.} for odd symplectic Grassmannians, such a quantum (and even equivariant quantum) Chevalley formula has been obtained very recently by Mihalcea and Shifler \cite{odd:mih1}. In Section \ref{odd}, we use our proof of the conjecture by the second author to deduce several results for the specific case of odd symplectic Grassmannians (case (3) of Pasquier's list). As a corollary, we prove quantum--to--classical results that enable us to compute Gromov--Witten invariants on $\IG(m,2n+1)$ starting from cohomological (thus classical) invariants on $\IG(m+1,2n+3)$ (see Theorem \ref{t:quant-class}). In particular, this result contains the quantum Pieri rule even though at the moment there is no any closed formula at our disposal. Using these results, we deduce a presentation for the quantum cohomology ring $\QH(\IG(m,2n+1))$. Let $(\tau'_p)_{p \in [1,2n+1-m]}$ be the Chern classes of the tautological quotient bundle on $\IG(m,2n+1)$ and for $r \geq 1$, define
\[
	d_r := \det(\tau_{1+j-i}')_{1 \leq i,j \leq r} \quad \textrm{and} \quad b_r:= (\tau_r')^2 + 2 \sum_{i \geq 1} (-1)^i \tau_{r+i}' \tau_{r-i}',
\]
with the convention that $\tau_0' := 1$ and $\tau_p' := 0$ if $p<0$ or $p>2n+1-m$. We prove (Theorem \ref{t:quantpres}):

\begin{thm*}
	The quantum cohomology ring $\QH(\IG(m,2n+1),\C)$ is generated by the classes $(\tau_p')_{1 \leq p \leq 2n+1-m}$ and the quantum parameter $q$, and the relations are
	\[
		{\renewcommand{\arraystretch}{2.2}
			\begin{array}{lcl}
				d_r & = & 0 \text{ for $m+1 \leq r \leq 2n+1-m$,} \\
				d_{2n+2-m} & = & \begin{cases}
							0 & \text{if $m$ is even,} \\
							-q & \text{if $m$ is odd,}
								 \end{cases} \\
				b_s & = & (-1)^{2n+1-m-s} q \tau_{2s-2n-2+m}' \text{ for $n+2-m \leq s \leq n$,}
			\end{array}}
	\]
\end{thm*}

Finally, let $\mathrm{BQH}(X)$ denote the big quantum cohomology ring of $X$. Recall the Dubrovin conjecture:

\begin{conj*}
	Let $X$ be a smooth projective Fano variety. The ring $\mathrm{BQH}(X)$ is generically semisimple if and only if the bounded derived category $\Dd^b(X)$ has a full exceptional collection.
\end{conj*}

The Dubrovin conjecture and the above results on quantum cohomology therefore indicate the existence of full exceptional collections on many horospherical varieties of the classification. Some results in this direction have already been obtained. For instance, the second author proved in \cite{pech:q} that Dubrovin's conjecture holds for $X^3(n,2)$. Dubrovin's conjecture for $X^2$ follows from results by the third author \cite{pe:ss} for the semisimplicity of quantum cohomology and from general results by Kuznetsov \cite{Ku} about the homological projective duality. Indeed, the variety $X^2$ is a hyperplane section of the maximal orthogonal Grassmannian $\OG(5,10)$ (see Subsection \ref{ss:special}). In the last two sections, we consider the derived category side of Dubrovin's conjecture. In Section \ref{section:exc-gen}, for all cases of Pasquier's list except the case (2) we construct an exceptional pair of vector bundles on $X$. In Section \ref{sec:g2}, we extend that construction and produce a full exceptional  collection on the horospherical variety $X^5$. In more detail, we find two exceptional objects $\Ub$ and $\Sbh$ of the derived category $\Dd ^b(X^5)$ which can be completed to an exceptional triple $\A =\{\Ub,\Sbh,\Oo _{X^5}\}$. Here the object $\Ub$ is a vector bundle, but $\Sbh$ is a two--term complex. It turns out that $\A$ is the starting block of a Lefschetz exceptional collection on $X^5$ in the sense of \cite{KuHPD}. More precisely, we prove (Theorem \ref{th:Lefexccoll}):

\begin{thm*}
	The collection 
\begin{equation}\label{eq:Lefschetz_exc_coll_G_2hor-intro}
\scal{\A \otimes \Oo _{X^5}(-3),\A \otimes \Oo _{X^5}(-2),\A \otimes \Oo _{X^5}(-1),\A}
\end{equation}
	
is a full rectangular Lefschetz exceptional collection on $X^5$.
\end{thm*}

\noindent
    {\bf Notation.}
In this paper, $G$ denotes a complex connected reductive group. Choosing a maximal torus $T$ in $G$, let $R$ denote the set of roots with respect to $T$. Given a root $\alpha \in R$, denote $\g_\alpha$ the root subspace so that we have a root decomposition $\g = \gh \bigoplus_{\alpha \in R} \g_\alpha$ where $\gh = {\rm Lie}(T)$. We choose a Borel subgroup $B \supset T$ and declare the roots of $B$ to be the negative roots $R^- \subset R$. Denote $R^+ = -R^-$ the positive roots, so that $R = R^+ \cup R^-$. Denote $S \subset R^+$ the subset of simple roots. If needed, the Borel subgroup corresponding to $R^+$ will be denoted by $B^+$. A standard parabolic subgroup is a closed subgroup $P \subset G$ containing $B$. We denote $G/B$ the full flag variety of $G$. Denote $X(T)$ the weight lattice and $X^+(T)$ the set of dominant weights relative to positive roots $R^+$ \emph{i.e.} $\lambda \in X^+(T)$ if $\scal{\lambda,\alpha^\vee} \geq 0$ for any simple coroot $\alpha^\vee$. Given a parabolic subgroup $P$ and a $P$–variety $F$ (\emph{i.e.}, a variety equipped with a left action of $P$), define the crossed product $G \times^P F$ to be the quotient of $G \times F$ modulo the equivalence relation $(g,v) = (gp,p^{-1}\cdot v)$ for $g \in G$, $p \in P$ and $v \in F$. Given a character $\chi \in X(T)$, set $\C_\chi$ to be the corresponding one–dimensional $B$–module obtained by the restriction $B \to T$. We denote $\mathcal{L}_\chi$ the $G$–equivariant line bundle on $G/B$ defined as the quotient of $G \times^B \C_\lambda$. Given a dominant weight $\lambda \in X^+(T)$, we denote $V(\lambda)$ the induced module ${\rm Ind}_B^G(\C_\lambda)$, the highest weight irreducible representation of $G$ of highest weight $\lambda$. Given a character $\chi \in X(T)$, let $P$ be a standard parabolic subgroup and $L$ its Levi subgroup containing $T$. Set $V_P(\chi) = {\rm Ind}_{B\cap L}^L(\chi)$ seen as a $P$--representation via the surjective map $P \to L$. Note that we have $G \times^P V_P(\chi) = \pi_*\mathcal{L}_\chi$ where $\pi : G/B \to G/P$ is the canonical map.

\begin{remark}
Our choice for $B$ follows the book \cite{jantzen} which is better suited to the geometric framework of sheaf cohomology on the flag variety. However, this convention is not universal in the literature, and there is another commonly used notation (for instance, in \cite{pa:h}, our main source for the classification of smooth  horospherical varieties of Picard rank one, where the Borel subgroup $B$ is chosen to be associated to positive roots). In that case, the line bundle $\mathcal{L}_\lambda$ has to be induced from the representation $\C_{-\lambda}$. Our convention following \cite{jantzen} allows to avoid the sign problems with weights and to state
Borel--Weil--Bott's theorem in a simpler form.
\end{remark}

\noindent
{\bf Acknowledgements.} R.G., N.P., and A.S. would like to thank the Mathematical Institute of D\"usseldorf where this work first started. C.P. thanks the University of Kent for support. R.G. and N.P. were supported by a DFG grant Project “Gromov--Witten Theorie, Geometrie und Darstellungen” (PE 2165/1--2). N.P. was supported by a public grant as part of the “Investissement d'Avenir” project, reference ANR--11--LABX-0056--LMH, LabEx LMH. A.S. was supported by the strategic research fund of the University of D\"usseldorf (grant SFF F--2015/946--8), by the SFB/TR 45 of the University of Duisburg--Essen, and by Prof. M. Levine's University Fellowship. A.S. also acknowledges the support of the IHES and of the Klaus Tschira 
Stiftung at the final stage of the present work.
He is indebted to A.Kuznetsov who suggested using Hecke modifications to construct exceptional vector bundles in Sections 7--8. Finally, all the authors are grateful to the anonymous referees for their very careful reading of the paper and for the many improvements suggested.





\section{Horospherical varieties of Picard rank one : geometry}

In this section we recall some basic properties of smooth projective horospherical varieties of Picard rank one. Our reference is \cite{pa:h}, see also \cite{pa:these,survey}.

\subsection{Horospherical varieties}

Let $G$ be a complex connected reductive group. A \emph{$G$-variety} is a reduced scheme of finite type over the field of complex numbers $\C$, equipped with an algebraic action of $G$. Some of the simplest $G$--varieties are spherical varieties.

\begin{defn}
	Let $G$ be a reductive group and $B \subset G$ be a Borel subgroup. A $G$--variety $X$ is called \emph{spherical} if $X$ has a dense $B$--orbit.
\end{defn}

Spherical varieties admit many equivalent characterisations. For example, a normal $G$--variety is spherical if and only if it has finitely many $B$--orbits. We refer to \cite{knop,survey,gandini,sanya} and references therein for more on the geometry of spherical varieties or their classification.

We focus on a special case of spherical varieties: \emph{horospherical varieties}. We give two equivalent definitions of horospherical varieties, one from a representation-theoretic point of view, the other more geometric. We refer to Pasquier's PhD thesis \cite{pa:these} for details.

\begin{defn}\label{defi-horo}
	Let $X$ be a $G$--spherical variety and let $H$ be the stabiliser of a point in the dense $G$--orbit in $X$. The variety $X$ is called \emph{horospherical} if $H$ contains a conjugate of $U$, the maximal unipotent subgroup of $G$ contained in the Borel subgroup $B$.
\end{defn}

\begin{remark}
	A more geometric equivalent definition can be given as follows. Let $X$ be a $G$--variety. Then $X$ is horospherical if there exists a parabolic subgroup $P \subset G$, a torus $S$ which is a quotient of $P$, and a toric $S$--variety $F$ such that there exists a diagram of $G$--equivariant morphisms
	\[
  		\xymatrix{\Xt = G \times^P F \ar[r]^-p \ar[d]_-\pi & G/P \\
    X}
   	\]
where $G \times^P F$ is the crossed product, $\pi$ is birational, and $p$ is a fibration with fibers isomorphic to $F$. The \emph{rank} of the horospherical variety $X$ is the dimension of the toric variety $F$, namely, $\rk(X) := \dim F$.
\end{remark}

\subsection{Classification}

Horospherical varieties can be classified using \emph{colored fans}, see \cite{pa:these}. We focus here on smooth projective horospherical varieties of Picard rank one. Pasquier \cite{pa:h} constructed all these varieties as well as their minimal embeddings.

\begin{thm}[Pasquier]\label{theo:class}
	Let $X$ be a smooth projective $G$--horospherical variety of Picard rank one. Then either
  	\begin{enumerate}[(0)]
  		\item $X$ is homogeneous,
  	\end{enumerate}
or $X$ can be constructed in a uniform way from a triple $(\textrm{Type}(G),\oY,\oZ)$ belonging to the following list:
  	\begin{enumerate}[(1)]
  	  	\item $(B_n,\omega_{n-1},\omega_n)$ with $n \geq 3$;
 	 	\item $(B_3,\omega_{1},\omega_3)$; 
  		\item $(C_n,\omega_{m},\omega_{m-1})$ with $n \geq 2$ and $m \in [2,n]$;
 	 	\item $(F_4,\omega_{2},\omega_3)$;
  		\item $(G_2,\omega_{1},\omega_2)$,
  	\end{enumerate}
where $\textrm{Type}(G)$ is the semisimple Lie type of $G$ and $\oY,\oZ$ are fundamental weights.
\end{thm}

As stated in the above theorem, some smooth projective horospherical varieties of Picard rank one are homogeneous (for instance, projective spaces). We will not consider them since their moduli space of curves, quantum cohomology and quantum $K$-theory have been studied before (see~\cite{thomsen,kim-pandharipande,pe:courbes,buch:grass,KT1,KT2,BKT1,BKT2,CMP1,CMP2,CMP3,CMP4,cp:qk,BCMP,BCMP2} and references therein). We will therefore focus on the second part of the list, \emph{i.e.}, on the cases from (1) to (5). In what follows, we denote by $X^1(n), X^2, X^3(n,m), X^4$ and $X^5$ the corresponding varieties.

\subsection{Construction}\label{sect-real}

From now on, let $X$ be a smooth projective non--homogeneous horospherical variety of Picard rank one with associated triple $(\textrm{Type}(G),\oY,\oZ)$. We introduce the following notation which will be used throughout the paper.

\begin{notation}
	Let $G$ be the reductive group for which $X$ is horospherical. Then we denote by $V_Y:=V(\oY)$ and $V_Z:=V(\oZ)$ the irreducible $G$--representations with highest weights $\oY$ and $\oZ$, and we let $v_Y$ and $v_Z$ be the corresponding lowest weight vectors.

	We write $P_Y$ for the stabiliser of $[v_Y]$ in $\P(V_Y)$ and $P_Z$ for the stabiliser of $[v_Z]$ in $\P(V_Z)$. Both stabilizers are standard parabolic subgroups of $G$. We denote by $Y$ the $G$--orbit of $[v_Y]$ in $\P(V_Y)$ and by $Z$ that of $[v_Z]$ in $\P(V_Z)$, so that $Y \simeq G/P_Y$ and $Z \simeq G/P_Z$.
\end{notation}

Pasquier \cite{pa:h} proves that $X$ 
is obtained as follows.

\begin{prop}[Pasquier]
	Let $X$ be one of the varieties in cases (1) to (5) of Theorem \ref{theo:class}. Then $X = \overline{G\cdot[v_Y+v_Z]} \subset \P(V_Y\oplus V_Z)$ is the closure of the $G$--orbit $G\cdot~[v_Y+~v_Z]$ in $\P(V_Y\oplus V_Z)$.
\end{prop}

\subsection{Orbits}

Let $\Aut(X)$ be the automorphism group of $X$. The group $G$ acts on $X$ with three orbits: $Y$, $Z$, and the complement $\cU$ of $Y \cup Z$ in $X$, which is open and dense in $X$.

The group $\Aut(X)$ is a semi--direct product of $G$ with its unipotent radical. It acts on $X$ with two orbits. While until then the roles of $Y$ and $Z$ were interchangeable, we now choose to use the notation $Z$ for the unique closed $\Aut(X)$--orbit and we call it the \emph{closed orbit of $X$}. The other $\Aut(X)$--orbit is $\cU_Y = X \setminus Z = \cU \cup Y$. We also define $\cU_Z = X \setminus Y = \cU \cup Z$.

The $G$--orbits $Y$ and $Z$ are projective and isomorphic to $G/P_Y$ and $G/P_Z$, respectively, where we choose $P_Y$ and $P_Z$ parabolic subgroups containing the same Borel subgroup $B$. The open $G$--orbit $\cU$ is a $\G_m$-bundle over the incidence variety $G/(P_Y \cap P_Z)$. We refer to \cite{pa:h} for proofs.

\subsection{Blow--ups and projections}\label{ss:blowups}

We refer to \cite{pa:h} for the results in this subsection. Let $\pi_Z : \Xt_Z \to X$ be the blow-up of $Z$ in $X$. It is obtained via base change from the blow--up of $\P(V_Z)$ in $\P(V_Y\oplus V_Z)$. In particular, there is a natural projection $p_Y : \Xt_Z \to Y$, obtained as the restriction of the projection of the former blow--up to $\P(V_Y)$. The projection $p_Y : \Xt_Z \to Y$ restricts to a projection $p_Y : \cU_Y \to Y$, which realises $\cU_Y$ as a vector bundle over $Y$. This bundle is $G$--equivariant and thus it is obtained from a $P_Y$--representation. We write $N_Y$ for the corresponding locally free sheaf; it is the normal bundle to $Y$ in $X$. The projection $p_Y : \Xt_Z \to Y$ realises $\Xt_Z$ as the projective bundle $\P_Y(N_Y \oplus \cO_Y)$. Write $E_Z$ for the exceptional divisor.

Exchanging the roles of $Y$ and $Z$ we denote by $\pi_Y : \Xt_Y \to X$ the blow--up of $Y$ in $X$. In particular, there is a natural projection $p_Z : \Xt_Y \to Z$, which restricts to a projection $p_Z : \cU_Z \to Z$. This projection realises $\cU_Z$ as a $G$--equivariant vector bundle over $Z$, which is thus obtained from a $P_Z$--representation. We write $N_Z$ for the corresponding locally free sheaf; it is the normal bundle to $Y$ in $X$. The projection $p_Z : \Xt_Y \to Z$ realises $\Xt_Y$ as the projective bundle $\P_Z(N_Z \oplus \cO_Z)$. Write $E_Y$ for the exceptional divisor.

In both cases, the exceptional divisors $E_Y$ and $E_Z$ are isomorphic to the incidence variety $E = G/(P_Y \cap P_Z)$. The maps $p_Y$ and $\pi_Z$ from $E_Z$ to $Y$ and $Z$ and the maps $\pi_Y$ and $p_Z$ from $E_Y$ to $Y$ and $Z$ are the projections $p$ and $q$ from $E$ to $G/P_Y$ and $G/P_Z$.

Summarising, we get the commutative diagrams
\[
	\xymatrix{E_Z \ar@{^(->}[r]^-{j_Z}  \ar@/^1.5pc/[rr]^-p \ar[d]_-q & \Xt_Z \ar[r]^-{p_Y} \ar[d]^-{\pi_Z} & Y \\
  	Z \ar@{^(->}[r]^-{i_Z} & X & \\} 
	\xymatrix{E_Y \ar@{^(->}[r]^-{j_Y}  \ar@/^1.5pc/[rr]^-q \ar[d]_-p & \Xt_Y \ar[r]^-{p_Z} \ar[d]^-{\pi_Y} & Z \\
	Y \ar@{^(->}[r]^-{i_Y} & X & \\}
\]

Let $\pi_{YZ} : \Xt_{YZ} \to X$ be the blow--up of $Y \cup Z$ in $X$. There is a morphism $p_{YZ} : \Xt_{YZ} \to E$ which is a $\P^1$--bundle, and both connected components of the exceptional divisor are mapped isomorphically onto $E$. We will denote this last map by $\xi = p_{YZ} : \Xt_{YZ} \to E$.

\subsection{Numerical invariants}

Using the above description, it is easy to compute some numerical invariants for $X$, $Y$ and $Z$. For example, their dimensions are given in Table~\ref{t:dimX}.

\begin{table}
	\[
		\begin{array}{c||c|c|c}
  			\text{Case} & \dim X & \codim_X Y & \codim_X Z \\
  			\hline
  			X^1(n) & n(n+3)/2 & 2 & n \\
  			X^2 & 9 & 4 & 3 \\
  			X^3(n,m) & m(2n + 1 - m) - m(m-1)/2 & m & 2(n + 1 - m) \\
  			X^4 & 23 & 3 & 3 \\
  			X^5 & 7 & 2 & 2 \\
  		\end{array}
  	\]
  	\caption{Dimension of $X$.}\label{t:dimX}
\end{table}

If $W$ is a Fano variety of Picard rank one and $H$ is the ample generator of $\Pic(W)$, we denote by $c_1(W)$ the unique positive integer such that $-K_W = c_1(W) H$, where $K_W$ is the canonical divisor of $W$. We describe $c_1(X)$, $c_1(Y)$ and $c_1(Z)$ in Table~\ref{t:num-inv}.

\begin{table}
	\[
		\begin{array}{c||c|c|c||c|c}
  			\text{Case} & c_1(X) & c_1(Y) & c_1(Z) & \codim_X Y & \codim_X Z \\
  			\hline
 		 	X^1(n) & n + 2 & n + 1 & 2n & 2 & n \\
  			X^2 & 7 & 5 & 6 & 4 & 3 \\
 		 	X^3(n,m) & 2n+2-m & 2n+1-m & 2n+2-m & m & 2(n + 1 - m) \\
 		 	X^4  & 6 & 5 & 7 & 3 & 3 \\
 		 	X^5 & 4 & 3 & 5 & 2 & 2 \\
		\end{array}
	\]
	\caption{Numerical invariants.}\label{t:num-inv}
\end{table}

\begin{fact}\label{fact:c1-codim}
	We have $c_1(X) = \codim_X(Y) + \codim_X(Z)$.
\end{fact}

\subsection{Cohomology classes}\label{ss:cohomology}

The description in Subsection~\ref{ss:blowups}  implies that $X$ has a cellular decomposition. It is therefore easy to describe geometric bases of the cohomology groups of $X$. 

Since $Y$ and $Z$ are projective and homogeneous under $G$, we can consider their Schubert varieties $(Y_u)_{u \in W^{P_Y}}$ and $(Z_v)_{v \in W^{P_Z}}$ for a given Borel subgroup $B \subset G$. These varieties define cohomology classes as follows:
\[
	\sb_u = [Y_u] \in \Hr^{2 \ell(u)}(Y,\Z) \textrm{ and } \taub_v = [Z_v] \in \Hr^{2 \ell(v)}(Z,\Z).
\]
We also introduce $Y'_u = \pi_Z(p_Y^{-1}(Y_u))$ and $Z'_v = \pi_Y(p_Z^{-1}(Z_v))$. 
Denoting by $i_Y : Y \to X$ and $i_Z : Z \to X$ the closed embeddings, we may define cohomology classes on $X$ as follows:
\begin{align*}
	\s'_u = [Y'_u] \in \Hr^{2 \ell(u)}(X,\Z) \text{ and } \s_u = i_*[Y_u] \in \Hr^{2 \ell(u) + 2 \codim_X(Y)}(X,\Z) \\
	\tau'_v = [Z'_v] \in \Hr^{2 \ell(v)}(X,\Z) \text{ and } \tau_v = j_*[Z_v] \in \Hr^{2 \ell(v) + 2 \codim_X(Z)}(X,\Z).
\end{align*}
Using the cellular decompositions by Schubert cells in $Y$ and $Z$ and the cellular decomposition induced on $\cU_Y$ and $\cU_Z$ we obtain the following result.

\begin{fact}\label{f:bases}
	We have the following bases of $\Hr^*(X,\Z)$.
	\begin{enumerate}[1.]
		\item The classes $((\s'_u)_{u \in W^{P_Y}},(\tau_v)_{v \in W^{P_Z}})$ form a $\Z$--basis of $\Hr^*(X,\Z)$.
		\item  The classes $((\tau'_v)_{v \in W^{P_Z}},(\s_u)_{u \in W^{P_Y}})$ form a $\Z$--basis of $\Hr^*(X,\Z)$.
	\end{enumerate}
\end{fact}

The proof is similar to that of 
\cite[Proposition 11.3.2]{k:flag}, and it relies on the following observation: the closure of a cell (in $X$) is a union of cells. Indeed, this is already true for the cells in $Y$ and $Z$. For the cells in $\cU_Y$ and $\cU_Z$, the cells are $B$--stable (actually they are the inverse image of $B$--orbits in the vector bundle $p_Y:\cU_Y \to Y$ or $p_Z:\cU_Z \to Z$) so their closures in $X$ are unions of cells in $\cU_Y$ or $\cU_Z$ or unions of $B$--orbits in the other orbit: in $Z$ and in $Y$, respectively. 

\begin{notation}
	Recall that $p$ and $q$ are the projections from the incidence variety $E$ to $Y$ and $Z$ respectively. 
	\begin{enumerate}[1.]
		\item Let $u \in W^{P_Y}$, then $q(p^{-1}(Y_u))$ is a Schubert variety in $Z$ of the form $Z_v$. We define $\uh \in W^{P_Z}$ such that $q(p^{-1}(Y_u)) = Z_\uh$. For $v \in W^{P_Z}$ we define $\vh$ in a similar way: $p(q^{-1}(Z_v)) = Y_\vh$.
		\item Let $u \in W^{P_Y}$, then $p^{-1}(Y_u)$ is the Schubert variety $E_\ut$ in $E$ associated to an element $\ut \in W^{P_Y \cap P_Z}$. For $v \in W^{P_Z}$ we define $\vt$ in a similar way.
		\item Let $h_Y$ and $h_Z$ the hyperplane classes in $Y$ and $Z$. We define coefficients $c_Y(u,u')$ and $c_Z(v,v')$ as follows:
			\[
				h_Y \cup \sb_u = \sum_{u' \in W^{P_Y}} c_Y(u,u') \sb_{u'} \textrm{ and } h_Z \cup \taub_v = \sum_{v' \in W^{P_Z}} c_Z(v,v') \taub_{v'}.
			\]
			These coefficients are well known since they are given by the Chevalley formula for homogeneous spaces, see \cite{chevalley:cel}.

\item If $M$ is a complete nonsingular irreducible variety of dimension $n$, then the Poincar\'e duality map $H^k(M, \Z)\to H_{2n-k}(M,\Z)$, taking $\sigma$ to $\sigma \cap [M]$, is an isomorphism, where $[M]\in H_{2n}(M,\Z)\simeq \Z$ is the fundamental class of $M$. 
If, moreover, $H_{\ast}(M,\Z)$ is torsion free, then,  by Poincar\'e duality, there is a perfect pairing $\scal{\,,\,}: H^k(M,\Z)\otimes H^{2n-k}(M,\Z)\to \Z$,    
given by $\scal{\alpha,\beta}:=(\alpha \cup \beta)\cap [M]$. This is called the Poincar\'e duality pairing. If $\{x_u\}$ is a homogeneous basis for $H^*(M,\Z)$, then the Poincar\'e dual basis is the basis $\{x_u^\vee\}$ dual for this pairing, so $\scal{x_u,x_v^\vee}=\delta_{u,v}$. 

\item For $u\in W^{P_Y}$, we define $u^{\vee}\in W^{P_Y}$ such that $\overline{\sigma}_{u^\vee}$ is Poincar\'e dual to $\overline{\sigma}_u$ in $H^*(Y,\Z)$. That is, $\overline{\sigma}_{u^\vee}={\overline{\sigma}_u}^\vee$, with the notation of item 4. It is well-known that 
 $u^\vee=w_0 u w_Y$, where 
$w_0$ is the longest element of $W$ and $w_Y$ is the longest element in $W_{P_Y}$. For $v\in W^{P_Z}$ we define $v^\vee$ in a similar fashion.
	\end{enumerate}
\end{notation}

\begin{prop}\label{prop:poincare}
	In $\Hr^*(X,\Z)$, the singular cohomology of $X$, the two bases defined above: $((\s'_u)_{u \in W^{P_Y}} , (\tau_v)_{v \in W^{P_Z}})$ and $((\tau'_v)_{v \in W^{P_Z}}$, $(\s_u)_{u \in W^{P_Y}})$ are dual for the Poincar\'e duality pairing. 
       We have
	\[
		\sigma_u^\vee = \sigma'_{u^\vee} \textrm{ and } \tau_v^\vee = \tau'_{v^\vee}.
	\]
\end{prop}

\begin{proof}
	Using the projection formula we get ${\pi_Z}_*(p_Y^*\sb_{u_1} \cup \pi_Z^*\sigma_{u_2}) = \sigma'_{u_1} \cup \sigma_{u_2} $. But, again by the projection formula, we have ${p_Y}_*(p_Y^*\sb_{u_1} \cup \pi_Z^*\sigma_{u_2}) = \sb_{u_1} \cup \sb_{u_2} = \delta_{u_1,u_2^\vee}$. 
 We thus have $\sigma'_{u_1} \cup \sigma_{u_2} = \delta_{u_1,u_2^\vee}$. The same method works for the $\tau$ classes. We are left with proving $\s_u \cup \tau_v = 0$ and  $\s'_u \cup \tau'_v = 0$. To compute these intersections, we use general translates of $Y_u$, $Y'_u$, $Z_v$ and $Z'_v$ under the action of $G$. First note that since $Y$ and $Z$ do not meet, we get $\s_u \cup \tau_v = 0$. Next we want to compute the intersection of general $G$-translates of $Y'_u$ and $Z'_v$. We look for intersection points in $Y$, $Z$ and $\cU$. Since $Y'_u \cap Z = Z_\uh$ is a Schubert variety of dimension at most $\dim X - \ell(u) - 1$, its intersection with a general translate of $Z'_v \cap Z = Z_v$ is of dimension at most $\dim X - \ell(u) - 1 - \ell(v) = -1$ and therefore empty. The same argument proves that no point of the intersection is contained in $Y$. Finally, if there was an intersection point in $\cU$, then using the restriction of $p_{YZ}$ to $\cU$ would give a point in the intersection of general translates of $p_{YZ}(\pi_{YZ}^{-1}(Y'_u)) = p^{-1}(Y_u) = E_\ut$ and $p_{YZ}(\pi_{YZ}^{-1}(Z'_v)) = p^{-1}(Z_v)) = E_\vt$ in $E$. The sum of their codimensions is $\ell(u) + \ell(v) = \dim X = \dim E +1$ and therefore these general translates do not meet.
\end{proof}

We state the following cup product formulas for later reference.

\begin{prop}
	Let $u_1,u_2 \in W^{P_Y}$ and $v_1,v_2 \in W^{P_Z}$ with $\ell(u_1) + \ell(u_2) = \dim X$ and $\ell(v_1) + \ell(v_2) = \dim X$.
	\begin{enumerate}[1.]
		\item If $\codim_X Z = 2$, then $\s'_{u_1} \cup \s'_{u_2} = \delta_{\uh_1,\uh_2^\vee}$.
		\item If $\codim_X Y = 2$, then $\tau'_{v_1} \cup \tau'_{v_2} = \delta_{\vh_1,\vh_2^\vee}$.
	\end{enumerate}
\end{prop}

\begin{remark}
	The condition $\codim_X Z = 2$ is satisfied for $X^3(n,n)$ and $X^5$. The condition $\codim_X Y = 2$ is satisfied in for $X^1(n)$, $X^3(n,2)$ and $X^5$.
\end{remark}

\begin{proof}
	We prove item 1, the proof of item 2 being similar. We again look at general $G$-translates of $Y'_{u_1}$ and $Y'_{u_2}$. If they were meeting in $\cU_Y$, then via $p_Y$ there would be an intersection point for general translates of $Y_{u_1}$ and $Y_{u_2}$ in $Y$. This is not possible for dimension reasons. These varieties can therefore only meet in $Z$ and we have $Y'_{u_i} \cap Z = Z_{\uh_i}$ which is of codimension at least $\ell(u_i) + 1 - \codim_X Z$ in $Z$. We look for the intersection of general translates of $Z_{\uh_1}$ and $Z_{\uh_1}$. The sum of their codimension in $Z$ is at least $\ell(u_1) + 1 - \codim_X Z + \ell(u_2) + 1 - \codim_X Z = \dim Z + 2 - \codim_X Z = \dim Z$. This proves the result.
\end{proof}

\begin{remark}\label{rem:ss'tt'}
	The same proof gives the following more precise result. Let $u_1,u_2 \in W^{P_Y}$ with $\ell(u_1) + \ell(u_2) = \dim X$. Assume that  $\codim_Z Z_{\uh_1} + \codim_Z Z_{\uh_2} = \dim Z$, then $\s'_{u_1} \cup \s'_{u_2} = \delta_{\uh_1,\uh_2^\vee}$.
\end{remark}

\subsection{Hasse diagram}

In this subsection we describe the multiplication by the hyperplane class $h$ in $\Hr^*(X,\Z)$.

\begin{prop}\label{prop-chev}
	We have 
	\[
		h \cup \s'_u = \tau_\uh + \sum_{u' \in W^{P_Y}} c_Y(u,u') \s'_{u'} \textrm{ and } h \cup \tau_v = \sum_{v' \in W^{P_Z}} c_Z(v,v') \tau_{v'}.
	\]
\end{prop}

\begin{proof}
	The first formula follows from the equalities $\pi_Z^*h = [E] + p_Y^*h_Y$ and $\pi_Z(E \cap p_Y^{-1}(Y_u)) = Z_\uh$. The second formula can be directly computed in $Z$. 
\end{proof}

\begin{remark}
	Similar formulas hold for the products $h \cup \s_u$ and $h \cup \tau'_v$.
\end{remark}


\subsection{Special geometry}\label{ss:special}

	In two cases of Pasquier's classification, there is a special geometric feature that we want to use: the variety $X$ is the zero locus of a section of a vector bundle. Let us fix first some notation. Denote by $\OG(5,10)$ a connected component of the Grassmannian of maximal isotropic subspaces in $\C^{10}$ endowed with a non-degenerate quadratic form. Denote by $\Gr(m,n)$ the Grassmannian of $m$-dimensional subspaces in $\C^n$ and let $\Uu_m$ be the tautological subbundle. Finally, denote by $\IG(m,2n+1)$ the Grassmannian of isotropic subspaces of dimension $m$ in $\C^{2n+1}$ endowed with an antisymmetric form of rank $2n$.

\begin{prop}\label{prop:section}
	If $X=X^2$ or $X=X^3(n,m)$, then $X$ is obtained as the zero locus of a vector bundle on a bigger homogeneous space $\Xf$. More precisely
	\begin{enumerate}[1.]
		\item The variety $X^2$ is a general hyperplane section of $\Xf = \OG(5,10)$.
		\item The variety $X^3(n,m)$ is the zero locus of a section of $\Lambda^2\Uu_m^\vee$ in $\Xf = \Gr(m,2n+1)$.
	\end{enumerate}
\end{prop}

\begin{proof}
	\begin{enumerate}[1.]
		\item We use Mukai's classification of Fano varieties of large index (see \cite{ip:fano}). Since $X$ is a Fano variety of dimension $9$ and index $7$, it has to be a hyperplane section of $\OG(5,10)$.
		\item The variety $X$ is isomorphic to $\IG(m,2n+1)$, which is by definition the zero locus of a section of $\Lambda^2\Uu_m^\vee$ in $\Gr(m,2n+1)$.
	\end{enumerate}
\end{proof}


\section{Moduli space of curves}

Let $X$ be as above a projective horospherical variety with $\Pic(X) = \Z$. Let $d\in H_2(X,\Z)$ and denote by $\Mod{d}{k}{X}$ the moduli space of stable maps of genus $0$, degree $d$ and $k$ marked points. Here the degree of a map $f:C \to X$ with $C$ a nodal curve is by definition the push-forward $f_*[C]$. Identifying $H_2(X,\Z)$ with $\Z$ this moduli space has the expected dimension:
\[
	\dim X + \scal{c_1(T_X),d} + k-3 = \dim X + dc_1(X) + k-3.
\]
Any irreducible component of $\Mod{d}{k}{X}$ has dimension at least equal to the expected dimension. We refer to \cite{k:c,fp:moduli,lee:f} for further properties of $\Mod{d}{k}{X}$.

\subsection{Curves with irreducible source}

We first consider the open subset $\Modc{d}{k}{X}$ of $\Mod{d}{k}{X}$ consisting of stable maps with irreducible source. In this subsection, we describe the irreducible components of $\Modc{d}{k}{X}$. We start with an useful lemma on maps to projective bundles. This result was proved for projective bundles of rank one in \cite[Proposition 4]{pe:courbes} and the proof of the general result is very similar. We include this proof for the reader's convenience.

\begin{lemma}
	Let $p : M \to N$ be a projective bundle, \emph{i.e.}, $M = \P_N(E)$, where $E$ a globally generated vector bundle over a projective variety. Let $\a \in H_2(M,\Z)$ be such that $a = - \a \cap c_1(\cO_p(-1)) \geq 0$ (here $\cO_p(-1)$ is the tautological rank one subbundle associated to the projective bundle $p$).

	Then the map $\Modc{\a}{k}{M} \to \Modc{p_*\a}{k}{N}$ induced by composition with $p$ realises $\Modc{\a}{k}{M}$ as a non empty open subset of a projective bundle of rank $p_*\a \cap c_1(E) + \rk(E)(a + 1) - 1$ over $\Modc{p_*\a}{k}{N}$.
\end{lemma}

\begin{proof}
	The fiber of the map $\Modc{\a}{k}{M} \to \Modc{p_*\a}{k}{N}$ over $f : C \to N$ is given by sections of $\P_C(f^*E)$. Since $C$ is rational and irreducible, we may assume that $C = \P^1$ and that these sections are injective maps of vector bundles $f^*\cO_p(-1) = \cO_C(-a) \to f^*E$ modulo scalars, where $a = - \a \cap c_1(\cO_p(-1)) \geq 0$. The fiber is thus the subset of $\P\Hom(\cO_C(-a),f^*E)$ of injective maps of vector bundles. Since $E$ is globally generated, so is $f^*E$.
This implies that our sections form a non-empty open subset of this projective space.

	This construction can be realised in families, which gives us the assertion. Indeed, for $C = \P^1$ let $f : C \times \Modc{p_*\a}{k}{N} \to N$ be the universal map. We will need the projection $\pr : C \times \Modc{p_*\a}{k}{N} \to \Modc{p_*\a}{k}{N}$. Consider the sheaf $\pr_*(f^*E \otimes \cO_C(a))$. Since $E$ is globally generated and $a \geq 0$, we have $R^1\pr_*(f^*E \otimes \cO_C(a)) = 0$, hence $\pr_*(f^*E \otimes \cO_C(a))$ is a vector bundle of rank $\deg f^*E + \rk(E)(a + 1)= p_*\a \cap c_1(E) + \rk(E)(a + 1)$ and $\Modc{p_*\a}{k}{M}$ is an open subset of $\P_{\Modc{p_*\a}{k}{N}}(\pr_*(f^*E \otimes \cO_C(a)))$.
\end{proof}

We use this lemma to describe the irreducible components of $\Modc{d}{k}{X}$.

\begin{prop}
	Let $X$ be as above. We have
	\begin{enumerate}[(a)]
		\item $\Modc{d}{k}{X} = \overline{\Modc{d}{k}{\cU_Y}} \cup \Modc{d}{k}{Z}$.
		\item $\overline{\Modc{d}{k}{\cU_Y}}$ is irreducible of dimension $\dim X + d c_1(X) + k-3$.
		\item $\Modc{d}{k}{Z}$ is irreducible of dimension $\dim Z + d c_1(Z) + k-3$.
	\end{enumerate}
\end{prop}

\begin{proof}
  Recall from \cite{thomsen,kim-pandharipande,pe:courbes} that for a homogeneous space, the moduli space of stable maps is irreducible of the expected dimension, therefore $\Modc{d}{k}{Z}$ is irreducible of dimension $\dim Z + d c_1(Z) + k-3$ and $\Modc{d_Y}{k}{Y}$ is irreducible of dimension $\dim Y + d_Y c_1(Y) + k-3$ for any non negative integer $d_Y$. This in turn imply that $\Modc{d}{k}{\cU_Y}$ and $\overline{\Modc{d}{k}{\cU_Y}}$ are irreducible of expected dimension. Indeed, in \cite[Proposition 3]{pe:courbes} it is proved that given a globally generated vector bundle $\varphi : W \to W'$ such that $\Modc{d}{k}{W'}$ is irreducible of expected dimension, then so is $\Modc{d}{k}{W}$. Since $p_Y : \cU_Y \to Y$ is a globally generated vector bundle, the result follows.

	Let $f : C \to X$ be a stable map, where $C$ is irreducible and nodal. Up to normalising, we may assume that $C$ is smooth, hence $C = \P^1$. If $f$ factors through $Z$ then we are in $\Modc{d}{k}{Z}$. So we may assume that $f(C)$ meets $\cU_Y$ and need to prove that $f$ is in the closure of $\Modc{d}{k}{\cU_Y}$. For this it is enough to prove that stable maps with irreducible source having a non-trivial intersection with $\cU_Y$ and $Z$ form a family of dimension smaller than the expected dimension $\dim X + d c_1(X) + k-3$.

	For this we consider stable maps from $C$ to $\Xt_Z$, the blow--up of $X$ in $Z$. The map $f$ lifts to a map $\ft : C \to \Xt_Z$. We compute its degree with respect to the basis $(\pi_Z^*\cO_X(1),p_Y^*\cO_Y(1))$ of $\Pic(\Xt_Z)$. We have $d = \deg f^* \cO_X(1) = \deg \ft^*\pi_Z^*\cO_X(1)$. Define $d_Y = \deg \ft^*p_Y^*\cO_Y(1)$. Since $p_Y^*\cO_Y(1)$ is semiample, we have $d_Y \geq 0$. Since $\cO_{\Xt_Z}(E_Z) = \pi_Z^*\cO_X(1)\otimes p_Y^*\cO_Y(-1)$ and $\deg \ft^*\cO_{\Xt_Z}(E_Z)$ is the intersection multiplicity of the curve $f$ with $Z$, we have $d - d_Y \geq 0$.

	Now we consider the open subset $\Modc{(d,d_Y)}{k}{\Xt_Z}$ of stable maps with irreducible source, target $\Xt_Z$ and degree $(d,d_Y)$ with respect to the basis $(\pi_Z^*\cO_X(1),p_Y^*\cO_Y(1))$ of $\Pic(\Xt_Z)$. We only need to assume that both $d$ and $d_Y$ are non--negative (this is always true for non--empty moduli spaces) and $d - d_Y \geq 0$. We have a natural map $\Modc{(d,d_Y)}{k}{\Xt_Z} \to \Modc{d_Y}{k}{Y}$ given by the composition with $p_Y : \Xt_Z \to Y$. Note that $p_Y$ is the projective bundle $\P_Y(\cO_Y \oplus N_{Y/X})$ and that $N_{Y/X}$ is globally generated. We want to apply the previous lemma so we need to compute $a = - \a \cap c_1(\cO_{p_Y}(-1))$ with $\a = (d,d_Y)$. But a map $\ft : C \to \Xt_Z$ induces an injective map of vector bundles $\cO_C(-a) = \ft^*\cO_{p_Y}(-1) \to \cO_C \oplus \ft^*N_{Y/X}$, and the intersection of $\ft(C)$ with $E$ is the vanishing locus of the first factor. In particular $a = d - d_Y \geq 0$. Note that $c_1(N_{Y/X})=(c_1(X) - c_1(Y))H_Y$, where $H_Y$ is the ample generator of $\Pic(Y)$. Therefore, applying the previous lemma we get that $\Modc{(d,d_Y)}{k}{\Xt_Z}$ is a non-empty open subset of a projective bundle of rank $d_Y(c_1(X) - c_1(Y)) + \codim_X Y(d - d_Y+1)$ over $\Modc{d_Y}{k}{Y}$. Therefore $\Modc{(d,d_Y)}{k}{\Xt_Z}$ is irreducible of dimension 
	\[
		\dim \Modc{(d,d_Y)}{k}{\Xt_Z} = \textrm{exp-dim } \Mod{d}{k}{X} - (d - d_Y)(c_1(X) - \codim_X Y).
	\]
In particular, since $c_1(X) - \codim_X Y > 0$, this dimension is smaller that the expected dimension $\textrm{exp-dim }\Mod{d}{k}{X}$ of $\Mod{d}{k}{X}$, except for $d_Y = d$, thus for curves not meeting $Z$. This proves the result.
\end{proof}

\begin{cor}
	Let $X$ be as above. 
	
	If $\codim_X Z > d(c_1(Z) - c_1(X))$, then $\Modc{d}{k}{X}$ is irreducible of dimension $\dim X + d c_1(X) + k-3$.

	Otherwise, $\Modc{d}{k}{X}$ has two irreducible components $\overline{\Modc{d}{k}{\cU_Y}}$ and $\Modc{d}{k}{Z}$ of respective dimension $\dim X + d c_1(X) + k-3$ and $\dim Z + d c_1(Z) + k-3$.
\end{cor}

\begin{proof}
	This follows from the fact that any irreducible component of $\Modc{d}{k}{X}$ is of dimension at least $\dim X + d c_1(X) + k-3$. In the second case, $\overline{\Modc{d}{k}{\cU_Y}}$ will never be in the closure of $\Modc{d}{k}{Z}$.
\end{proof}

\begin{cor} 
	\begin{enumerate}[(a)]
		\item In all the cases of Pasquier's list, $\Modc{1}{k}{X}$ is irreducible of dimension $\dim X + c_1(X) + k-3$.
		\item  If $X=X^2$ or $X=X^3(n,m)$, then $\Modc{d}{k}{X}$ is irreducible of dimension $\dim X + d c_1(X) + k-3$ for all $d$.
		
	\end{enumerate}
\end{cor}

\subsection{Irreducible components}

The former results enable us to give a description of the irreducible components of $\Mod{d}{k}{X}$. Depending on the situation it is possible to have many of them. Here we only consider the simplest cases where there are few irreducible components.

\begin{thm}\label{thm:irr}
	\begin{enumerate}[(a)]
	
		\item If $X=X^2$ or $X=X^3(n,m)$, then $\Mod{d}{k}{X}$ is irreducible of dimension $\dim X + d c_1(X) + k-3$ for any $d$.
		
		\item $\Mod{1}{k}{X}$ is irreducible of dimension $\dim X + c_1(X) + k-3$ for all $X$ of Pasquier's list. With the exception of $X^5$, we have the vanishing $\Hr^1(C,f^*T_X) = 0$ for any stable map $f : C \to X$.
		
		\item $\Mod{2}{k}{X}$ is irreducible of dimension $\dim X + 2 c_1(X) + k-3$ if $X$ is one of the varieties $X^1(3),X^2,X^3(n,m),X^4$.
		
		\item $\Mod{2}{k}{X}$ has two irreducible components of dimension $\dim X + 2 c_1(X) + k-3$ if $X=X^1(4)$ or $X=X^5$.
		
		\item $\Mod{3}{k}{X}$ has two irreducible components of dimension $\dim X + 3 c_1(X) + k-3$ if $X=X^1(3)$ or $X=X^4$.
	
	\end{enumerate}
\end{thm}

\begin{proof}
	Note that to prove the above results (except for the vanishing in \emph{(b)}) we only need to prove the equality $\Mod{d}{k}{X} = \overline{\Modc{d}{k}{X}}$.

	For $r \geq 2$, $(d_1,\cdots,d_r)$ with $d = \sum_id_i$ a partition of $d$ and $\coprod_{i = 1}^r A_i = [1,k]$ a partition, there is a natural map $\Mod{d_1}{|A_1|+1}{X} \times_X \cdots \times_X \Mod{d_r}{|A_r|+1}{X} \to \Mod{d}{k}{X}$. Furthermore the complement $\Mod{d}{k}{X} \setminus \Modc{d}{k}{X}$ is covered by the union of the images of these maps. 

	To prove the above results we only need to check that the boundary varieties 			\[
		\Mod{d_1}{|A_1|+1}{X} \times_X \cdots \times_X \Mod{d_r}{|A_r|+1}{X}
	\]
	have dimension smaller than the expected dimension $\dim X + d c_1(X) + k-3$. In all above cases, we have $\dim \Modc{d}{k}{X} = \dim X + d c_1(X) + k-3$. So an easy induction argument implies the result.

	To prove the vanishing result in \emph{(b)}, we only need to prove that for a line $L$ contained in $Z$, we have $\Hr^1(L,N_{Z/X}\vert_L) = 0$. Since $N_{Z/X}$ is the vector bundle $G \times^{P_Z}V(\oY - \oZ)$ and since we can choose the line to be the one dimensional Schubert variety in $Z$, it is enough to check that for any weight $\lambda$ of  $V(\oY - \oZ)$ we have $\scal{\a_Z^\vee,\lambda} \geq -1$, where $\alpha_Z$ is the simple root dual to $\oZ$. This is an easy check and is true in all cases except $X=X^5$.
\end{proof}

\begin{remark}
	\begin{enumerate}[1.]
		\item The non--obstruction result \emph{(b)} above was first proved for the varieties $X^3(n,m)$ in \cite{pech:q}. This in particular implies that the corresponding stack is smooth and thus $\Mod{1}{k}{X}$ has only quotient (and thus rational) singularities, see below for more details.
		\item It is worth noting the following result from \cite[Proposition 2]{lee:f}. Let $X$ be a smooth projective complex algebraic variety such that $\Mod{d}{k}{X}$ has the expected dimension. Then the virtual class of $\Mod{d}{k}{X}$ is equal to its fundamental class. In particular, this is true in all cases of the previous theorem.  
	\end{enumerate}
\end{remark}

\subsection{Singularities}

In this section we prove regularity results for $\Mod{d}{k}{X}$.

\begin{thm}   
	\begin{enumerate}[(i)]
	\item With the exception of $X^5$, for all $X$ of Pasquier's list the moduli space $\Mod{1}{k}{X}$ only has quotient singularities. \\
                 Let $X$ be either $X^2$ or $X^3(n,m)$.

		\item If $f:C\to X$ is a stable map of genus $0$ such that no component of $C$ is mapped entirely into $Z$, then $\Mod{d}{k}{X}$ has quotient singularity at $[f]$.
		\item The singular locus $\Sing (\Mod{d}{k}{X})$ has codimension at least 2. 
		\item The moduli space $\Mod{d}{k}{X}$ is normal and Cohen-Macaulay.  
	\end{enumerate}
\end{thm}

\begin{proof}
	\begin{enumerate}[(i)]
		\item By Theorem \ref{thm:irr} part \emph{(b)}, the corresponding stack is smooth and the result follows.
		\item It follows from \cite[Lemma 2.1]{bf:rc} that if $X$ has a $G$--action with a dense orbit and $f:\P^1\to X$ meets the dense orbit, then $\Hr^1(\P^1,f^*T_X)=0$. The results follows as in (i).
		\item It follows from (ii) that the singular locus of $\Mod{d}{k}{X}$ lies inside the locus of curves which have a component contained in $Z$. But for any degree $d$ the moduli space $\Mod{d}{k}{Z}$ is irreducible of dimension $\dim(Z) + d c_1(Z) + k-3$, which is of codimension at least $2$ in $\Mod{d}{k}{X}$. This proves the result.

		\item In view of (iii), to prove normality it suffices to show that $\Mod{d}{k}{X}$ is Cohen-Macaulay. For this, recall that in the cases under consideration, $X$ is the zero locus of a vector bundle on a projective homogeneous space $\Xf \subset \P^N$ (see Proposition \ref{prop:section}). That is, there exists a vector bundle $E$ on $\Xf$ and a global section $s\in \Hr^0(\Xf,E)$ such that $X=V(s)$. Furthermore, note that $E$ is globally generated. Let $\Vf$ be the total space of the vector bundle $E$. Proceeding as in the proof of \cite[Proposition 3]{pe:courbes} one checks that $\Mod{d}{k}{\Vf}$ is a vector bundle over $\Mod{d}{k}{\Xf}$. Indeed, we have a morphism $\phi : \Mod{d}{k}{\Vf} \to \Mod{d}{k}{\Xf}$. Consider the following diagram 
	\[
		\xymatrix{\Cc \ar[d]^\pi \ar[r]^{\ev}& \Xf \\
			\Mod{d}{k}{\Xf}
		}
	\]
where $\Cc=\Mod{d}{k+1}{\Xf}$ is the universal curve, $\pi$ is the map that forgets the $(k+1)$-th marked point, and ${\ev}$ is the evaluation map. Since $E$ is globally generated it follows that $R^1\pi_*{\ev}^*E=0$ and that $\pi_*{\ev}^*E$ is locally free. The fiber of $\phi$ over $[f:C\to \Xf] \in \Mod{d}{k}{\Xf}$ is $\Hr^0(C,f^*E)$ and thus $\Mod{d}{k}{\Xf}$ is the bundle associated to $\pi_*{\ev}^*E$ over $\Mod{d}{k}{\Xf}$.

	Furthermore, let $f^*s$ denote the pullback section of $f^*E$. It is clear that $f:C \to \Xf$ factors through $X=V(s)$ if and only if $s\circ f=0$ or, equivalently, $f^*s\in \Hr^0(C,f^*E)$ and $f^*s=0$. The section $s$ induces a section $\cO_{\Mod{d}{k}{\Xf}}\to \pi_*\cO_C\to \pi_*(f^*E)$ defining $\Mod{d}{k}{X}$ as its vanishing locus: $\Mod{d}{k}{X} = \Mod{d}{k}{\Xf} \cap V(\pi_*f^*s)$. Since $\Mod{d}{k}{\Xf}$ has quotient singularities, it has rational singularities and hence it is Cohen-Macaulay. In view of \cite[Lemma page 108]{fp:cm}, $\Mod{d}{k}{X}$ is also Cohen-Macaulay. This fact together with (iii) finally implies that $\Mod{d}{k}{X}$ is normal. \qedhere
	\end{enumerate}
\end{proof} 

\begin{question}
	\label{conj}
	Does $\Mod{d}{k}{X}$ have rational singularities if $X=X^2$ or $X=X^3(n,m)$?
\end{question}


\section{Quantum cohomology}

\subsection{Reminders}

For $X$ a smooth projective complex variety with Picard rank one, we have $k$ evaluation maps $\ev_i : \Mod{d}{k}{X} \to X$ on $\Mod{d}{k}{X}$, sending a stable map to its value at the $i$-th marked point for $i \in [1,k]$. The Gromov-Witten invariants are defined as follows:
\[
	\scal{\a_1,\cdots,\a_k}_{X,d,k} := \int_{[\Mod{d}{k}{X}]^{\rm vir}} \prod_{i = 1}^k \ev_i^*\a_i
\]
for $\sum_i \deg\a_i = \dim X + d c_1(X) + k -3$, and it is equal to zero otherwise. Here $\a_i \in \Hr^*(X,\R)$ are cohomology classes and $[\Mod{d}{k}{X}]^{\rm vir}$ is a homology class, the virtual fundamental class (see~\cite{BF} for a definition). An important feature of that class is the equality $[\Mod{d}{k}{X}]^{\rm vir} = [\Mod{d}{k}{X}]$ if $\Mod{d}{k}{X}$ has expected dimension $\dim X + d c_1(X) + k -3$ (see~\cite[Proposition 2]{lee:f}). The small quantum cohomology is defined as follows. Let $(e_i)_i$ be a basis of $\Hr^*(X,\R)$ and let $(e^\vee_i)_i$ be the dual basis for the Poincar\'e pairing. Then 
\[
	e_i * e_j = \sum_{d \geq 0} \sum_\ell \scal{e_i,e_j,e^\vee_\ell}_{X,d,3} e_\ell q^d.
\]
It is a non--trivial result that this product defines an associative ring (see for example \cite{fp:moduli}). When the moduli space has expected dimension, the Gromov--Witten invariants are enumerative. We make this precise in the next subsection.

\begin{remark}
	The varieties $Y$ and $Z$ are homogeneous for the group $G$. The quantum cohomology algebra of homogeneous spaces has been extensively studied and is well understood in many cases. Of special interest for our explicit computations in Section \ref{example} are minuscule, cominuscule, adjoint and coadjoint varieties. The quantum cohomology of these varieties were studied in \cite{CMP1,CMP2,CMP3,CMP4,adjoint}. Many of these varieties are isotropic Grassmannians, and their quantum cohomology is well understood thanks to the work of Buch, Kresch and Tamvakis \cite{buch:grass,KT1,KT2,BKT1,BKT2}. We refer to these papers for explicit computations.
\end{remark}

\subsection{Enumerability of Gromov-Witten invariants.}
               
In this section we use a generalization of the Kleiman--Bertini Transversality Theorem to quasi--homogeneous spaces due to Graber \cite[Lemma 2.5]{gra:e}.

\begin{lemma}\label{kleiman.lem}
	Let $X$ be a variety endowed with an action of a connected algebraic group $G$ with only finitely many orbits and $Z$ be an irreducible scheme with a morphism $f:Z\to X$. Let $Y$ be a subvariety of $X$ that intersects the orbit stratification properly (i.e., for any $G$-orbit $\cO$ of $X$, we have $\codim_\cO (Y\cap\cO)=\codim_X Y$). Then there exists a dense open subset $U$ of $G$ such that for every $g\in U$, $f^{-1}(gY)$ is either empty or has pure dimension $\dim Y + \dim Z-\dim X$. 
Moreover, if $X$, $Y$ and $Z$ are smooth and we denote by $Y_{\rm reg}$ the subset of $Y$ along which the intersection with the stratification is transverse, then the (possibly empty) open subset $f^{-1}(gY_{\rm reg})$ is smooth. 
\end{lemma}                

We apply the previous lemma to the following situation. 

\begin{thm}\label{enum.thm}
	Let $X$ be a $G$--variety with finitely many $G$--orbits and dense $G$-orbit $\cU$. Let $d \in H_2(X,\Z)$ be such that $\Mod{d}{k}{X}$ is irreducible and equals the closure of the locus of curves contained in $\cU$. Take $Y_1,\ldots, Y_k$ to be subvarieties of $X$ that intersect the orbit stratification properly and that represent cohomology classes $\Gamma_1,\ldots, \Gamma_k$ satisfying 
	\[
		\sum_i \codim_X Y_i=\dim \Mod{d}{k}{X}.
	\]
	Then there is a dense open subset $U \subset G^k$ such that for all $(g_1,\ldots, g_k)\in U$, the Gromov-Witten invariant $\scal{\Gamma_1,\ldots,\Gamma_k}_{X,d,k}$ is equal to the number of stable curves of degree $d$ in $X$ incident to the translates $g_1Y_1,\ldots, g_kY_k$.
\end{thm}

\begin{proof}
	First, we shall prove by repeatedly applying Lemma \ref{kleiman.lem}, 
that for generic $(g_1,\ldots, g_k)$ in $G^k$, the scheme-theoretic intersection 
	\[
		\ev^{-1}(g_1Y_1\times \ldots \times g_kY_k)=\bigcap_{i=1}^k \ev_i^{-1}(g_iY_i)
	\] 
consists of a finite number of reduced points supported on any 
preassigned nonempty open subset $M^*$ of $M = \Mod{d}{k}{X}$ and, in 
particular, in the smooth locus of $M$. 
Here $\ev=\ev_1\times \ldots\times \ev_k$.  
Bearing this in mind, consider the following diagram 
\[
	\xymatrix{ 
                                   & M\setminus M^*\ar[d]^{\ev}\\
		\underline{Y}=Y_1\times \ldots \times Y_k \ar@{^(->}[r] & X^{k}. 
				}
\]  
Since $M$ is irreducible, $M\setminus M^*$ has codimension at least one, so Lemma \ref{kleiman.lem} implies that there exists an open subset $V_1$ of $G^k$ such that the inverse image of any of the translates $g\cdot\underline{Y}$ with $g\in V_1$,  
is empty. Thus in general the intersection is completely supported in $M^*$. 

Next let $\cU$ be the dense $G$--orbit of $X$. By the above, we can 
restrict to the locus where ${\ev}$ is in $\cU^k$. Now let $M^{sm}$ be the smooth locus of $\Mod{d}{k}{X}$ and $\underline{Y}^s:=\Sing \, \underline{Y}$ and consider the diagram 
\[
	\xymatrix{ 
                                   & M^{sm}\cap {\ev}^{-1}(\cU^k)\ar[d]^{\ev}\\
			\underline{Y}^s\cap \cU^k\ar@{^(->}[r] & \cU^{k}. 
			}
\]  
Applying Kleiman's tranversality theorem to the homogeneous space $\cU$, we get an open subset $V_2\subset G^k$ such that ${\ev}^{-1}(g\cdot\underline{Y}^s\cap \cU^k)=\emptyset$, for any $g\in V_2$. Now let $\underline{Y}^*=(\underline{Y}\setminus \underline{Y}^s)\cap \cU^k$. By yet another application of Kleiman's tranversality theorem to the diagram 
\[
	\xymatrix{ & M^{sm}\cap {\ev}^{-1}(\cU^k)\ar[d]^{\ev}\\
				\underline{Y}^*\ar@{^(->}[r] & \cU^{k}. 
			}
\] 
we get an open subset $V_3\subset G^k$ such that for all $g\in V_3$, the inverse image in $M^{sm}\cap {\ev}^{-1}(\cU^k)$ of the translate $g\cdot \underline{Y}^*$ is either empty or smooth of the expected dimension (since all varieties involved 
are smooth). Hence it consists of a finite number of reduced points (possibly zero). 
We conclude that for any $g\in V_1\cap V_2\cap V_3$, the inverse image of 
$g\cdot \underline{Y}$ is of the expected dimension, reduced, and supported in the given open set. 
\end{proof}
 
Next we specify the previous result to cases (2) and (3) of Pasquier's list.  
 
\begin{cor}\label{cor:enum} 
	Let $X$ be a horospherical variety of Picard number one. If $X$ belongs to the classes (2) or (3) of Pasquier's list, then Theorem \ref{enum.thm} holds, i.e., the 
Gromov-Witten invariants of $X$ are enumerative.  
\end{cor}

\begin{remark}
	The previous corollary gives a positive answer to a conjecture from \cite[Conjecture 5.1]{pech:t} regarding the enumerativity of the degree $d$ Gromov-Witten invariants of the odd symplectic Grassmannians. The case when $d=1$ was already proven by the second author in \cite{pech:t}.  
\end{remark}

\subsection{Quantum hyperplane multiplication}\label{subsec-chev}

In this section we consider the multiplication with the hyperplane generator $h \in \Hr^2(X,\Z)$. Our aim is in particular to deal with the varieties 
$X^1(3),X^2,X^3(3,3)$ and $X^5$ of Pasquier's list, but we keep the computations as general as possible. Let us start with an easy dimension count.

\begin{lemma}\label{lem:dmax}
	Let $\sigma \in \Hr^*(X,\Z)$. The highest power of $q$ appearing in a product $h * \sigma$ in $\QH(X)$ satisfies $d \leq \frac{\dim X +1}{c_1(X)}$. Furthermore, for $X^1(3)$ and $X^5$, there is a non-trivial term of degree $2$ in the product $h * \sigma$ only if $\sigma = [\mathrm{pt}]$, where $[\mathrm{pt}]$ is the class of a point. Moreover, this term is a multiple of $q^2$.
\end{lemma}

\begin{proof}
	The degree of $h * \sigma$ is at most $\dim X +1$. If there is a $q^d$ appearing in the product, then its degree is at least $d c_1(X)$ leading to the inequality $d c_1(X) \leq \dim X +1$. The result follows.
\end{proof}

\begin{remark}
	We shall prove in Section \ref{odd} that for $X^3(n,m)$, the product $h * \sigma$ has terms of degree at most $1$ in $q$.
\end{remark}

\subsection{Degree one invariants}

In this subsection, we compute the degree one invariants involved in the quantum multiplication with $h$. We therefore compute Gromov-Witten invariants of the form $\scal{h,\sigma,\tau}_{X,1,3}$ with $\sigma,\tau \in \Hr^*(X,\Z)$ such that $\deg\sigma + \deg\tau +1 = \dim \Mod{1}{3}{X} = \dim X + c_1(X)$.

\begin{lemma}\label{lem:s's}
	Let $u_1,u_2 \in W^{P_Y}$. We have $\scal{h,\s'_{u_1},\s_{u_2}}_{X,1,3} = 0$.
\end{lemma}

\begin{proof}
	We may assume $\deg\s'_{u_1} + \deg\s_{u_2} +1 = \dim \Mod{1}{3}{X} = \dim X + c_1(X)$, which leads to $\ell(u_1) + \ell(u_2) = \dim Y + c_1(X) -1$. We have $\scal{h,\s'_{u_1},\s_{u_2}}_{X,1,3} = {\pi_M}_*(\ev_1^*h \cup \ev_2^*\s'_{u_1} \cup \ev_3^*\s_{u_2})$, 
where ${\pi_M}_*$ denotes the proper pushforward to a point, or integration against the virtual class of the moduli space of stable maps. 
 Choosing general translates $H$, $g_1.Y'_{u_1}$ and $g_2.Y_{u_2}$ with $g_1,g_2 \in G$, we have $\ev_1^*h \cup \ev_2^*\s'_{u_1} \cup \ev_3^*\s_{u_2} = \ev_1^{-1}(H) \cap \ev_2^{-1}(g_1.Y'_{u_1}) \cap \ev_3^{-1}(g_2.Y_{u_2})$. We prove that this set is empty. It is enough to prove that there is no line meeting $H$, $g_1.Y'_{u_1}$ and $g_2.Y_{u_2}$. 

	We discuss several cases. Let $L$ be such a line. If $L$ were in $Z$ then $L$ would not meet $g_2.Y_{u_2}$, a contradiction. If $L$ meets $Z$ (necessarily in one point since $Z = X \cap \P(V_Z)$), then $p_Y(\pi_Z^{-1}(L))$ is a point and lies on $p_Y(\pi_Z^{-1}(g_1.Y'_{u_1})) \cap p_Y(\pi_Z^{-1}(g_2.Y_{u_2}))$. But $p_Y(\pi_Z^{-1}(g_1.Y'_{u_1}))$ is a translate of $Y_{u_1}$ and $p_Y(\pi_Z^{-1}(g_2.Y_{u_2}))$ is a general translate of $Y_{u_2}$. An easy dimension count gives a contradiction, for   
$\ell(u_1)+\ell(u_2)>\dim Y$. 
Finally if $L$ does not meet $Z$, then $p_Y(\pi_Z^{-1}(L))$ is a line meeting $p_Y(\pi_Z^{-1}(g_1.Y'_{u_1}))$ and $p_Y(\pi_Z^{-1}(g_2.Y_{u_2}))$. This is possible if $\ell(u_1) + \ell(u_2) \leq \dim \Mod{1}{2}{Y} = \dim Y + c_1(Y) - 1$. This is a contradiction since $c_1(X) > c_1(Y)$.
\end{proof}

\begin{lemma}\label{lem:st}
	Let $u \in W^{P_Y}$ and $v \in W^{P_Z}$. We have $\scal{h,\s_{u},\tau_{v}}_{X,1,3} = \delta_{\ut^\vee,\vt}$.
\end{lemma}

\begin{proof}
	As in the previous lemma, we may assume $\ell(u) + \ell(v) = \dim Y + \dim Z  + c_1(X) - \dim X - 1$ and we are looking for lines $L$ meeting general translates $H$, $g_1.Y_{u}$ and $g_2.Z_{v}$. In particular, the line $L$ is not contained in $Z$ (otherwise $L \cap g_1.Y_v = \emptyset$) and not in $Y$ (otherwise $L \cap g_2.Z_v = \emptyset$) but meets $Y$ in $g_1.Y_u$ and $Z$ in $g_2.Z_v$. The line $L$ is thus mapped to a point via $p_{YZ}$ \emph{i.e} $p_{YZ}(\pi^{-1}_{YZ}(L))$ is a point. Thus $p_{YZ}(\pi^{-1}_{YZ}(g_1.Y_{u}))$ a general translate of $E_\ut$ and $p_{YZ}(\pi^{-1}_{YZ}(g_2.Z_{v}))$ a general translate of $E_{\vt}$ have to meet. The sum of their codimensions is $\ell(u) + \ell(v) = \dim E + c_1(X) - \codim_X Y - \codim_X Z$ and by Fact \ref{fact:c1-codim} we get $\ell(u) + \ell(v) = \dim E$. In particular this intersection is empty unless $\vt = \ut^\vee$. If this intersection is non-empty and therefore is a point $e$, the line is given by $\pi_{YZ}(p_{YZ}^{-1}(e))$.
\end{proof}

\begin{lemma}\label{lem:s't'}
	Let $u \in W^{P_Y}$ and $v \in W^{P_Z}$. We have 
	\[
		\scal{h,\s'_{u},\tau'_{v}}_{X,1,3} = \left\{
			\begin{array}{ll}
				0 & \text{if $c_1(X) \geq c_1(Z)$} \\
				\scal{h,\taub_{\uh},\taub_{v}}_{Z,1,3} & \text{if $c_1(X) = c_1(Z) -1$} \\
			\end{array}
		\right.
	\]
\end{lemma}

\begin{remark}
	Note that $c_1(X) \geq c_1(Z)$ for $X^2$ and $X^3(n,m)$, while $c_1(X) = c_1(Z) -1$ for $X^1(3), X^4$ and $X^5$.
\end{remark}

\begin{proof}
	As above, we may assume $\ell(u) + \ell(v) = \dim X + c_1(X) - 1$ and we are looking for lines $L$ meeting general translates $H$, $g_1.Y'_{u}$ and $g_2.Z'_{v}$. If $L \subset Y$, then $L$ meets general translates of $H \cap Y$, $Y_u$ and $Y_\vh$. The sum of their codimension is $\ell(u) + \ell(v) + 1 - \codim_X Y + 1 = \dim Y + c_1(X) +1 > \dim \Mod{1}{3}{Y}$, so there is no such $L$. If $L$ is contained in $\cU$, then $p_{YZ}(\pi_{YZ}^{-1}(L))$ is a line meeting general translates of $E_\ut$ and $E_\vt$. The sum of their codimensions is $\ell(u) + \ell(v) = \dim X -1 + c_1(X) = \dim E + c_1(X)$. But the curve $p_{YZ}(\pi_{YZ}^{-1}(L))$ has intersection $c_1(X)$ with $-K_E$, so this sum is one more than the dimension of the moduli space, which implies that there is no such curve. Assume that $L$ meets $\cU$. If the line $L$ meets $Z$ then $p_Y(\pi_Z^{-1}(L))$ is a point meeting general translates of $Y_u$ and $Y_\vh$. An easy dimension count proves that this is not possible. The same arguments prove that there is no line $L$ meeting $\cU$ and $Y$.

        We thus have the inclusion $L \subset Z$ and since $Y'_u \cap Z = Z_\uh$, the number of lines $L$ as above is given by the Gromov-Witten invariant $\scal{h,\taub_{\uh},\taub_{v}}_{Z,1,3}$. Again by a dimension count the result follows. 
\end{proof}

\begin{lemma}\label{lem:tt}
	Let $v_1,v_2 \in W^{P_Z}$. We have 
	\[
		\scal{h,\tau_{v_1},\tau'_{v_2}}_{X,1,3} = \left\{
			\begin{array}{ll}
				0 & \text{if $c_1(X) > c_1(Z)$} \\
				\scal{h,\taub_{v_1},\taub_{v_2}}_{Z,1,3} & \text{if $c_1(X) = c_1(Z)$} \\
			\end{array}
		\right.
	\]
\end{lemma}

\begin{remark}
	Note that $c_1(X) > c_1(Z)$ for $X^2$, while $c_1(X) = c_1(Z)$ for $X^3(n,m)$.
\end{remark}

\begin{proof}
  As above, we may assume $\ell(u) + \ell(v) = \dim Z + c_1(X) - 1$ and we are looking for lines $L$ meeting general translates $H$, $g_1.Z_{v_1}$ and $g_2.Z'_{v_2}$. An easy dimension count using the composition $p_Z \circ \pi_Y$ proves that $L$ does not meet $Y$. Projecting to $Z$ we have that $p_Z(\pi_Y^{-1}(L))$ is a line $L'$ which meets general translates of $Z_{v_1}$ and $Z_{v_2}$. Such a line $L'$ does not exist for dimension reasons if $c_1(X) > c_1(Z)$, and there are $\scal{h,\taub_{v_1},\taub_{v_2}}_{Z,1,3}$ such lines if $c_1(X) = c_1(Z)$. Since $g_1$ and $g_2$ are general in $G$ we may assume that these lines $L'$ are general. We may thus assume that $c_1(X) = c_1(Z)$. We are thus in case (3) of Pasquier's classification and any lift of $L'$ in $\cU_Z$ meeting the zero section is a solution for $L$. But $\cU_Z$ is a vector bundle over $Z$ whose restriction to a general line is trivial. In particular the only lift meeting the zero section is the zero section itself, proving the result.
\end{proof}

\begin{lemma}\label{lem:st'}
	Let $u \in W^{P_Y}$ and $v \in W^{P_Z}$. For $c_1(X) - c_1(Y) + 1 - \codim_X Y = 0$, we have $\scal{h,\s_{u},\tau'_{v}}_{X,1,3} = \scal{h,\sb_{u},\sb_{\vh}}_{Y,1,3}$.
\end{lemma}

\begin{remark}
	Note that we have $c_1(X) - c_1(Y) + 1 - \codim_X Y \leq 0$ in all cases, with equality for $X=X^1(n), X^3(n,2)$ and $X^5$.
\end{remark}

\begin{proof}
	As above, we may assume $\ell(u) + \codim_X Y + \ell(v) = \dim X + c_1(X) - 1$ and we are looking for lines $L$ meeting general translates $H$, $g_1.Y_{u}$ and $g_2.Z'_{v}$. If $L$ is not contained in $Y$, then $p_Z(\pi_Y^{-1}(L))$ is a point meeting general translates of $Z_{\uh}$ and $Z_v$. An easy dimension count proves that this is not possible. If $L \subset Y$, then we are counting curves in $Y$ meeting a hyperplane and general translates of $Y_u$ and $Y_{\vh}$. The condition implies that the classes have the right codimension, which implies the result.
\end{proof}

\begin{lemma}\label{lem:t't'}
Let $v_1,v_2 \in W^{P_Z}$, we have
\[
	\! \scal{h,\tau'_{v_1},\tau'_{v_2}}_{X,1,3} \! = \! \left\{
		\begin{array}{ll}
 			\! 0 & \! \textrm{if $c_1(X) - c_1(Y) > \codim_X Y - 2$} \\
 			\! \scal{h,\sb_{\vh_1},\sb_{\vh_2}}_{Y,1,3} & \! \textrm{if $c_1(X) - c_1(Y) = \codim_X Y - 2$.} \\
		\end{array}
	\right.
\]
\end{lemma}

\begin{remark}
	Note that we have $c_1(X) - c_1(Y) + 2 - \codim_X Y > 0$ for $X=X^1(n),X^2,X^3(n,2)$ and $X^5$, while $c_1(X) - c_1(Y) + 2 - \codim_X Y = 0$ for $X=X^3(n,3)$ and $X=X^4$.
\end{remark}

\begin{proof}
	As above, we may assume $\ell(u) + \ell(v) + 1 = \dim X + c_1(X)$ and we are looking for lines $L$ meeting general translates $H$, $g_1.Z'_{v_1}$ and $g_2.Z'_{v_2}$. If $L$ is contained in $Z$, then we are computing $\scal{h_Z,\taub_{v_1},\taub_{v_2}}_{Z,1,3}$. But $\ell(v_1) + \ell(v_2) + 1 = \dim X + c_1(X) > \dim Z + c_1(Z)$ so there is no such $L$. If $L$ meets $Z$, then $p_Y(\pi_Z^{-1}(L))$ is a point contained in the intersection of general translates of $p_Y(\pi_Z^{-1}(Z'_{v_1})) = Y_{\vh_1}$ and $p_Y(\pi_Z^{-1}(Z'_{v_2})) = Y_{\vh_2}$. But the sum of their codimensions is $\ell(v_1) + 1 - \codim_X Y + \ell(v_2) + 1 - \codim_X Y = \dim Y + c_1(X) + 1 - \codim_X Y = \dim Y + 1 + \codim_X Z > \dim Y$ so there is no such $L$. If $L$ meets $Y$ but is not contained in $Y$, then $p_Z(\pi_Y^{-1}(L))$ is a point meeting general translates of $Z_{v_1}$ and $Z_{v_2}$. An easy dimension count proves that this is not possible. If $L$ is contained in $\cU$, we consider $L' = p_{YZ}\pi_{YZ}^{-1}(L)$ meeting general translates of $p_{YZ}\pi_{YZ}^{-1}(Z'_{v_1}) = E_{\vt_1}$ and $p_{YZ}\pi_{YZ}^{-1}(Z'_{v_2}) = E_{\vt_2}$. The sum of their codimension in $E$ is $\ell(v_1) + \ell(v_2) = \dim X -1 + c_1(X) = \dim E + c_1(X)$. The curve $L'$ has intersection $c_1(X)$ with $-K_E$, thus there is no such line. We are left with $L$ contained in $Y$. So we are computing the number of lines corresponding to the invariant $\scal{h_Y,\sb_{\vh_1},\sb_{\vh_2}}_{Y,1,3}$. The sum of the corresponding codimensions is  $\ell(v_1) + 1 - \codim_X Y + \ell(v_2) + 1 - \codim_X Y + 1 = \dim Y + c_1(X) + 2 - \codim_X Y = \dim Y +  c_1(Y) + c_1(X) - c_1(Y) +2 - \codim_X Y$. The result follows.
\end{proof}

\subsection{Degree $2$ invariants}

	In this subsection we compute some degree $2$ Gromov-Witten invariants. We concentrate on the case $\dim X = 2 c_1(X) -1$, where there is a unique non-vanishing degree $2$ invariant, $\scal{h,[\mathrm{pt}],[\mathrm{pt}]}_{X,2,3}$.

\begin{lemma}\label{lem:q^2}
	If $\dim X = 2 c_1(X) -1$, then  $\scal{h,[\mathrm{pt}],[\mathrm{pt}]}_{X,2,3} = 2$.
\end{lemma}

\begin{remark}
	Note that we have $\dim X = 2c_1(X) - 1$ for $X=X^1(3)$ and $X=X^5$.
\end{remark}

\begin{proof}
	Recall that $\scal{h,[\mathrm{pt}],[\mathrm{pt}]}_{X,2,3} = 2 \scal{[\mathrm{pt}],[\mathrm{pt}]}_{X,2,2}$ so we need to prove that there is a unique conic through two general points in $X$. Consider two general points $p_1$ and $p_2$ in $\cU_Y$. Any stable map of degree $2$ passing through $p_1,p_2$ gives by projection $p_Y$ to $Y$ a stable map of degree at most $2$ passing through two general points in $Y$. With our assumptions we have $\dim Y = 2 c_1(Y) - 1$ and there exists exactly one (irreducible) stable map $f : \P^1 \to Y$ of degree $2$ through two general points in $Y$ (in our case, $Y$ is an adjoint variety, so this follows from \cite[Proposition 10]{adjoint}). We need to prove that $f$ lifts to a unique stable map in $\cU_Y$ passing through $p_1$ and $p_2$. This is equivalent to proving that the vector bundle $N_Y$ satisfies $f^*N_Y = \cO_{\P^1}(1)^{\codim_X Y}$. The unique irreducible curve passing through the $B$--stable and the $B^+$--stable points in $Y$ is the curve obtained using the cocharacter $\theta^\vee$ with $\theta$ the highest root of $G$. Using the fact that $N_Y = p_*\mathcal{L}_{\omega_Y - \omega_Z} = G \times^{P_Y} V_{P_Y}(\omega_Y - \omega_Z)$ it is easy to check that the condition $f^*N_Y = \cO_{\P^1}(1)^{\codim_X Y}$ is satisfied.
\end{proof}


\section{Examples}\label{example}

In this section we compute the quantum hyperplane multiplication for the following examples: $X^1(3), X^2, X^3(3,3)$ and $X^5$. Our goal is to extend the Hasse diagram to a quantum Hasse diagram and deduce semisimplicity results for the quantum cohomology.

\subsection{Semisimplicity} 

Using our quantum Chevalley formulas, we obtain

\begin{thm}
	If $X$ is one of the varieties $X^1(3),X^2,X^3(3,3)$ or $X^5$, then $\QH(X)$ is semisimple.
\end{thm}

\begin{proof}
  The proof in all cases goes as follows. We work in the specialisation $q = 1$ of the quantum cohomology. If the algebra obtained for $q = 1$ is semisimple, so is $\QH(X)$ (this is actually an equivalence).

  Using the previous section we give in Propositions \ref{prop-chevalley-case1}, \ref{prop-chevalley-case2}, \ref{prop-chevalley-case3} and \ref{prop-chevalley-case5} Chevalley formulas for the quantum multiplication by the hyperplane class $h$ in the above cases.

  Using these Chevalley formulas, we can compute the quantum multiplication by $h^{c_1(X)}$ which is an endomorphism of the space
  $$\QH^0(X) := \oplus_{k \geq 0} \Hr^{2kc_1(X)}(X,\Z).$$
  Writing down the matrix of this endomorphism, which is a matrix of size $qh^0(X) = \dim \QH^0(X)$ we can check in all cases (see below) that it has $qh^0(X)$ different non-vanishing eigenvalues. Since $1 \in \QH^0(X)$, if $\mu(X)$ is the minimal polynomial $h^{c_1(X)}$ on $\QH^0(X)$, then $\mu(X^{c_1(X)})$ is a vanishing polynomial for $h$ on $\QH(X)$ with different (non-vanishing) eigenvalues. Since in all cases, an easy check gives the equality $\dim \QH(X) = c_1(X) \cdot \dim \QH^0(X)$, we see that $\mu(X^{c_1(X)})$ is the minimal and the characteristic polynomial of $h$ on $\QH(X)$. This implies that the endomorphism of $\QH(X)$ defined by multiplication with $h$ is semisimple with simple (non-vanishing) eigenvalues, therefore it generates the algebra which has to be semisimple.

  We now explicitly compute the matrix for $X^2$ and check that it has $qh^0(X)$ different non-vanishing eigenvalues. In the other cases, we only give the matrices obtained from Chevalley formulas. Recall that for $X=X^2$, we have $\dim X = 9$ and $c_1(X) = 7$. Set $\QH^{2i}(X) := \oplus_{k \geq 0} \Hr^{2kc_1(X)+2i}(X,\Z)$. We have the following basis:
  $$\begin{array}{lll}
    \QH^0(X) = \scal{1,\tau_{v_4}}, & \QH^2(X) = \scal{h,\tau_{v_5}}, & \QH^4(X) = \scal{\sigma'_{u_2},\tau_{v_6}}, \\
    \QH^6(X) = \scal{\sigma'_{u_3},\tau_{v_0}},& \QH^8(X) = \scal{\sigma'_{u_4},\tau_{v_1}}, & \QH^{10}(X) = \scal{\sigma'_{u_5},\tau_{v_2}}, \\
    \QH^{12}(X) = \scal{\tau_{v_3},\tau_{v'_3}}. & &\\
  \end{array}$$
  In particular we have $qh^0(X) = 2$ and we see that $\dim \QH(X) = c_1(X) \cdot qh^0(X) = 7 \cdot 2 = 14$. Let $A_i$ be the matrix, in the above basis, of the endomorphism $\QH^{2i}(X) \to \QH^{2i + 2}(X)$ induced by the quantum multiplication by $h$. By Proposition \ref{prop-chevalley-case2}, we have
  $$A_0 = \left(\begin{array}{cc}
    1 & 1 \\
    0 & 1 \\
  \end{array}
  \right),
  A_1 = \left(\begin{array}{cc}
    1 & 1 \\
    0 & 1 \\
  \end{array}
  \right),
  A_2 = \left(\begin{array}{cc}
    2 & 1 \\
    1 & 0 \\
  \end{array}
  \right),
  A_3 = \left(\begin{array}{cc}
    1 & 0 \\
    1 & 1 \\
  \end{array}
  \right),$$
$$  A_4 = \left(\begin{array}{cc}
    1 & 0 \\
    1 & 1 \\
  \end{array}
  \right),
  A_5 = \left(\begin{array}{cc}
    1 & 1 \\
    0 & 1 \\
  \end{array}
  \right),
  A_6 = \left(\begin{array}{cc}
    0 & 1 \\
    1 & 1 \\
  \end{array}
  \right).$$
  The matrix of $h^{c_1(X)}$ on $\QH^0(X)$ is the product $A_6A_5A_4A_3A_2A_1A_0$. We obtain:
  $$\left(\begin{array}{cc}
    5 & 12 \\
    12 & 29 \\
  \end{array}
  \right)$$
  which has two distinct non vanishing eigenvalues and minimal polynomial $\mu(X) = x^2 - 34 x +1$.

If $X$ is one of the varieties $X^1(3),X^2,X^3(3,3),X^5$, the matrices of $h^{c_1(X)}$, in an appropriate basis, are the following ones:
  $$\left(\begin{array}{cccc}
    18 & 20 & 15 & 38 \\
    20 & 35 & 23 & 48 \\
    15 & 33 & 20 & 38 \\
     0 & 10 &  6 &  4 \\
  \end{array}\right), 
  \left(\begin{array}{cccc}
16 & 20 & 63 & 34 \\
20 & 22 & 69 & 38 \\
 5 &  9 & 26 & 18 \\
 7 &  6 & 19 &  9 \\
  \end{array}\right),
  \left(\begin{array}{ccc}
18 & 14 & 38 \\
10 & 10 & 22 \\
 0 &  4  & 4 \\
  \end{array}\right).$$ 
  Again one can check that they have disctinct non vanishing eigenvalues.
\end{proof}

\begin{remark}\label{rem:ss}
	\begin{enumerate}
		\item The semisimplicity for $X=X^2$ can also be obtained as follows. It follows from Proposition \ref{prop:section} that $X$ is a general linear section of $\OG(5,10)$ but it was proved in \cite{pe:ss}, that the quantum cohomology of such hyperplane sections is semisimple.
		\item In the next section, we study more closely the quantum cohomology for $X=X^3(n,m)$ and prove that there is no degree $2$ term in the quantum multiplication by $h$. We also give a presentation of $\QH(X)$ in some cases. For similar varieties with non semisimple small quantum cohomology but semisimple big quantum cohomology, we refer to \cite{pe:ss}, \cite{GMS} and \cite{MSP}.
		\item The semisimplicity for $X=X^5$ can also be obtained as follows. In \cite{BP}, the authors proved that $X$ has a deformation isomorphic to the isotropic Grassmannian $\OG(2,7)$ and in \cite{adjoint} it was proved that $\QH(\OG(2,7))$ is semisimple. Since the quantum cohomology is a deformation invariant, this proves the result.
	\end{enumerate}
\end{remark}

\subsection{Case $X=X^1(3)$}

We have $\dim X = 9$ and $c_1(X) = 5$. $Z$ is a $6$-dimensional quadric. There is only one cohomology class for each degree in $Z$ except in degree $3$ so we denote by $(v_i)_{i \in [0,6]}$ with $\ell(v_i) = i$ and $v'_3$ with $\ell(v'_3) = 3$ the elements of $W^{P_Z}$. We choose to set $v_3 = s_1s_2s_3$ and $v'_3 = s_3s_2s_3$. We have $Y = \OG(2,7)$ and there is one cohomology class in degree $0$, $1$, $6$ and $7$ and two cohomology classes otherwise. We denote by $(u_i)_{i \in [0,7]}$ with $\ell(u_i) = i$ and $u'_2,u'_3,u'_4,u'_5$ with $\ell(u'_i) = i$ the elements in $W^{P_Y}$. We choose to set $u'_2 = s_1s_2$, $u'_3 = s_3 s_1s_2$, $u'_4 = s_2 s_3 s_1s_2$, $u'_5 = s_3 s_2 s_3 s_1s_2$. Note that $\s'_{u_0} = \tau'_{v_0} = 1$, $\s'_{u_1} = \tau'_{v_1} = h$, $\s_{u_7} = \tau_{v_6} = [\mathrm{pt}]$ and $\s_{u_6} = \tau_{v_5} = [\mathrm{line}]$.

\begin{prop}
  \label{prop-chevalley-case1}
	We have the formulas
	\begin{enumerate}[(1)]
		\item $h * 1 = h$
		\item $h * h = 2 \s'_{u_2} + \s'_{u'_2}$
		\item $h * \s'_{u_2} = \s'_{u_3} + \s'_{u'_3}$ and $h * \s'_{u'_2} = 2 \s'_{u'_3} + \tau_{v_0}$
		\item $h * \s'_{u_3} = 2 \s'_{u_4} + \s'_{u'_4}$, $h * \s'_{u_3} = \s'_{u_4} + 2 \s'_{u'_4} + \tau_{v_1}$ and $h * \tau_{v_0} = \tau_{v_1}$
		\item $h * \s'_{u_4} = \s'_{u_5}$, $h * \s'_{u'_4} = \s'_{u_5} + 2 \s'_{u'_5} + \tau_{v_2}$ and $h * \tau_{v_1} = \tau_{v_2} + q$
		\item $h * \s'_{u_5} = 2 \s'_{u_6} + \tau_{v_3}$, $h * \s'_{u'_5} = \s'_{u_6} + \tau_{v'_3}$ and $h * \tau_{v_2} = \tau_{v_3} + \tau_{v'_3} + q h$
		\item $h * \s'_{u_6} = \s'_{u_7} + \tau_{v_4}$, $h * \tau_{v_3} = \tau_{v_4} + q \s'_{u'_2}$ and $h * \tau_{v'_3} = \tau_{v_4} + q \s'_{u_2}$
		\item $h * \s'_{u_7} = \tau_{v_5} + q \tau_{v_0}$ and $h * \tau_{v_4} = \tau_{v_5} + q \s'_{u'_3}$
		\item $h * \tau_{v_5} = \tau_{v_6} + q \s'_{u'_4} + q \tau_{v_1}$
		\item $h * \tau_{v_6} = q \s'_{u'_5} + q \tau_{v_2} + 2 q^2.$
	\end{enumerate}
\end{prop}

\begin{proof}
	We will only discuss the quantum corrections, since the classical part follows from the classical Hasse diagram. By Lemma \ref{lem:dmax}, there is no power of $q$ higher than $1$ except maybe in $h * \tau_{v_6}$. For the first four lines, there is no quantum correction since the degree is smaller than $5 = c_1(X)$. For the fifth line, we only have to compute $\scal{h,\s'_{u_4},\s_{u_7}}$, $\scal{h,\s'_{u'_4},\s_{u_7}}$  and $\scal{h,\tau_{v_1},\s_{u_7}}$. The result follows from Lemma \ref{lem:s's} and Lemma \ref{lem:st}. For line (6), we only have to compute $\scal{h,\s'_{u_5},\s_{u_6}}$, $\scal{h,\s'_{u'_5},\s_{u_6}}$ and $\scal{h,\tau_{v_2},\s_{u_6}}$. The result follows from Lemma \ref{lem:s's} and Lemma \ref{lem:st}. For line (7), we only have to compute $\scal{h,\s'_{u_6},\s_{u_5}}$, $\scal{h,\tau_{v_3},\s_{u_5}}$ and $\scal{h,\tau_{v'_3},\s_{u_5}}$. The result follows from Lemma \ref{lem:s's} and Lemma \ref{lem:st}. 

	Line (8) requires more work. We need to compute the Gromov-Witten invariants $\scal{h,\s'_{u_7},\s_{u_4}}$, $\scal{h,\tau_{v_4},\s_{u_4}}$, $\scal{h,\s'_{u_7},\tau'_{v_6}}$ and $\scal{h,\tau_{v_4},\tau'_{v_6}}$. The first three are obtained using Lemma \ref{lem:s's}, Lemma \ref{lem:st} and Lemma \ref{lem:s't'}. For the last one, we prove the equality $\tau_{v_4} = \s_{u_5} - \s'_{u_7}$. The result then follows from Lemmas \ref{lem:s't'} and \ref{lem:st'}. To prove the equality $\tau_{v_4} = \s_{u_5} - \s'_{u_7}$, we first write $\s'_{u_2} = a \s_{u_0} + b \tau'_{v_2}$. Multiplying with $\s'_{u_7}$ and using Remark \ref{rem:ss'tt'}, we have $1 = a$. Multiplying with $h^{\cup 7}$, we obtain $56 = 56a + b \tau'_{v_2} \cup h^{\cup 7}$. Finally we get $\s'_{u_2} = \s_{u_0}$. Now write  $\s_{u_5} = \lambda \s'_{u_7} + \mu \tau_{v_4}$. Multiplying with $\s'_{u_2} = \s_{u_0}$ we get $1 = \lambda$. Multiplying with $h \cup h$ we get $2 = \lambda + \mu$. It follows that $\lambda = 1$ and $\mu = 1$.

	For line (9), we need to compute $\scal{h,\tau_{v_5},\s_{u_3}}$, $\scal{h,\tau_{v_5},\s_{u'_3}}$ and $\scal{h,\tau_{v_5},\tau'_{v_5}}$. The first two are obtained using Lemma \ref{lem:st}. For the last one we remark that $\tau_{v_5} = [{\rm line}] = \s_{u_6}$ and use Lemma \ref{lem:st'}.

        For line (10), we need to compute $\scal{h,\tau_{v_6},\s_{u_2}}$, $\scal{h,\tau_{v_6},\s_{u'_2}}$ and $\scal{h,\tau_{v_6},\tau'_{v_4}}$. The first two are obtained using Lemma \ref{lem:st}. For the last one we remark that $\tau_{v_6} = [{\rm pt}] = \s_{u_7}$ and use Lemma \ref{lem:st'}. Finally the $q^2$ term follows from Lemma \ref{lem:q^2}.
\end{proof}

\subsection{Case $X=X^2$}

We have $\dim X = 9$ and $c_1(X) = 7$. Both $Y$ and $Z$ are quadrics with $\dim Y = 5$ and $\dim Z = 6$. There is only one cohomology class for each degree in $Y$ so we denote by $(u_i)_{i \in [0,5]}$ with $\ell(u_i) = i$ the elements of $W^{P_Y}$. In $Z$, for all degree except $3$ there is only one cohomology class for each degree so we denote by $(v_i)_{i \in [0,5]}$ with $\ell(v_i) = i$ and by $v'_3$ the elements in $W^{P_Z}$. We choose to set $v_3 = s_1s_2s_3$ and $v'_3 = s_3s_2s_3$. Note that $\s'_{u_0} = \tau'_{v_0} = 1$, $\s'_{u_1} = \tau'_{v_1} = h$, $\s_{u_5} = \tau_{v_6} = [{\rm pt}]$ and $\s_{u_4} = \tau_{v_5} = [{\rm line}]$.

\begin{prop}
  \label{prop-chevalley-case2}
	We have the formulas
	\begin{enumerate}[(1)]
		\item $h * 1 = h$
		\item $h * h = \s'_{u_2}$
		\item $h * \s'_{u_2} = 2 \s'_{u_3} + \tau_{v_0}$
		\item $h * \s'_{u_3} = \s'_{u_4} + \tau_{v_1}$ and $h * \tau_{v_0} = \tau_{v_1}$
		\item $h * \s'_{u_4} = \s'_{u_5} + \tau_{v_2}$ and $h * \tau_{v_1} = \tau_{v_2}$
		\item $h * \s'_{u_5} = \tau_{v_3}$ and $h * \tau_{v_2} = \tau_{v_3} + \tau_{v'_3}$
		\item $h * \tau_{v_3} = \tau_{v_4}$ and $h * \tau_{v'_3} = \tau_{v_4} + q$
		\item $h * \tau_{v_4} = \tau_{v_5} + q h$
		\item $h * \tau_{v_5} = \tau_{v_6} + q \s'_{u_2}$
		\item $h * \tau_{v_6} = q \s'_{u_3}$.
	\end{enumerate}
\end{prop}

\begin{proof}
	The first 6 lines follow from the classical Hasse diagram, and the quantum parts follow from Lemmas \ref{lem:s's}, \ref{lem:s't'}, \ref{lem:st} and \ref{lem:tt}.
\end{proof}

\subsection{Case $X=X^3(3,3)$}

We have $\dim X = 9$ and $c_1(X) = 5$. Moreover $Y$ is a $6$-dimensional quadric. There is only one cohomology class for each degree in $Y$ except in degree $3$, so we denote by $(u_i)_{i \in [0,6]}$ with $\ell(u_i) = i$ and $u'_3$ with $\ell(u'_3) = 3$ the elements in $W^{P_Y}$. We choose to set $u_3 = s_1s_2s_3$ and $u'_3 = s_3s_2s_3$. We have $Z = \IG(2,6)$ and there is one cohomology class in degree $0$, $1$, $6$ and $7$ and two cohomology classes otherwise. We denote by $(v_i)_{i \in [0,7]}$ with $\ell(v_i) = i$ and $v'_2,v'_3,v'_4,v'_5$ with $\ell(v'_i) = i$ the elements in $W^{P_Z}$. We choose to set $v_2 = s_1s_2$, $v_3 = s_3 s_1s_2$, $v_4 = s_2 s_3 s_1s_2$, and $v_5 = s_3 s_2 s_3 s_1s_2$. Note that $\s'_{u_0} = \tau'_{v_0} = 1$, $\s'_{u_1} = \tau'_{v_1} = h$, $\s_{u_6} = \tau_{v_7} = [{\rm pt}]$ and $\s_{u_5} = \tau_{v_6} = [{\rm line}]$.

\begin{prop}
  \label{prop-chevalley-case3}
	We have the formulas
	\begin{enumerate}[(1)]
		\item $h * 1 = h$
		\item $h * h = 2 \s'_{u_2} + \tau_{v_0}$
		\item $h * \s'_{u_2} = 2 \s'_{u_3} + \s'_{u'_3} + \tau_{v_1}$ and $h * \tau_{v_0} = \tau_{v_1}$
		\item $h * \s'_{u_3} = \s'_{u_4} + \tau_{v_2}$, $h * \s'_{u'_3} = 2 \s'_{u_4} + \tau_{v'_2}$ and $h * \tau_{v_1} = \tau_{v_2} + \tau_{v'_2}$
		\item $h * \s'_{u_4} = 2 \s'_{u_5} + \tau_{v_3}$, $h * \tau_{v_2} = \tau_{v_3} + q$ and $h * \tau_{v'_2} = \tau_{v_3} + \tau_{v'_3}$
		\item $h * \s'_{u_5} = \s'_{u_6} + \tau_{v_4}$, $h * \tau_{v_3} = 2 \tau_{v_3} + \tau_{v'_4} + q h$ and $h * \tau_{v'_3} = \tau_{v_4} + 2 \tau_{v'_4}$
		\item $h * \s'_{u_6} = \tau_{v_5}$, $h * \tau_{v_4} = \tau_{v_5} + \tau_{v'_5} + q \s'_{u_2}$ and $h * \tau_{v'_4} = \tau_{v'_5} + q \tau_{v_0}$
		\item $h * \tau_{v_5} = \tau_{v_6}  + q \s'_{u_3}$ and $h * \tau_{v'_5} = \tau_{v_6} + q \s'_{u'_3} + q \tau_{v_1}$
		\item $h * \tau_{v_6} = \tau_{v_7}  + q \s'_{u_4} + q \tau_{v'_2}$
		\item $h * \tau_{v_7} = q \s'_{u_5} + q \tau_{v'_3}$.
	\end{enumerate}
\end{prop}

\begin{proof}
	The first 6 lines follows from the classical Hasse diagram. The quantum parts follow from Lemmas \ref{lem:s's}, \ref{lem:s't'}, \ref{lem:st}, \ref{lem:tt} and the fact (see Corollary \ref{cor:deg1}) that there is no $q^2$ term in the quantum multiplication with $h$.
\end{proof}

\subsection{Case $X=X^5$}\label{subsec:4.5}

We have $\dim X = 7$ and $c_1(X) = 4$. There is only one cohomology class for each degree in $Y$ and $Z$ so we denote by $(u_i)_{i \in [0,5]}$ with $\ell(u_i) = i$ and $(v_i)_{i \in [0,5]}$ with $\ell(v_i) = i$ the elements in $W^{P_Y}$ and $W^{P_Z}$. Note that $\s'_{u_0} = \tau'_{v_0} = 1$, $\s'_{u_1} = \tau'_{v_1} = h$, $\s_{u_5} = \tau_{v_5} = [\mathrm{pt}]$ and $\s_{u_4} = \tau_{v_4} = [\mathrm{line}]$.

\begin{prop}
  \label{prop-chevalley-case5}
	We have the formulas
	\begin{enumerate}
		\item $h * 1 = h$
		\item $h * h = 3 \s'_{u_2} + \tau_{v_0}$
		\item $h * \s'_{u_2} = 2 \s'_{u_3} + \tau_{v_1}$ and $h * \tau_{v_0} = \tau_{v_1}$
		\item $h * \s'_{u_3} = 3 \s'_{u_4} + \tau_{v_2}$ and $h * \tau_{v_1} = \tau_{v_2} + q$
		\item $h * \s'_{u_4} = \s'_{u_5} + \tau_{v_3}$ and $h * \tau_{v_2} = 2 \tau_{v_3} + q h$
		\item $h * \s'_{u_5} = \tau_{v_4} + q \tau_{v_0}$ and $h * \tau_{v_3} = \tau_{v_4} + q \s'_{u_2}$
		\item $h * \tau_{v_4} = \tau_{v_5} + q \s'_{u_3} + q \tau_{v_1}$
		\item $h * \tau_{v_5} = q \s'_{u_4} + q \tau_{v_2} + 2 q^2.$
	\end{enumerate}
\end{prop}

\begin{proof}
	We will only discuss the quantum corrections, since the classical part follows from the classical Hasse diagram. By Lemma \ref{lem:dmax} there is no power of $q$ higher than $1$ except maybe in $h * \tau_{v_5}$. For the first three lines there is no quantum correction since the degree is smaller than $4 = c_1(X)$. For the fourth line, we only have to compute $\scal{h,\s'_{u_3},\s_{u_5}}$ and $\scal{h,\tau_{v_1},\s_{u_5}}$. The result follows from Lemma \ref{lem:s's} and Lemma \ref{lem:st}. For line (5), we only have to compute $\scal{h,\s'_{u_4},\s_{u_4}}$ and $\scal{h,\tau_{v_2},\s_{u_4}}$. The result follows from Lemma \ref{lem:s's} and Lemma \ref{lem:st}. 

	Line (6) requires more work. We need to compute the Gromov-Witten invariants $\scal{h,\s'_{u_5},\s_{u_3}}$, $\scal{h,\tau_{v_3},\s_{u_3}}$, $\scal{h,\s'_{u_5},\tau'_{v_5}}$ and $\scal{h,\tau_{v_3},\tau'_{v_5}}$. The first three are obtained using Lemma \ref{lem:s's}, Lemma \ref{lem:st} and Lemma \ref{lem:s't'}. For the last one, we prove the equality $\tau_{v_3} = \tau'_{v_5}$. The result then follows from Lemma \ref{lem:t't'}. Let us now prove $\tau_{v_3} = \tau'_{v_5}$. Write $\tau'_{v_5} = \lambda \s'_{u_5} + \mu \tau_{v_2}$. Multiplying with $\tau'_{v_2}$ we get $1 = \mu$. Multiplying with $h \cup h$ we get $1 = \lambda + \mu$. It follows that $\lambda = 0$ and $\mu = 1$.

	For line (7) we need to compute $\scal{h,\tau_{v_4},\s_{u_2}}$ and $\scal{h,\tau_{v_4},\tau'_{v_4}}$. The first invariant is obtained using Lemma \ref{lem:st}. For the last one we remark that $\tau_{v_4} = [\mathrm{line}] = \s_{u_4}$ and use Lemma \ref{lem:st'}. 
	
	For line (8) we need to compute $\scal{h,\tau_{v_5},\s_{u_1}}$ and $\scal{h,\tau_{v_5},\tau'_{v_3}}$. The first invariant is obtained using Lemma \ref{lem:st}. For the last one we remark that $\tau_{v_5} = [\mathrm{line}] = \s_{u_5}$ and use Lemma \ref{lem:st'}. Finally the $q^2$ term follows from Lemma \ref{lem:q^2}. 
\end{proof}


\section{The odd symplectic Grassmannian}\label{odd}

In this section we apply our enumerativity result, Corollary~\ref{cor:enum}, to deduce some results concerning the quantum multiplication in odd symplectic Grassmannians (case (3) of Pasquier's list). More precisely, in Section~\ref{s:kerspan} we use Buch's Ker/Span techniques to relate Gromov--Witten invariants of odd symplectic Grassmannians to intersection numbers in an auxiliary variety. This leads in Section~\ref{s:quantclass} to Theorem~\ref{t:quant-class}, which compares the quantum and the classical multiplication. Finally, in Section~\ref{s:quantpres} we obtain a presentation of the quantum cohomology of any odd symplectic Grassmannian, see Theorem~\ref{t:quantpres}.

\subsection{The odd symplectic Grassmannian as a horospherical variety}\label{s:IG:horo}

As in Section~\ref{ss:special} we let $n \geq 2$ be an integer and $\omega$ be an antisymmetric form of maximal rank on a complex vector space $V$ of dimension $2n+1$. If $2 \leq m \leq n$, we denote by $X=\IG(m,V)=\IG(m,2n+1)$ the Grassmannian of vector subspaces of $V$ which are isotropic for $\omega$:
	\[
		\IG(m,V) = \{ \Sigma \subset V \mid \dim \Sigma=m, \omega_{\mid \Sigma}=0 \}.
	\]
To view $X$ as a horospherical variety we let $K \subset V$ be the kernel of the form $\omega$ and $W$ be a complement to $K$ in $V$, so that $\omega\vert_W$ is symplectic. Setting $G := \Sp(W) \cong \Sp_{2n}$ we see that $X$ is $G$--horospherical and corresponds to the triple $(C_n,\omega_{m},\omega_{m-1})$, i.e., to the case (3) of Pasquier's classification. 

Indeed, $X$ is contained in $\P(\Lambda^mV)$ and as a $G$--representation we have $V = W \oplus K$, therefore $\Lambda^m V = \Lambda^m W \oplus \Lambda^{m-1} W \otimes K$, which are the irreducible representations of highest weights $\omega_m$ and $\omega_{m-1}$, respectively. One easily checks that $X$ is the closure of the $G$--orbit of the class $[v_Y + v_Z]$ of the sum of the lowest weight vectors $v_Y,v_Z$ of these two representations.

The $G$--orbits in $X$ are
\[
	Y = \{ \Sigma \in X \mid \Sigma \subset W \} = \IG(m,W)
	\textrm{ and } Z = \{ \Sigma \in X \mid K \subset \Sigma \} \cong \IG(m-1,W).
\]
Moreover, the blow--ups $\Xt_Y$, $\Xt_Z$ and $\Xt_{YZ}$ are explicitly given by:
\begin{align*}
	\Xt_Y &= \{ (\Sigma,\Sigma') \in X \times Z \mid \Sigma' \cap W \subset \Sigma \cap W \}, \\
	\Xt_Z &= \{ (\Sigma,\Sigma'') \in X \times Y \mid K + \Sigma \subset K + \Sigma'' \}, \\
	\Xt_{YZ} &= \{ (\Sigma,\Sigma',\Sigma'') \in X \times Z \times Y \mid (\Sigma,\Sigma') \in \Xt_Y \textrm{ and } (\Sigma,\Sigma'') \in \Xt_Z \}.
\end{align*}
The maps $\pi_Y,\pi_Z,p_Y,p_Z,\pi_{YZ},p_{YZ}$ are given by the obvious projections (for the last one, note that $E \subset Y \times Z$ is the incidence variety).

The automorphism group $\Aut(X)$ is the \emph{odd symplectic group}
\[
	\Sp_{2n+1}=\Sp(V) := \{ M \in \GL(V) \mid M^t \omega M= \omega \}.
\]
The odd symplectic group is not reductive, however its properties are closely related to those of the usual symplectic group, see~\cite{m:o}. 

\subsection{Schubert varieties in symplectic Grassmannians}\label{ss:schub-even}

As seen in Section~\ref{ss:cohomology} we have two Poincar\'e dual bases of the cohomology of $X:=\IG(m,V)$, obtained from the Schubert classes of the closed $\Sp_{2n}$-orbits $Y$ and $Z$, which are symplectic Grassmannians. There are many indexations for Schubert classes of symplectic Grassmannians; in this section we present those which will play a part later. 

\textbf{Indexation by Weyl group elements.} Schubert varieties of the symplectic Grassmannian $\IG(m,2n)$ are indexed by elements of $W^P$, i.e., by minimal length coset representatives of $W/W_P$. Such elements are signed permutations which may be written as
\[
	w = (w(1),w(2),\dots,w(n)) =  (y_1,y_2,\dots,y_{m-\ell},\bar z_\ell, \dots, \bar z_2,\bar z_1,v_1,v_2,\dots,v_{n-m}),
\]
where $0 \leq \ell \leq m$, $0<y_1 <y_2<\dots<y_{m-\ell}$, $0<z_1<z_2<\dots<z_\ell$, $0<v_1<v_2<\dots<v_{n-m}$, and $\bar z_i:=-z_i$.

\textbf{Indexation by pairs of partitions.} This indexation was introduced in~\cite{PR}. We recall it here since it will be needed in the proof of Theorem~\ref{t:quantpres}. A \emph{pair of partitions} for $\IG(m,2n)$ is a pair $\alpha=(\alpha^t,\alpha^b)$, where $\alpha^t$ is a strict partition contained in a $(n-m) \times n$ rectangle, $\alpha^b$ a strict partition contained in a $m\times n$ rectangle, and $\alpha^t_{n-m} \geq \ell(\alpha^b)+1$. Weyl group elements $w$ and  pairs of partitions $\alpha$ are related by the bijective map $w \mapsto \alpha$ given by
\begin{itemize}
\item $\alpha^t_r=n+1-v_r+\# \{1 \leq j \leq \ell \mid z_j < v_r\}$ for $1 \leq r \leq n-m$;
\item $\alpha^b_j = n+1-z_j$ for $1 \leq j \leq \ell$.
\end{itemize}

\textbf{Indexation by $k$--strict partitions.} This is the indexation from~\cite[Definition 1.1]{BKT2}. For $\IG(m,2n)$ we let $k=n-m$, and we say that a partition $\lambda=(\lambda_1 \geq \dots \geq \lambda_m \geq 0)$ is \emph{$k$--strict} if no part greater than $k$ is repeated, i.e. $\lambda_j > k \Rightarrow \lambda_{j} > \lambda_{j+1}$. Thereafter we will denote by $1^p$ the partition with $p$ parts all equal to $1$.

We have that $k$--strict partitions $\lambda$ are in bijective correspondence with pairs of partitions $\alpha=(\alpha^t,\alpha^b)$, via the map $\lambda \mapsto \alpha$ given by
\begin{itemize}
\item $\alpha^t_r = n+1-m - r + \#\{ 1 \leq i \leq m \mid \lambda_i \geq r \}$ for $1 \leq r \leq n-m$ ;
\item $\alpha^b_j = \lambda_j +m-n$ for $1 \leq j \leq \max \{1 \leq i \leq m \mid \lambda_i > n-m\}$.
\end{itemize}

\textbf{Indexation by index sets.} We say a sequence $\bp = (p_1 < \dots < p_m)$ with $p_1 \geq 1$ and $p_m \leq 2n$ is an \emph{index set} for $\IG(m,2n)$ if $p_i + p_j \neq 2n + 1$ for any $i,j$. There is a bijection between index sets $\bp$ and $(n-m)$--strict partitions $\lambda$ with $\lambda_1 \leq 2n-m$. This correspondence is given by
\begin{itemize}
\item $\lambda_j = 2n+1-m-p_j + \# \left\{ i < j \mid p_i+p_j > 2n+1 \right\}$;
\item $p_j = 2n+1-m-\lambda_j + \# \left\{ i < j \mid \lambda_i+\lambda_j \leq 2(n-m)+j-i \right\}$.
\end{itemize}
Index sets are most useful for understanding Schubert varieties geometrically. Indeed, let $F_\bullet:=(\left\{0\right\}=F_0 \subset F_1 \subset F_2 \subset \dots \subset F_{2n}=\C^{2n})$ be an \emph{$\Sp_{2n}$--flag}, i.e., a complete flag of $\C^{2n}$ such that $F_j^\perp=F_{2n-j}$ for any $0 \leq j \leq 2n$. Then we define the Schubert variety of $\IG(m,2n)$ associated with the index set $\bp:=\left(p_1,\dots,p_m \right)$ by
	\[
		X_{\bp}(F_\bullet) := \left\{ \Sigma \in \IG(m,2n) \mid \dim (\Sigma \cap F_{p_j}) \geq j, \forall\, 1 \leq j \leq m \right\}.
	\]

\subsection{Schubert varieties in odd symplectic Grassmannians}\label{ss:schub-odd}
We now come back to the odd symplectic Grassmannian $X:=\IG(m,V)=\IG(m,2n+1)$ from Section~\ref{s:IG:horo}. In addition to the two Poincar\'e dual bases introduced in Section~\ref{ss:cohomology}, in 2008 Mihai~\cite{m:o} constructed Schubert--type varieties for $X$ associated with Borel subgroups of the automorphism group $\Aut(X)=\Sp_{2n+1}$. Let us explain how these families of cohomology classes compare.

The Borel subgroups of $\Sp_{2n+1}$ are in one--to--one correspondence with complete flags of $V$ of the form 
\[
	F_\bullet:=(\left\{0\right\}=F_0 \subset K=F_1 \subset F_2 \subset \dots \subset F_{2n} \subset F_{2n+1}=V),
\]
where $K=\ker \omega$, and for any $1 \leq j \leq 2n+1$, $F_j^\perp = F_{2n+2-j}$. We call these flags \emph{$\Sp_{2n+1}$--flags}. Schubert varieties of $X$ have been defined by Mihai as follows.

\begin{defn}
	Let $F_\bullet$ be an $\Sp_{2n+1}$--flag and $1 \leq p_1< p_2 < \dots < p_m \leq 2n+1$ be integers such that $p_i+p_j \neq 2n+3$ for any $i,j$. We call such a sequence $\bp=( p_1,\dots,p_m)$ an \emph{odd index set for $\IG(m,2n+1)$}.
	We define the Schubert variety of $X$ associated with the odd index set $\bp$ by
	\[
		X_{\bp}(F_\bullet) := \left\{ \Sigma \in X \mid \dim (\Sigma \cap F_{p_j}) \geq j, \forall\, 1 \leq j \leq m \right\}.
	\]
\end{defn}
Mihai showed that these varieties form a CW complex; thus the corresponding cohomology classes $\upsilon_{\bp} := [X_{\bp}(F_\bullet)]$ form a $\Z$--basis for the cohomology ring $\Hr^*(X,\Z)$ of $X$. In~\cite[Section 2.2]{pech:t}, the second author introduced another indexation for these classes analogous to what happens in the symplectic case, see Section~\ref{ss:schub-even}. Namely, an \emph{$(n-m)$--strict partition for $\IG(m,2n+1)$} is a sequence of integers $\mu=(\mu_1 \geq \dots \geq \mu_m \geq -1)$ such that $\mu_1 \leq 2n+1-m$ and $\mu_1=2n+1-m$
if $\mu_m=-1$. The next proposition is straightforward.

\begin{prop}\label{p:kstrict}
	There is a bijection between odd index sets $\bp$ and $(n-m)$--strict partitions $\mu$ for $\IG(m,2n+1)$, given by
\begin{itemize}
	\item $\mu_j = 2n+2-m-p_j + \# \left\{ i < j \mid p_i+p_j > 2n+3 \right\}$;
	\item $p_j = 2n+2-m-\mu_j + \# \left\{ i < j \mid \mu_i+\mu_j \leq 2(n-m)+j-i \right\}$.
\end{itemize}
\end{prop}

It follows from this proposition that we may index the Schubert varieties of $X=\IG(m,2n+1)$ by suitable $(n-m)$--strict partitions. In particular if $\bp$ is an odd index set for $\IG(m,2n+1)$ and $\lambda$ is the corresponding $(n-m)$--strict partition as in Proposition~\ref{p:kstrict}, we set $\upsilon_\lambda := \upsilon_\bp$. In the rest of the section we will sometimes denote by $\bp(\lambda)$ the index set corresponding to the strict partition $\lambda$.

We now consider the two cohomology bases of $X$ constructed in Section~\ref{ss:cohomology}. We start by checking that the basis (1) from Fact~\ref{f:bases}, namely $\{ \sigma_u', \tau_v \}$, coincides with Mihai's Schubert--type basis $\{(\upsilon_\bp)\} = \{(\upsilon_\lambda)\}$.

\begin{prop}
	Let $\lambda$ be an $(n-m)$--strict partition of $Y=\IG(m,2n)$ and $\mu$ be an $(n+1-m)$--strict partition of $Z=\IG(m-1,2n)$. 
	Then in $\Hr^*(X,\Z)$,
	\begin{enumerate}
		\item $\sigma_\lambda' = \upsilon_\lambda$;
		\item $\tau_\mu=\upsilon_{2n+1-m,\mu_1-1,\dots,\mu_{m-1}-1}$.
	\end{enumerate}
\end{prop}

\begin{proof}
Recall the decomposition $V = K \oplus W$ and the description of the closed $\Sp_{2n}$--orbits $Y=\left\{ \Sigma \in X \mid \Sigma \subset W \right\}$ and $Z=\left\{ \Sigma \in X \mid K \subset \Sigma \right\}$ from Section~\ref{s:IG:horo}. Consider an $\Sp_{2n}$--flag:
\[
	G_\bullet := \left( \left\{0\right\} \subset G_1 \subset G_2 \subset ... \subset G_{2n-1} \subset W \right),
\]
where for any $1 \leq j \leq n$ we have that $G_j^\perp \cap W=G_{2n-j}$. Here the orthogonal is taken relatively to the antisymmetric form of maximal rank $\omega$ on $V=\C^{2n+1}$. The Schubert varieties of $Y$ and $Z$ from Section~\ref{ss:cohomology} can be defined with respect to such a flag. To compare them to Mihai's varieties, we extend $G_\bullet$ as follows:
\[
	K \oplus G_\bullet := \left( \left\{0\right\} \subset K \subset K \oplus G_1 \subset \dots \subset K \oplus G_{2n-1} \subset K \oplus W=V \right).
\]
For $1 \leq j \leq 2n$, we have $(K \oplus G_j)^\perp= G_j^\perp = K \oplus (G_j^\perp \cap W)= K \oplus G_{2n-j}$, hence $K \oplus G_\bullet$ is an $\Sp_{2n+1}$--flag. Now consider the subvariety of $Y$ associated with $\lambda$,
	\[
		Y_{\bp(\lambda)}(G_\bullet) = \left\{ \Lambda \in Y \mid \dim(\Lambda \cap G_{p_j(\lambda)}) \geq j, \forall 1 \leq j \leq m \right\},
	\]
and similarly for $Z$,
	\[
		Z_{\bp(\mu)}(G_\bullet) = \left\{ \Lambda \in Z \mid \dim(\Lambda \cap G_{p_j(\mu)}) \geq j, \forall 1 \leq j \leq m-1 \right\}.
	\]
Clearly, 
	\[
		Y_{\bp(\lambda)}'(G_\bullet) := \pi_Z(p_Y^{-1}(Y_{\bp(\lambda)}(G_\bullet)))= X_{\bp(\lambda)+\mathbf{1}}(K \oplus G_\bullet),
	\]
	where the maps $\pi_Z$ and $p_Y$ are defined in Section~\ref{ss:blowups}, and
	\[
		\bp(\lambda)+\mathbf{1}:=(p_1(\lambda)+1,\dots,p_m(\lambda)+1).
	\]
	Since $[X_{\bp(\lambda)+\mathbf{1}}(K \oplus G_\bullet)]=\upsilon_\lambda$, the first statement follows. Similarly, 
	\[
		\tau_\mu=[i_Z(Z_{\bp(\mu)}(G_\bullet))] = [X_{1,\bp(\mu)+\mathbf{1}}(K \oplus G_\bullet)]=\upsilon_{2n+1-m,\mu_1-1,\dots,\mu_{m-1}-1},
	\]
	which concludes the proof.
\end{proof}

It is also interesting to see what the basis (2) from Fact~\ref{f:bases}, namely $\{\sigma_u,\tau_v')\}$ corresponds to. We will see that these basis elements no longer correspond to Schubert--type classes of $X$; however, they coincide with pullbacks of Schubert classes $\sigma_\nu^+$ from a larger symplectic Grassmannian $X^+=\IG(m,V^+)$. Here $V^+ \supset V$ is a $(2n+2)$--dimensional vector space endowed with a symplectic form $\omega^+$ such that $\omega^+_{\mid V}=\omega$, so that $X^+$ is the Grassmannian of isotropic $m$-spaces in $V^+$. 

Consider the embedding $j : X \hookrightarrow X^+$ and the induced surjective pullback map in cohomology $j^* : \Hr^*(X^+,\Z) \to \Hr^*(X,\Z)$. It was already noticed in~\cite{m:o} that $X$ can be identified with the Schubert variety of $X^+$ associated to the $(n+1-m)$--strict partition $1^m$. This immediately implies that the pullback map can be computed as follows, see also~\cite[Equation (2.3)]{pech:t}.

\begin{prop}\label{p:restrictions1}
	Let $\nu$ be an $(n+1-m)$--strict partition for $X^+$, and write
	\[
		\sigma_\nu^+ \cup \sigma_{1^m}^+ = \sum_{\beta} c_{\nu}^\beta \,\sigma_\beta^+,
	\]
	where the coefficients $c_{\nu}^\beta$ are computed from the Pieri rule of~\cite{PR}. Then 
	\[
		j^* \sigma_\nu^+ = \sum_\beta c_{\nu}^\beta\, \upsilon_{\beta-1^m} \in \Hr^*(X,\Z).
	\]
\end{prop}

Note that the classical Pieri rule of~\cite{PR} uses the indexation of Schubert classes of $X^+$ in terms of pairs of partitions introduced in Section~\ref{ss:schub-even}. Therefore in later sections we will use that indexation each time we wish to compute the coefficients $c_{\nu}^\beta$ explicitly.

We may now identify the basis elements (2) from Fact~\ref{f:bases} with pullbacks by $j$.

\begin{prop}\label{p:restrictions2}
	Let $\lambda$ be an $(n-m)$--strict partition of $Y=\IG(m,2n)$ and $\mu$ be an $(n+1-m)$-strict partition of $Z=\IG(m-1,2n)$. Then the classes $\sigma_\lambda,\tau_\mu' \in \Hr^*(X,\Z)$ satisfy
	\[
		\sigma_\lambda=j^* \sigma_{\lambda+1^m}^+ \text{ and } \tau_\mu'=j^*\sigma_\mu^+,
	\]
	where $\sigma_\nu^+$ denotes the Schubert class of $X^+=\IG(m,2n+2)$ associated with the $(n+1-m)$-strict partition $\nu$.
\end{prop}

\begin{proof}
	From Proposition~\ref{prop:poincare} we know that $\sigma_\lambda$ is the Poincar\'e dual of the class $\sigma_{\lambda^\vee}' \in \Hr^*(X,\Z)$, where $\lambda^\vee$ denotes the dual partition with respect to Poincar\'e duality \emph{in $Y=\IG(m,2n)$}. Using Proposition~2.11 from \cite{pech:t} we know that this Poincar\'e dual class coincides with a pulled-back class, i.e.
	\[
		\sigma_\lambda = j^*(\sigma^+_{\lambda^\vee+1^m})^\vee = j^*\sigma_{\lambda+1^m}^+,
	\]
	where last equality is obtained by looking at Poincar\'e duality in terms of index sets.
	The statement concerning $\tau_\mu'$ follows by a similar argument.
\end{proof}

\subsection{Ker/Span techniques}\label{s:kerspan}

In this section we relate degree $d$ Gromov-Witten invariants of the odd symplectic Grassmannian $X=\IG(m,2n+1)$ to intersection numbers in an auxiliary variety $I_d^+$. To obtain this comparison, following an idea introduced by Buch in~\cite{buch:grass} and Buch, Kresch and Tamvakis in~\cite{BKT1} we replace degree $d$ rational curves on $X$ or $X^+$ by their \emph{kernel} and \emph{span}.
\begin{defn}
	Let $C \subset X$ (resp. $C \subset X^+$) be a degree $d$ rational curve in $X$ (resp. in $X^+)$. 
	\begin{itemize}
		\item The \emph{kernel} $\Ker C$ of $C$ is the subspace of $V$ (resp. of $V^+$) defined by
			\[
				\Ker C := \bigcap_{\Sigma \in C} \Sigma.
			\]
		\item The \emph{span} $\Span C$ of $C$ is the subspace of $V$ (resp. of $V^+$) defined by
			\[
				\Span C := \Span \{ \Sigma \mid \Sigma \in C \}.
			\]
	\end{itemize}
\end{defn} 

We define a variety $I_{k,s,e}$ of kernel/span pairs in $V$ as a subvariety of the product of Grassmannians $G \times G' = \Gr(m-k,V) \times \Gr(m+s,V)$ as follows:
\[
	I_{k,s,e} := \left\{ (A,B) \in G \times G' \left|
		\begin{array}{l}
  			A \subset B \subset A^\perp \subset V,\\
  			\dim B \cap B^\perp = m-k+e \\
		\end{array}
	\right. \right\}.
\]
Here the orthogonal is defined with respect to the antisymmetric form of maximal rank $\omega$ on $V=\C^{2n+1}$. The variety $I_{k,s,e}^+$  of kernel/span pairs in $V^+=\C^{2n+2}$ is defined similarly, replacing $V$ with $V^+$ and orthogonality in $V$ with orthogonality in $V^+$ (using the symplectic form $\omega^+$ extending $\omega$). For simplicity we denote $\bigcup_ e I_{d,d,e}$ (resp. $\bigcup_ e I_{d,d,e}^+$) by $I_d$ (resp. $I_d^+$).

\begin{lemma}
	For any degree $d$ rational curve $C \subset  X$ we have
	\[
		(\Ker C,\Span C) \in \bigcup_{k,s=0}^d\bigcup_{e=0}^{k+s} I_{k,s,e}.
	\]
	The same result holds for $C \subset X^+$, replacing $I_{k,s,e}$ with $I_{k,s,e}^+$.
\end{lemma}

\begin{proof}
We have the inequalities $\dim \Ker C \geq m-d$ and $\dim \Span C \leq m+d$, see \cite{buch:grass}. Moreover, since the curve $C$ is contained in an isotropic Grassmannian we have that $\Ker C \subset \Span C \subset (\Ker C)^\perp$. The inequalities 
\[
	0 \leq \dim (\Span C \cap (\Span C)^\perp) - \dim \Ker C \leq \dim \Span C-\dim \Ker C.
\]
follow.
\end{proof}

Requiring that a given degree $d$ curve $C$ should intersect a given Schubert variety $X_{\lambda}$ of $X$ (resp. $X_{\lambda}^+$ of $X^+$) defines a subvariety of $I_{k,s,e}$ (resp. $I_{k,s,e}^+$) as follows. For the sake of readability, we omit here from our notations the flags with respect to which Schubert varieties are defined:
\begin{align*}
	I_{k,s,e}(\lambda) &= \left\{ (A,B) \in I_{k,s,e} \mid \exists \, \Sigma \in X_{\lambda}, A \subset \Sigma \subset B \right\}, \\
	I_{k,s,e}^+(\lambda) &= \left\{ (A,B) \in I_{k,s,e} \mid \exists \, \Sigma^+ \in X_{\lambda}^+, A \subset \Sigma^+ \subset B \right\}.
\end{align*}
We also write $I_d(\lambda)$ for $\bigcup_ e I_{d,d,e}(\lambda)$ (resp. $I_d^+(\lambda)$ for $\bigcup_ e I_{d,d,e}^+(\lambda)$).
We may now state our comparison result.

\begin{prop}\label{p:comparison}
	Let $d \geq 0$ be an integer and $\alpha,\beta,\gamma$ be elements of the cohomology basis $\{ \sigma_\lambda,\tau_\mu'\}$ of $X=\IG(m,2n+1)$ (that is, basis (2) from Fact~\ref{f:bases}, indexed using strict partitions). Assume that
	\[
		\codim \alpha+\codim \beta+ \codim \gamma= \dim \Mod{d}{3}{X}=\dim X + d(2n+2-m).
	\]
	Following Proposition~\ref{p:restrictions2} we denote by $\alpha^+,\beta^+,\gamma^+$ the Schubert classes of $X^+=\IG(m,2n+2)$ such that
	\[
		j^* \alpha^+ = \alpha; j^* \beta^+=\beta; j^*\gamma^+=\gamma.
	\]
	Let $X_{\alpha^+}^+$, $X_{\beta^+}^+$ and $X_{\gamma^+}^+$ be Schubert varieties of $X^+$ in general position associated with $\alpha^+$, $\beta^+$, and $\gamma^+$, respectively, and denote by $I_d^+(\alpha^+)$, $I_d^+(\beta^+)$, $I_d^+(\gamma^+)$ the corresponding subvarieties of $I_d^+$. Then 
	\[
		\scal{\alpha,\beta,\gamma}_{X,d,3} =\int_{I_d^+} [I_d^+(\alpha^+)] \cup [I_d^+(\beta^+)] \cup [I_d^+(\gamma^+)] \cup [I_d].
	\]
	Moreover, the Gromov-Witten invariant $\scal{\alpha,\beta,\gamma}_{X,d,3}$ vanishes unless each of the classes $\alpha, \beta, \gamma \in \Hr^*(X,\Z)$  is either of the form $\sigma_\lambda$ with $\lambda \supset \rho_{d-1}$, or of the form $\tau_\mu'$ with $\mu \supset \rho_d$. Here we denoted by $\rho_\ell$ the $\ell$-staircase partition $(\ell,\ell-1,\dots,1)$.
\end{prop}

A proof of Proposition~\ref{p:comparison}, provided Corollary~\ref{cor:enum} holds, was already given in~\cite[Section 5.1]{pech:t} with different notation. For the non--Francophone reader's convenience we repeat here the main steps of the proof.

\begin{proof}[Sketch of proof of Prop.~\ref{p:comparison}]
	First of all we notice via a dimension count that the intersection $I_{k,s,e}^+(\alpha^+) \cap I_{k,s,e}^+(\beta^+) \cap I_{k,s,e}^+(\gamma^+) \cap I_{k,s,e}$ is empty unless $k=s=d$ and $e=0$, in which case it consists of a finite number of points. Then we use Corollary~\ref{cor:enum} to prove our first statement, namely that
	\[
		\scal{\alpha,\beta,\gamma}_{X,d,3} =\int_{I_d^+} [I_d^+(\alpha^+)] \cup [I_d^+(\beta^+)] \cup [I_d^+(\gamma^+)] \cup [I_d].
	\]
	Let $f : \P^1 \to X \subset X^+$ be a degree $d$ morphism such that
	\[
		f(0) \in X_{\alpha^+}^+ \cap X, f(1) \in X_{\beta^+}^+ \cap X, f(\infty) \in X_{\gamma^+}^+ \cap X.
	\]
	Then the Ker/Span pair $(\Ker f,\Span f)$ is an element of some intersection
	\[
		I_{k,s,e}^+(\alpha^+) \cap I_{k,s,e}^+(\beta^+) \cap I_{k,s,e}^+(\gamma^+) \cap I_{k,s,e} \subset I_{k,s,e}^+.
	\]
	This means we must have $k=s=d$ and $e=0$, so that $(\Ker f,\Span f)$ is in fact an element of $I_d^+(\alpha^+) \cap I_d^+(\beta^+) \cap I_d^+(\gamma^+) \cap I_d \subset I_d^+$ as required. 
	
	Conversely, consider an element $(A,B) \in I_d^+(\alpha^+) \cap I_d^+(\beta^+) \cap I_d^+(\gamma^+) \cap I_d$. We need to show that there exists a unique degree $d$ morphism $f : \P^1 \to X$ such that $f(0) \in X_{\alpha^+}^+ \cap X$, $f(1) \in X_{\beta^+}^+ \cap X$, and $f(\infty) \in X_{\gamma^+}^+ \cap X$. By definition of $(A,B)$ there exist elements $P,Q,R$ of $X$ such that $A \subset P,Q,R \subset B$ and 
	\[
		P \in X_{\alpha^+}^+, Q \in X_{\beta^+}^+, R \in X_{\gamma^+}^+.
	\]
	By a dimension count, we show that $P \cap Q= P \cap R= Q \cap R = A$. Moreover the set $\left\{ \Sigma^+ \in X^+ \mid A \subset \Sigma^+ \subset B \right\} \cap X_{\alpha^+}^+$ must be zero--dimensional, otherwise we could choose a space $P$ in it whose intersection with $Q$ is larger than $A$. Similarly, the intersections with $X_{\beta^+}^+$ and $X_{\gamma^+}^+$ are also zero--dimensional. We then use the following lemma from~\cite{BKT1}.
        
	\begin{lemma}[{\cite[Lemma 1.3]{BKT1}}]
		Consider
		\[
			T_d^+ := \left\{ (A,\Sigma^+,B) \mid (A,B) \in I_d^+, \Sigma^+ \in X^+, A \subset \Sigma^+ \subset B \right\},
		\]
		and for any Schubert variety $X_{\alpha^+}^+$ of $X^+$,
		\[
			T_{d}^+(\alpha^+) := \left\{ (A,\Sigma^+,B) \in T_d^+ \mid \Sigma^+ \in X_{\alpha^+}^+ \right\}.
		\]
		The projection map $T_{d}^+(\alpha^+) \to I_d^+(\alpha^+)$ is generically 1:1 when $\nu \supset \rho_d$, where $\sigma^+_\nu=\alpha^+$, and has positive--dimensional fibres otherwise.
	\end{lemma}
	
	Applying the lemma to our setting we deduce that $\alpha^+=\sigma^+_\nu$ with $\nu \supset \rho_d$. Thus the invariant $\scal{\alpha,\beta,\gamma }_{X,d,3}$ vanishes when $\nu \not\supset \rho_d$, which is equivalent to the vanishing statement we need to prove. Moreover
	\[
		\left\{ \Sigma^+ \in X^+ \mid A \subset \Sigma^+ \subset B \right\} \cap X_{\alpha^+}^+ = \{P\} = \left\{ \Sigma \in X \mid A \subset \Sigma \subset B \right\},
	\]
	and similarly for $Q$ and $R$. We conclude the proof using the following result.
        
	\begin{lemma}[{\cite[Proposition 1]{BKT1}}]
		Let $P,Q,R$ be three elements of $X^+$ with pairwise intersections all equal to an $(m-d)$--dimensional subspace $A$. Then there exists a unique morphism $f : \P^1 \to X^+$ of degree $d$ such that
		\[
			f(0)=P,f(1)=Q,f(\infty)=R. 
		\]
	\end{lemma}
        Note that $(A,B) \in I_d$ implies $B \subset V \subset V^+$. Since $B$ is the span of the curve constructed by the former lemma, this curve factors through $X$.
\end{proof}

As a consequence of Proposition~\ref{p:comparison} we now obtain that the quantum Pieri rule for odd symplectic Grassmannians only involves classical terms and degree one Gromov--Witten invariants.

\begin{cor}\label{cor:deg1}
  Let $1 \leq p \leq 2n+1-m$ be an integer. Then no term of $q$--degree $2$ or higher appears in a quantum product $\tau'_p * \alpha$ with $\alpha = \sigma_\lambda$ or $\alpha = \tau'_\mu$.
\end{cor}

\begin{proof}
The terms in $q^d$ in the product $\tau'_p * \alpha$ are of the form $\scal{\tau'_p,\alpha,\beta}_{X,d,3}\, \beta^\vee$ with $\beta = \sigma_\nu$ or $\beta = \tau'_\xi$. By Proposition~\ref{p:comparison}, the corresponding Gromov-Witten invariant vanishes for any degree $d$ such that the staircase partition $\rho_d \not \subset (p)$, therefore, as soon as $d \geq 2$.
\end{proof}

\subsection{A quantum--to--classical principle}\label{s:quantclass}

Proposition~\ref{p:comparison} computes some three-pointed degree $d$ Gromov-Witten invariants of $X=\IG(m,2n+1)$ as numbers of points in the intersection of three given subvarieties of the variety $I_d^+$ of Ker/Span pairs of dimension $(m-d,m+d)$ in $V^+ \cong \C^{2n+2}$. We now focus on degree one invariants and prove a `quantum--to--classical principle' for them, that is, we express such invariants as numbers of intersection points of subvarieties of a larger odd symplectic Grassmannian $\tildeX=\IG(m+1,2n+3)$. 

Our first step is to relate our invariants to points in $\tildeXp=\IG(m+1,2n+4)$. We start by introducing some notation for $\tildeX$ and $\tildeXp$. Denote by $S$ a symplectic vector space of dimension $2$, and write
\[
	\tildeV := V \oplus S, \quad \tildeVp := V^+ \oplus S.
\]
Clearly the antisymmetric form of maximal rank $\omega$, respectively $\omega^+$, extends to an antisymmetric form of maximal rank on $\tildeV$, respectively $\tildeVp$, and both extend the symplectic form chosen on $S$. Thus we may define
\[
	\tildeX := \IG(m+1,\tildeV) \quad \textrm{and} \quad
	\tildeXp:=\IG(m+1,\tildeVp).
        \]
We also introduce some particular subvarieties of $\tildeX$ and $\tildeXp$. Let $f$ be a non--zero element of $S$, $F_\bullet^+$ be an $\Sp_{2n+2}$--flag in $V^+$, and $\lambda$ be an $(n+1-m)$--strict partition of $X^+$. Consider the \emph{augmented} $\Sp_{2n+4}$--flag $\tildeFp_\bullet$:
\[
	\{0\} \subset \tildeFp_1=\C f \subset \tildeFp_2=\C f \oplus F_1^+ \subset  \dots \subset \tildeFp_{2n+3}=\C f \oplus F^+_{2n+2}  \subset \tildeV^+
\]
and the associated Schubert variety
\[
	\tildeXp_\lambda := \left\{ \tildeSp \in \tildeXp \mid \dim\left(\tildeSp \cap \tildeFp_{p_j(\lambda)+1}\right) \geq j, \forall 1 \leq j \leq m \right\}.
\]
For simplicity we have omitted the flag $\tildeFp_\bullet$ from the notation on the left--hand side.

\begin{prop}\label{p:quant-class}
	Let $\alpha,\beta,\gamma$ be elements of the cohomology basis $\{ \sigma_\lambda,\tau_\mu'\}$ of $X=\IG(m,2n+1)$ such that
	\[
		\codim \alpha+\codim \beta+ \codim \gamma= \dim \Mod{1}{3}{X}=\dim X + 2n+2-m.
	\]
	Following Proposition~\ref{p:restrictions2} we denote by $\sigma^+_\lambda,\sigma^+_\mu,\sigma^+_\nu$ the Schubert classes of the symplectic Grassmannian $X^+=\IG(m,2n+2)$ such that
	\[
		j^* \sigma^+_\lambda = \alpha; j^* \sigma^+_\mu=\beta; j^* \sigma^+_\nu=\gamma.
	\]
	If $\ell(\lambda)+\ell(\mu)+\ell(\nu) \leq 2m+1$, then
	\[
		\scal{\alpha,\beta,\gamma}_{X,1,3} = \frac{1}{2}\# \left( \tildeXp_\lambda \cap \tildeXp_\mu \cap \tildeXp_\nu \cap \tildeX \right),
	\]
	where $\tildeXp_\lambda$, $\tildeXp_\mu$ and $\tildeXp_\nu$ are associated to generic flags.
\end{prop}

Again, providing Corollary~\ref{cor:enum} holds, a proof of this proposition~ was already given in~\cite[Section 5.2]{pech:t}, and we limit ourselves to giving the main steps of the proof.

\begin{proof}[Sketch of proof of Proposition~\ref{p:quant-class}]
	We first show that for a generic choice of flags we have a well--defined map
	\begin{align*}
		\phi : \tildeXp_\lambda \cap \tildeXp_\mu \cap \tildeXp_\nu \cap \tildeX &\to I^+_1(\sigma^+_\lambda) \cap I^+_1(\sigma^+_\mu) \cap I^+_1(\sigma^+_\nu) \cap I_1 \\
		\tildeSp &\mapsto \left( \tildeSp \cap V^+,(\tildeSp \oplus S) \cap V^+ \right).
	\end{align*}
        The group $G :=\Sp(V^+) \oplus \Sp(S)$ acts on $\tildeXp$ with the following orbits:
	\begin{enumerate}
		\item $\cO_1 \cong \IG(m+1,V^+)$,
		\item $\cO_2 \cong \IG(m,V^+) \times \P(S)$,
		\item $\cO_3 = \left\{ \tildeSp \in \tildeXp \mid \dim (\tildeSp \cap V^+)=m \text{ and } \tildeSp \cap S = \{0\} \right\}$,
		\item $\cO_4 = \left\{ \tildeSp \in \tildeXp \mid \dim (\tildeSp \cap V^+)=m-1 \right\}$.
	\end{enumerate}
	Using the assumptions on $\lambda, \mu,\nu$, a dimension count shows that the intersection $\tildeXp_\lambda \cap \tildeXp_\mu \cap \tildeXp_\nu \cap \tildeX$ must be contained in the open orbit $\cO_4$, which proves that the map $\phi$ is well--defined.
	
	It remains to show that for $(A,B)$ in the intersection
	\[
		I^+_1(\sigma^+_\lambda) \cap I^+_1(\sigma^+_\mu) \cap I^+_1(\sigma^+_\nu) \cap I_1,
	\]
	the inverse image $\phi^{-1}(A,B)$ consists of exactly two points. In the proof of Proposition~\ref{p:comparison} we saw that we must have $A=B \cap B^\perp$, so that $U:=B/A \oplus S$ is a four--dimensional symplectic vector space. Moreover the corresponding Lagrangian Grassmannian $\IG(2,U)$ identifies with a subvariety of $\tildeXp$, namely
	\[
		\IG(2,U) \cong \left\{ \tildeSp \mid A \subset \tildeSp \subset B \oplus S \right\}.
	\]
	In the proof of Proposition~\ref{p:comparison} we also saw that there exists a unique triple of pairwise distinct subspaces $\Sigma_1^+ \in X^+_\lambda \cap X$, $\Sigma_2^+ \in X^+_\mu \cap X$ and $\Sigma_3^+ \in X^+_\nu \cap X$ such that $A \subset \Sigma^+_j \subset B$ for $j=1,2,3$. Moreover by a Schubert calculus argument it is easy to see that there exist exactly two elements of $\IG(2,U)$ which are incident to all three of $\Sigma_1/A \oplus \C f$, $\Sigma_2/A \oplus \C g$, and $\Sigma_3/A \oplus \C h$, where $f,g,h$ are generic elements of $S$. Thus there also exist two elements $\tildeSp_{j=1,2} \in \tildeXp$ such that $A \subset \tildeSp_j \subset B \oplus S$ and the dimensions
	\[
		\dim\left( \tildeSp_j \cap \Sigma_1 \oplus \C f \right), \dim\left( \tildeSp_j \cap \Sigma_2 \oplus \C g \right), \dim\left( \tildeSp_j \cap \Sigma_3 \oplus \C h \right)
	\]
	are all at least equal to $m$. Clearly the $\tildeSp_j$ belong to the quadruple intersection $\tildeXp_\lambda \cap \tildeXp_\mu \cap \tildeXp_\nu \cap \tildeX$.
	
	Conversely if $\tildeSp \in \tildeXp_\lambda \cap \tildeXp_\mu \cap \tildeXp_\nu \cap \tildeX$ is such that $A \subset \tildeSp_j \subset B \oplus S$, then we must have $A = \tildeSp \cap V^+$ and $B = (\tildeSp \oplus S) \cap V^+$, and one can also show that $\Sigma_1 = (\tildeSp \oplus \C f) \cap V^+$, and similarly for $\Sigma_2$ and $\Sigma_3$. This implies that the dimensions
	\[
		\dim\left( \tildeSp_j \cap \Sigma_1 \oplus \C f \right), \dim\left( \tildeSp_j \cap \Sigma_2 \oplus \C g \right), \dim\left( \tildeSp_j \cap \Sigma_3 \oplus \C h \right)
	\]
	are all at least equal to $m$, thus $\phi^{-1}(A,B)$ consists of exactly two points.
\end{proof}

We may now state our quantum--to--classical principle, which is an immediate corollary of Proposition~\ref{p:quant-class}. In the following theorem we denote by $\tildes_u$ the Schubert class of $\tildeX=\IG(m+1,2n+3)$ associated with an $(n-m)$--strict partition $u$, and by $\tildet_v'$ the Schubert class of $\tildeX$ associated with an $(n+1-m)$--strict partition $v$.

\begin{thm}\label{t:quant-class}
	Let $\alpha,\beta,\gamma$ be elements of the cohomology basis $\{ \sigma_u,\tau_v'\}$ of $X=\IG(m,2n+1)$ such that
	\[
		\codim \alpha+\codim \beta+ \codim \gamma= \dim X+2n+2-m.
	\]
	Denote by $\sigma^+_\lambda,\sigma^+_\mu,\sigma^+_\nu$ the Schubert classes of $X^+=\IG(m,2n+2)$ such that
	\[
		j^* \sigma^+_\lambda = \alpha; j^* \sigma^+_\mu=\beta; j^* \sigma^+_\nu=\gamma.
	\]
	If $\ell(\lambda)+\ell(\mu)+\ell(\nu) \leq 2m+1$, then
	\[
		\scal{\alpha,\beta,\gamma}_{X,1,3} = \frac{1}{2} \scal{\tilde\alpha,\tilde{\beta},\tilde{\gamma}}_{\tildeX,0,3},
	\]
	where 
	\begin{align*}
		\tilde\alpha = 	\begin{cases}
								\tildes_u & \text{if $\alpha=\sigma_u$}, \\
								\tildet_v' & \text{if $\alpha=\tau_v'$},
							\end{cases}
	\end{align*}
	and similarly for $\tilde{\beta}$ and $\tilde{\gamma}$.
\end{thm}

\begin{remark}
  The condition $\ell(\lambda)+\ell(\mu)+\ell(\nu) \leq 2m+1$ will always be satisfied in the case where one of the partitions, say $\nu$, has length one, i.e. $\nu=(p,0,\dots,0)$ for some $1 \leq p \leq 2n+1-m$. Keeping in mind the vanishing of invariants of degree larger than one, see Corollary~\ref{cor:deg1}, it follows that one may use Theorem~\ref{t:quant-class} to deduce a \emph{quantum Pieri formula} for the product of a class $\sigma_u$ or $\tau_v'$ by the classes $\tau_1',\dots,\tau'_{2n+1-m}$. 
\end{remark}

\subsection{Quantum presentation}\label{s:quantpres}

In this section we give a presentation for the quantum cohomology ring of $\IG(m,2n+1)$ in terms of the classes $\tau_1',\tau_2',\dots,\tau_{2n+1-m}'$ and the quantum parameter $q$. This presentation is a quantum version of the classical presentation obtained in \cite{pech:t}, which we recall here.

\begin{defn}
	For $r \geq 1$, define
	\[
		d_r := \det(\tau_{1+j-i}')_{1 \leq i,j \leq r} \quad \textrm{and} \quad b_r:= (\tau_r')^2 + 2 \sum_{i \geq 1} (-1)^i \tau_{r+i}' \tau_{r-i}',
	        \]
	with the convention that $\tau_0' := 1$ and $\tau_p' := 0$ if $p<0$ or $p>2n+1-m$.
\end{defn}

The following result is proved in the second author's PhD thesis, see \cite[Prop. 2.12]{pech:t}. However it is in French, therefore we repeat the proof here.

\begin{prop}\label{p:classical-presentation}
	The cohomology ring $\Hr^*(\IG(m,2n+1),\Z)$ is generated by the classes $(\tau'_p)_{1 \leq p \leq 2n+1-m}$ and the relations are
	\begin{align*}
		d_r=0 & \text{ for $m+1 \leq r \leq 2n+2-m$,} \\
		b_r=0 & \text{ for $n+2-m \leq r \leq n$}.
	\end{align*}
\end{prop}

\begin{proof}
Let $\SX$ denote the tautological bundle of $X=\IG(m,2n+1)$, and $\SXp$ that of $X^+=\IG(m,2n+2)$. Similarly we denote by $\QX$ and $\QXp$ the corresponding quotient bundles. We know that the Chern classes of $\QXp$ generate the cohomology of $X^+$, see~\cite[Section 1.2]{BKT2}, and it follows from Proposition~\ref{p:restrictions2} that the restriction map $\Hr^*(X^+,\Z) \to \Hr^*(X,\Z)$ is surjective, therefore, the Chern classes $c_p(\QX)$ for $1 \leq p \leq 2n+2-m$ generate the cohomology ring of $X$. Moreover, since $\QX$ has rank $2n+1-m$, the Chern class $c_{2n+2-m}(\QX)$ vanishes, hence $\Hr^*(X,\Z)$ is generated by the classes $c_p(\QX)=\tau_p'$ for $1 \leq p \leq 2n+1-m$. 

To find the relations, we use the method introduced in~\cite{BKT2} to obtain presentations of the cohomology of isotropic Grassmannians. Namely, let $R := \Z[a_1,\dots,a_{2n+1-m}]$ be a graded ring with $\deg a_i=i$. Set $a_0=0$ and $a_i=0$ if $i>2n+1-m$. We let $\delta_0=1$ and $\delta_r=\det(a_{1+j-i})_{1 \leq i,j \leq r}$ for $r >0$. We also set $\beta_r=a_r^2+2\sum_{i \geq 1} (-1)^i a_{r+i}a_{r-i}$ for $r \geq 0$. Now we define a (surjective) graded ring homomorphism $\phi \colon R \to \Hr^*(X,\Z)$ by setting $\phi(a_i)= \tau_i'$ for $1 \leq i \leq 2n+1-m$. We start by checking that the relations hold, i.e., that $\phi(\delta_r)=0$ for $r>m$ and $\phi(\beta_r)=0$ for $n+2-m \leq r \leq n$.

Expanding each determinant $\delta_r$ with respect to the first row yields the identity of formal series:
\[
	\left(\sum_{i=0}^{2n-1} a_i t^i \right)\left(\sum_{i \geq 0} (-1)^i \delta_i t^i\right) = 1.
\]
Noticing that the image by $\phi$ of the leftmost factor is the total Chern class $c(\QX)$, and that $c(\QX)c(\SX)=1$, we deduce that
\[
	\phi(\SX)=\left(\sum_{i \geq 0} (-1)^i \phi(\delta_i) t^i\right).
\]
Since $\SX$ has rank $m$ we obtain that $\phi(\delta_r)=d_r=0$ for $r>m$.

The quadratic relations $b_r$ come from similar relations on $X^+=\IG(m,2n+2)$. Indeed it follows from the presentation of $\Hr^*(X^+,\Z)$ proved in \cite{BKT2} that the relations
\[	
	\sigma^+_r + 2\sum_{i \geq 1} (-1)^i \sigma^+_{r+i}\sigma^+_{r-i}=0
\]
hold for $n+2-m \leq r \leq n$, and their restriction to $X$ gives $b_r=0$.

It now remains to prove that both sets of relations generate the kernel of $\phi$. This is done using~\cite[Lemma 1.1]{BKT2}, which states that to do so we only need to check that
\begin{itemize}
\item $\Hr^*(X,\Z)$ is a free $\Z$--module of rank $\deg\left(\frac{\prod_r \delta_r \prod_r \beta_r}{\prod_i a_i}\right)$;
\item for any field $K$ the $K$--vector space $(R/I) \otimes_\Z K$ has finite dimension, where $I$ is the ideal generated by the relations $\delta_r$ and $b_r$.
\end{itemize}
The rank of $\Hr^*(X,\Z)$ is easily seen to be equal to $2^{m-1}\binom{n}{m}\frac{2n+2-m}{n+1-m}$ as required, see~\cite[Section 2.2.3]{pech:t}. To check the second condition we prove that $R/I$ is a quotient of $R/(\delta_{m+1},\dots,\delta_{2n+1})$, which is a free $\Z$--module of finite rank, see~\cite[Lemma 1.2]{BKT2}. This means that we need to prove that the relations $\delta_r$ for $2n+3-m \leq r \leq 2n+1$ also belong to the ideal $I$. We use the identities of formal series
\[
	\left( \sum_{i=0}^{2n+1-m} a_i t^i \right)\left(\sum_{i=0}^{2n+1-m} (-1)^i a_i t^i\right) = \left(\sum_{i=0}^{2n+1-m} (-1)^i \beta_i t^{2i}\right)
\]
and
\[
	\left(\sum_{i=0}^{2n+1-m} (-1)^i a_i t^i\right)\left(\sum_{i \geq 0} \delta_i t^i\right)=1,
\]
so that
\[
	\left( \sum_{i=0}^{2n+1-m} a_i t^i \right)= \left(\sum_{i=0}^{2n+1-m} (-1)^i \beta_i t^{2i}\right)\left(\sum_{i \geq 0} \delta_i t^i\right).
\]
Modding out by the relations in $I$ gives
\[
	\left( \sum_{i=0}^{2n+1-m} a_i t^i \right) \equiv \left(\sum_{i=0}^{n+1-m} (-1)^i \beta_i t^{2i} + \sum_{i=n+1}^{2n+1-m} (-1)^i \beta_i t^{2i} \right)\left(\sum_{i=0}^m \delta_i t^i + \sum_{i \geq 2n+3-m} \delta_i t^i \right).
\]
Looking at the terms in degrees $2n+3-m$ to $2n+1$ we obtain $0 \equiv \delta_{2n+3-m} \equiv \dots \equiv \delta_{2n+1}$ as claimed, which concludes the proof.
\end{proof}

\begin{thm}\label{t:quantpres}
  The quantum cohomology ring $\QH(\IG(m,2n+1),\C)$ is generated by the classes $(\tau_p')_{1 \leq p \leq 2n+1-m}$ and the quantum parameter $q$, and the relations are
\[
	{\renewcommand{\arraystretch}{2.2}
\begin{array}{lcl}
		d_r & = & 0 \text{ for $m+1 \leq r \leq 2n+1-m$,} \\
		d_{2n+2-m} & = & \begin{cases}
						0 & \text{if $m$ is even,} \\
						-q & \text{if $m$ is odd,}
					 \end{cases} \\
		b_s & = & (-1)^{2n+1-m-s} q \tau_{2s-2n-2+m}' \text{ for $n+2-m \leq s \leq n$,}
	\end{array}}
\]
	with the convention that $\tau_p'=0$ if $p<0$.
\end{thm}

\begin{proof}
	To obtain this quantum presentation from the classical presentation of Proposition~\ref{p:classical-presentation}, we use the well--known method due to Siebert and Tian, see~\cite[Prop. 2.2]{ST}. This means that we replace the cup product with the quantum product in the classical relations $d_r$ and $b_s$. The quantum presentation is then obtained by considering these new relations instead of the original ones.
	
	Let us start by deforming the relations $b_s$ for $n+2-m \leq s \leq n$. We have
	\begin{align*}
		b_s^{\text{quant}} =& \tau_s' * \tau_s' + 2\sum_{i=1}^{2n+1-m-s} (-1)^i \tau_{s+i}' * \tau_{s-i}' \\
		= & q \sum_\lambda \Bigg( \scal{j^* \sigma_s^+,j^* \sigma_s^+,j^* \sigma_\lambda^+}_{X,1,3}  \\
		&+ \left. 2\sum_{i=1}^{2n+1-m-s} (-1)^i \scal{j^* \sigma_{s+i}^+,j^* \sigma_{s-i}^+,j^* \sigma_\lambda^+}_{X,1,3} \right) (j^* \sigma_\lambda^+)^\vee,
	\end{align*}
	where $j : X \hookrightarrow X^+=\IG(m,2n+2)$ is the natural inclusion.
	Indeed, the classical part of $b_s^{\text{quant}}$ is the classical relation $b_s$, which vanishes, and by Proposition~\ref{p:restrictions2} we have that $\tau_p' = j^*\sigma_p^+$. We now use the quantum--to--classical principle from Theorem~\ref{t:quant-class}. 
	
	Write $\tildeX := \IG(m+1,2n+3)$, $\tildeXp:=\IG(m+1,2n+4)$, and denote by $\tildej : \tildeX \hookrightarrow \tildeXp$ the natural inclusion. We also denote by $\tildes_\lambda^+$ the Schubert class of $\tildeXp$ associated with an $(n+1-m)$--strict partition $\lambda$ with $\lambda_1 \leq 2n+3-m$. Finally, we introduce the subvarieties $\tildeY$ and $\tildeZ$ of $\tildeX$ and the cohomology classes $\{\tildes_\mu,\tildet_\nu'\}$ and $\{\tildes_\mu',\tildet_\nu\}$ of $\tildeX$ as in Section~\ref{ss:schub-odd}, simply replacing $X$ with $\tildeX$ and adding the symbol $\tilde{}$ everywhere.
	It follows from Theorem~\ref{t:quant-class} that
	\begin{align*}
		b_s^{\text{quant}} =& \frac{1}{2}q \sum_\lambda \Bigg(\scal{\tildej^* \tildes_s^+,\tildej^* \tildes_s^+,\tildej^* \tildes_\lambda^+}_{\tildeX,0,3}  \\
		&\left. + 2\sum_{i=1}^{2n+1-m-s} (-1)^i \scal{\tildej^* \tildes_{s+i}^+,\tildej^* \tildes_{s-i}^+,\tildej^* \tildes_\lambda^+}_{\tildeX,0,3} \right) (j^* \sigma_\lambda^+)^\vee \\
		=& \frac{1}{2}q \sum_\lambda \left( \int_{\tildeX} \left( \tildet_s' \tildet_s' +2 \sum_{i=1}^{2n+1-m-s} (-1)^i \tildet_{s+i}'\tildet_{s-i}' \right) \cup \tildej^* \tildes_\lambda^+ \right) (j^* \sigma_\lambda^+)^\vee.
	\end{align*}
	In $\Hr^*(\tildeX,\Z)$, by Proposition~\ref{p:classical-presentation}, the classical relation
	\[
		\tildet_s' \tildet_s' +2 \sum_{i=1}^{2n+2-m-s} (-1)^i \tildet_{s+i}'\tildet_{s-i}'=0
	\]
	is satisfied, thus 
	\[
		\tildet_s' \tildet_s' +2 \sum_{i=1}^{2n+1-m-s} (-1)^i \tildet_{s+i}'\tildet_{s-i}' = 2(-1)^{2n+1-m-s} \tildet_{2n+2-m}'\tildet_{2s-2n-2+m}'.
	\]
	Replacing in the expression of $b_s^{\text{quant}}$, we obtain
	\begin{align*}
		b_s^{\text{quant}} &= (-1)^{2n+1-m-s}q \sum_\lambda \left( \int_{\tildeX}  \tildet_{2n+2-m}' \cup \tildet_{2s-2n-2+m}' \cup \tildej^* \tildes_\lambda^+ \right) (j^* \sigma_\lambda^+)^\vee \\
		&= (-1)^{2n+1-m-s}q \sum_\lambda \left( \int_{\tildeX}  \tildej^* (\tildes_{2n+2-m}^+ \cup \tildes_{2s-2n-2+m}^+) \cup \tildej^* \tildes_\lambda^+ \right) (j^* \sigma_\lambda^+)^\vee.
		\end{align*}
	To simplify notation we introduce the integer $p:=2s-2n-2+m$. We use the classical Pieri rule in $\tildeXp$ to compute the product $\tildes_{2n+2-m}^+ \cup \tildes_{p}^+$, see \cite[Theorem 1.1]{BKT2}:
	\[
		\tildes_{2n+2-m}^+ \cup \tildes_p^+= \tildes_{2n+2-m,p}^+ +2 \tildes_{2n+3-m,p-1}^+.
	\]
	We now prove that the second summand will not contribute to $b_s^{\text{quant}}$. Indeed, let us prove that $\tildej^* \tildes_\lambda^+=0$ when $\lambda_1=2n+3-m$. From Proposition~\ref{p:restrictions1} applied to $\tildeX$ we have that $\tildej^* \tildes_\lambda^+=0$ when $\tildes_\lambda^+ \cup \tildes_{1^{m+1}}^+=0$. The latter cup product can be computed using the Pieri rule of~\cite{PR}. Namely, let $\alpha=(\alpha^t,\alpha^b)$ be the pair of partitions associated with the $(n+1-m)$--strict partition $\lambda$ following the correspondence introduced in Section~\ref{ss:schub-even}. Clearly, $\lambda_1=2n+3-m$ implies that $\alpha^b_1=n+2$, i.e., the first part of $\alpha^b$ is maximal. Therefore, $\tildes^+_\alpha \cup \tildes^+_{1^{m+1}}=0$ by the Pieri rule. For more details we refer to~\cite[Section 2.3]{pech:t}. Thus we indeed have $\tildej^* \tildes_\lambda^+=0$ when $\lambda_1=2n+3-m$. It follows that
	\begin{align*}
		b_s^{\text{quant}} &= (-1)^{2n+1-m-s}q \sum_\lambda \left( \int_{\tildeX}  \tildej^* \tildes_{2n+2-m,p}^+ \cup \tildej^* \tildes_\lambda^+ \right) (j^* \sigma_\lambda^+)^\vee.
	\end{align*}
	Let us now express the restricted class $\tildej^* \tildes_{2n+2-m,p}^+$ in terms of the cohomology basis $\{\tildes_\mu,\tildet_\nu'\}$. As before, from Proposition~\ref{p:restrictions1} it follows that we first need to compute the cup product $\tildes_{2n+2-m,p}^+ \cup \tildes_{1^{m+1}}^+$ in $\Hr^*(\tildeXp,\Z)$. 
	
	We use again the classical Pieri rule for multiplication by the classes $\tildes_{1^r}^+$ from \cite{PR}, noting that the strict partition $(2n+2-m,p)$ corresponds to the pair of partitions 
	\[
		\alpha^b=(n+1,p+m-1-n), \alpha^t=(n+3-m,\dots,4,3)
	\]
	if $p>n+1-m$, and to
	\[
			\alpha^b=(n+1), \alpha^t=(n+3-m,n+2-m,\dots,n+4-m-p,n+2-m-p,\dots,3,2)
	\]
	otherwise.
	We deduce that
	\[
		\tildes_{2n+2-m,p}^+ \cup \tildes_{1^{m+1}}^+ = \tildes_{2n+3-m,p+1,1^{m-1}}^+
	\]
	if $p \leq 2n+1-2m$, and otherwise
	\[
		\tildes_{2n+2-m,p}^+ \cup \tildes_{1^{m+1}}^+ = \tildes_{2n+3-m,p+1,1^{m-1}}^++\tildes_{2n+3-m,2n+2-m,1^{p-2n-2+2m}}^+.
	\]
	Hence in the first case,
	\[
		\tildej^* \tildes_{2n+2-m,p}^+ = [\tildeX_{2n+2-m,p}] = \tildet_{p+1,1^{m-1}},
	\]
	while in the second case,
	\[
		\tildej^* \tildes_{2n+2-m,p}^+ = \tildet_{p+1,1^{m-1}} + \tildet_{2n+2-m,1^{p-2n-2+2m}}.
	\]
	By Poincar\'e duality, see Proposition~\ref{prop:poincare}, we know that $\int_{\tildeX}  \tildet_{p+1,1^{m-1}} \cup \tildej^* \tildes_\lambda^+$ is equal to $1$ if $\lambda_{m+1}=0$ and $\lambda$ is the Poincar\'e dual partition to $(p+1,1^{m-1})$ in $\tildeZ=\IG(m,2n+2)$. Otherwise it is equal to zero, and similarly for $\int_{\tildeX}  \tildet_{2n+2-m,1^{p-2n-2+2m}} \cup \tildej^* \tildes_\lambda^+$.
	
	Using \cite[Section 4.4]{BKT2}, we see that the Poincar\'e dual partition to $(p+1,1^{m-1})$ is
	\[
		(2n+1-m,2n-m,\dots,2n+3-2m,2n+2-2m-p)
	\]
	if $p \leq 2n+1-2m$, and
	\[
		(2n+1-m,2n-m,\dots,p+2,p,p-1,\dots,2n+2-2m,1) 
	\]
	otherwise. Similarly, the Poincar\'e dual partition to $(2n+2-m,1^{p-2n-2+2m})$ is
	\[
		(2n+1-m,2n-m,\dots,p+1,p-1,p-2,\dots,2n+2-2m).
	\]
	So if $p \leq 2n+1-2m$ then
	\begin{align*}
		b_s^{\text{quant}} &= (-1)^{2n+1-m-s}q (j^* \sigma_{2n+1-m,2n-m,\dots,2n+3-2m,2n+2-2m-p}^+)^\vee \\
		&= (-1)^{2n+1-m-s}q \sigma_{2n-m,2n-1-m,\dots,2n+2-2m,2n+1-2m-p}^\vee,
	\end{align*}
	and otherwise,
	\begin{align*}
		b_s^{\text{quant}} =& (-1)^{2n+1-m-s}q (j^* \sigma_{2n+1-m,\dots,p+2,p,\dots,2n+2-2m,1}^+)^\vee \\
		&+ (-1)^{2n+1-m-s}q (j^* \sigma_{2n+1-m,\dots,p+1,p-1,\dots,2n+2-2m}^+)^\vee\\
		=& (-1)^{2n+1-m-s}q \sigma_{2n-m,\dots,p+1,p-1,\dots,2n+1-2m}^\vee \\
		&+ (-1)^{2n+1-m-s}q (\tau_{2n+1-m,\dots,p+1,p-1,\dots,2n+2-2m}')^\vee.
	\end{align*}
	Moreover, if $p \leq 2n+1-2m$ then the Poincar\'e dual partition to the partition $(2n-m,2n-1-m,\dots,2n+2-2m,2n+1-2m-p)$ in $Y$ is $(p)$. On the other hand if $p \geq 2n+2-2m$ the Poincar\'e dual partition to $(2n-m,\dots,p+1,p-1,\dots,2n+1-2m)$ in $Y$ is $(p)$, while the Poincar\'e dual partition to $(2n+1-m,\dots,p+1,p-1,\dots,2n+2-2m)$ in $Z$ is $(1^{p-2n-2+2m})$. Hence
	\begin{align*}
		b_s^{\text{quant}} = \begin{cases}
			(-1)^{2n+1-m-s}q \sigma_{p}' & \text{if $p \leq 2n+1-2m$,}\\
			(-1)^{2n+1-m-s}q (\sigma_{p}'+\tau_{1^{p-2n-2+2m}}) & \text{otherwise.}
							 \end{cases}
	\end{align*}
	Let us compare the right--hand side of the above equation with the pullback $j^* \sigma_{p}^+$ using Proposition~\ref{p:restrictions1} and the classical Pieri rule from~\cite{PR}. Note that the strict partition $(p)$ corresponds to the pair of partitions given by
	\[
		\begin{cases}
			\alpha^b=(p+m-1-n), \alpha^t=(n+2-m,n+1-m,\dots,2) & \text{if $p>n+1-m$,} \\
			\alpha^b=\emptyset, \alpha^t=(n+1-m,n-m,\dots,1) & \text{otherwise}.
		\end{cases}
	\]
	We obtain
	\begin{align*}
		\sigma_{p}^+ \cup \sigma_{1^m}^+ = \begin{cases}
			\sigma_{p+1,1^{m-1}}^+ & \text{if $p \leq 2n+1-2m$,} \\
			\sigma_{p+1,1^{m-1}}^+ +\sigma_{2n+2-m,1^{p-2n-2+2m}}^+ & \text{otherwise.}
												   \end{cases}
	\end{align*}
	It follows from Proposition~\ref{p:restrictions1} that 
	\begin{align*}
		j^* \sigma_{p}^+ = \begin{cases}
			\sigma_{p}' & \text{if $p \leq 2n+1-2m$,} \\
			\sigma_{p}' +\tau_{1^{p-2n-2+2m}} & \text{otherwise.}
												   \end{cases}
	\end{align*}
	Since by definition $j^* \sigma_{p}^+=\tau_{p}'$ we finally get that
	\[
		b_s^{\text{quant}}=(-1)^{2n+1-m-s} q \tau_{2s-2n-2+m}'
	\]
	as claimed.
	
	In the second part of this proof, let us now deform the relations $d_r$ for $m+1 \leq r \leq 2n+2-m$. Since $\deg q=2n+2-m$, only the relation $d_{2n+2-m}$ is potentially modified. We consider the determinantal identity
	\[
		d_{2n+2-m}^{\text{quant}} = \sum_{p=1}^{2n+1-m} (-1)^{p+1} \tau_p' * d_{2n+2-p-m}.
	\]
	The classical part of the right--hand side of the above equality being trivial, it follows that
	\[
		d_{2n+2-m}^{\text{quant}} = q \sum_{p=1}^{2n+1-m} (-1)^{p+1} \scal{\tau_p', d_{2n+2-m-p} ,[pt]}_{X,1,3}.
	\]
	Moreover we have that for any $1 \leq s \leq m$, $d_s=j^* \sigma_{1^s}^+$, and $d_{s}=0$ if $s>m$, see \cite[Section 2.5]{pech:t}. This implies that
	\begin{align*}
		d_{2n+2-m}^{\text{quant}} &= q \sum_{p=2n+2-2m}^{2n+1-m} (-1)^{p+1} \scal{ \tau_p', d_{2n+2-m-p} ,[pt]}_{X,1,3} \\
					&= q \sum_{s=1}^m (-1)^{2n+3-m-s} \scal{\tau_{2n+2-m-s}', j^* \sigma_{1^s}^+ ,[pt]}_{X,1,3}.
	\end{align*}
	We have $\tau_{2n+2-m-s}'=j^* \sigma_{2n+2-m-s}^+$, and the class of a point in $X$ is 
	\[
		[pt]=\sigma_{2n-m,2n-1-m,\dots,2n+1-2m}=j^* \sigma_{2n+1-m,2n-m,\dots,2n+2-2m}^+,
	\]
	so using the quantum to classical principle we get
    \begin{align*}
		d_{2n+2-m}^{\text{quant}} =& \frac{1}{2} q \sum_{s=1}^m (-1)^{2n+3-m-s} \\
			& \scal{\tildej^* \tildes_{2n+2-m-s}^+, \tildej^* \tildes_{1^s}^+ , \tildej^* \tildes_{2n+1-m,2n-m,\dots,2n+2-2m}^+}_{\tildeX,0,3} \\
			=& \frac{1}{2} q \int_{\Xt} \left( \sum_{s=1}^m (-1)^{2n+3-m-s} \tildej^* \tildes_{2n+2-m-s}^+  \cup \tildej^* \tildes_{1^s}^+ \right) \\
			& \cup \tildej^* \tildes_{2n+1-m,2n-m,\dots,2n+2-2m}^+,
	\end{align*}
	where the superscript symbol $\;\tilde{}\;$ indicates that we are considering classes in $\tildeXp = \IG(m+1,2n+4)$. 
	Let us now look at the \emph{classical} relation $\tilde{d}_{2n+2-m}$ in $\tildeX$. Using a similar determinantal identity as before, we get
	\begin{align*}
		\tilde{d}_{2n+2-m} &= \sum_{p=1}^{2n+2-m} (-1)^{p+1} \tildet_p' \cup \tilde{d}_{2n+2-p-m} \\
			&= \sum_{p=2n+1-2m}^{2n+2-m} (-1)^{p+1} \tildet_p' \cup \tilde{d}_{2n+2-p-m} \\
			&= \sum_{s=0}^{m+1} (-1)^{2n+3-m-s} \tildet_{2n+2-m-s}' \cup \tilde{d}_s \\
			&= \sum_{s=0}^{m+1} (-1)^{2n+3-m-s} \tildej^* \tildes_{2n+2-m-s}^+  \cup \tildej^* \tildes_{1^s}^+.
	\end{align*}
	Thus replacing in the expression for $d_{2n+2-m}^{\text{quant}}$, we get
        	\begin{align*}
		d_{2n+2-m}^{\text{quant}} = & \frac{q}{2} \int_{\Xt} \left( \tilde{d}_{2n+2-m} -(-1)^{2n+3-m} \tildej^* \tildes_{2n+2-m}^+ \right. \\
			& \left.-(-1)^{2n+2-2m} \tildej^* \tildes_{2n+1-2m}^+ \cup \tildej^* \tildes_{1^{m+1}}^+ \right) \cup \tildej^* \tildes_{2n+1-m,2n-m,\dots,2n+2-2m}^+ \\
			= & (-1)^{2n+2-m}\frac{q}{2} \int_{\tildeX} \tildej^* ( \tildes_{2n+2-m}^+ \cup \tildes_{2n+1-m,2n-m,\dots,2n+2-2m}^+) \\
			&+ (-1)^{2n+1-2m}\frac{q}{2} \int_{\tildeX} \tildej^* \tildes_{2n+1-2m}^+ \cup \tildej^* (\tildes_{1^{m+1}}^+ \cup \tildes_{2n+1-m,\dots,2n+2-2m}^+).
	\end{align*}
	Using the classical Pieri rule in $\tildeXp$, see \cite[Theorem 1.1]{BKT2}, we have
	\begin{align*}
		\tildes_{2n+2-m}^+ \cup \tildes_{2n+1-m,\dots,2n+2-2m}^+ =& \tildes_{2n+2-m,2n+1-m,\dots,2n+2-2m}^+ \\
		&+ 2\tildes_{2n+3-m,2n+1-m,\dots,2n+3-2m,2n+1-2m}^+.
	\end{align*}
	Moreover, using the classical Pieri rule from~\cite{PR},
	\begin{align*}
		\tildes_{1^{m+1}}^+ \cup \tildes_{2n+1-m,2n-m,\dots,2n+2-2m}^+ =& \tildes_{2n+2-m,\dots,2n+3-2m,1}^+ \\
		&+ \tildes_{2n+3-m,2n+1-m,\dots,2n+3-2m}^+.
	\end{align*}
	Since we saw that $\tildej^* \sigma_\lambda=0$ if $\lambda_1=2n+3-m$, we obtain
	\begin{align*}
		d_{2n+2-m}^{\text{quant}} =& (-1)^{2n+2-m}\frac{q}{2} \int_{\tildeX} \tildej^* \tildes_{2n+2-m,2n+1-m,\dots,2n+2-2m}^+ \\
		& + (-1)^{2n+1-2m}\frac{q}{2} \int_{\tildeX} \tildej^* \tildes_{2n+1-2m}^+ \cup \tildej^* \tildes_{2n+2-m,\dots,2n+3-2m,1}^+.
	\end{align*}
	The first integral is equal to $1$ since $j^* \tildes_{2n+2-m,2n+1-m,\dots,2n+2-2m}^+$ is the class of a point in $\tildeX$. Thus it only remains to compute the following cup product using the classical Pieri rule \cite[Theorem 1.1]{BKT2},
	\begin{align*}
		\tildes_{2n+1-2m}^+ \cup \tildes_{2n+2-m,\dots,2n+3-2m,1}^+ =& \tildes_{2n+2-m,\dots,2n+2-2m}^+ \\
			&+ \tildes_{2n+3-m,2n+1-m,\dots,2n+3-2m,2n+1-2m}^+.
	\end{align*}
	Since $\tildej^* \tildes_{2n+3-m,2n+1-m,\dots,2n+3-2m,2n+1-2m}^+=0$ and $j^* \tildes_{2n+2-m,\dots,2n+2-2m}^+$ is the class of a point in $\tildeX$, we get
	\begin{align*}
		d_{2n+2-m}^{\text{quant}} =& \left((-1)^{2n+2-m}+ (-1)^{2n+1-2m}\right)\frac{q}{2} \\
		=& \begin{cases}
				0 & \text{if $m$ is even,} \\
				-q & \text{if $m$ is odd,}
			\end{cases}
	\end{align*}
	which concludes the proof of the theorem.
\end{proof}

\section{Recollections on semiorthogonal decompositions}

\subsection{Semiorthogonal decompositions}

We recall here basic definitions on semi\-orthogonal decompositions and mutations in triangulated categories. For convenient references, see, for instance~\cite[Sections 1.2 and 1.4]{Huyb} and \cite[Section 2]{Kuz}.

Let $k$ be a field. Assume given a $k$-linear triangulated category $\Dd$, equipped with a shift functor $[1]: \Dd\rightarrow \Dd$. For two objects $A, B \in \Dd$ let $\Hom^{\bullet}_{\Dd}(A,B)$ denote the graded $k$-vector space $\bigoplus_{i\in \Z}\Hom_{\Dd}(A,B[i])$. Let $\A \subset \Dd$ be a \emph{full triangulated subcategory}, that is, a full subcategory of $\Dd$ which is closed under shifts and cones of morphisms.

\begin{defn}
A full triangulated subcategory $\A$ of $\Dd$ is called \emph{right admissible} if the inclusion functor $\A \hookrightarrow \Dd$ has a right adjoint. Similarly, $\A$ is called \emph{left admissible} if the inclusion functor has a left adjoint. Finally, $\A$ is \emph{admissible} if it is both right and left admissible.
\end{defn}

\begin{defn}
The \emph{right orthogonal} of an admissible subcategory $\A\subset \Dd$ is the full subcategory $\A^\perp \subset \Dd$ of objects $B\in \Dd$ such that 
$\Hom_\Dd(A,B) = 0$ for all $A \in \A$. The \emph{left orthogonal} ${^\perp\A}$ is defined similarly. 
\end{defn}

If a full triangulated category $\A \subset \Dd$ is right admissible then every object $X\in \Dd$ fits into a distinguished triangle 

\[
	\dots \longrightarrow  Y \longrightarrow X \longrightarrow Z \longrightarrow Y[1] \longrightarrow \dots
\]
with $Y\in \A$ and $Z\in \A^\perp$, \cite[Remarks 1.43, (iii)]{Huyb}. One then says that there is a \emph{semiorthogonal decomposition} of $\Dd$ into the subcategories $(\A^\perp, \A)$. More generally, assume given a sequence of full triangulated subcategories $\A_1,\dots,\A_n \subset \Dd$. Denote by $\scal{\A_1,\dots,\A_n}$ the triangulated subcategory of $\Dd$ generated by $\A_1,\dots,\A_n$.

\begin{defn}\label{def:SOD}
	A sequence $(\A_1,\dots,\A_n)$ of full admissible subcategories of $\Dd$ is called \emph{semiorthogonal} if $\A_i \subset \A_j^\perp$ for $1\leq i < j\leq n$. The sequence $(\A_1,\dots,\A_n)$ is called a \emph{semiorthogonal decomposition} of $\Dd$ if via inclusion $\Dd$ is equivalent to the smallest full triangulated subcategory of $\Dd$ containing all of them.
\end{defn}

\begin{lemma}\cite[Lemma 1.61]{Huyb}\label{lem:Lemma 1.61_Huyb}
Any semiorthogonal sequence of full admissible triangulated subcategories $(\A_1,\dots,\A_n)\subset \Dd$ defines a semiorthogonal decomposition of $\Dd$, if and only if any object $A\in \Dd$ with $A\in \A_i^\perp$ for all $i=1,...,n$ is trivial.
\end{lemma}

\subsection{Mutations}

Let $\Dd$ be a triangulated category and assume $\Dd$ admits a semi\-orthogonal decomposition $\Dd = \scal{\A,\B}$.

\begin{defn}
The \emph{left mutation of $\B$ through $\A$} is defined to be $\Lb_\A(\B): = \A^\perp$. The \emph{right mutation of $\A$ through $\B$} is defined to be $\Rb_\B(\A): ={^\perp\B}$.
\end{defn}

In practice, computing left (resp., right) mutation of an admissible subcategory $\A$ of $\Dd$ through its right (resp., left) orthogonal $\A^\perp$ (resp., $^\perp\A$) amounts to the following \cite[Lemma 2.7]{Kuz}. Denote  $i: \A \to \Dd$ the embedding functor. Since the subcategory $\A$ is admissible there are left and right adjoint functors $\Dd \to \A$ which we denote by $i^*$ and $i^!$, respectively.  Given an object $F \in \Dd$, define the \emph{left mutation} $\Lb_\A(F)$ and the \emph{right mutation} $\Rb_\A(F)$ of $F$ through $\A$ by 
\[
	\Lb_\A(F): = \Cone(ii^!(F) \to F), \qquad \qquad \Rb_\A(F): = \Cone(F \to ii^*(F))[-1].
\]

These explicit formulae for mutations can be used to prove the following:

\begin{lemma}\cite[Corollary 2.9]{Kuz}
Let $\Dd$ be a triangulated category and assume that $\Dd =\langle \A _1,\A _2,\dots ,\A _n\rangle$ is a semiorthogonal decomposition of $\Dd$ as in Definition \ref{def:SOD}. Then for each $1\leq k\leq n-1$ there is a semiorthogonal decomposition
$$
\langle \A _1,\dots ,\A _{k-1},\Lb_{\A _k}(\A _{k+1}),\A _k,\A _{k+2},\dots \A _n\rangle
$$

and for each $2\leq k\leq n$ there is a semiorthogonal decomposition
$$
\langle \A _1,\dots ,\A _{k-2},\A _k,\Rb_{\A _k}(\A _{k-1}),\A _{k+1},\dots \A _n\rangle .
$$

\comment{
$\Dd = \scal{\A,\B}$, 
there are semiorthogonal decompositions $\Dd = \scal{\Lb_\A(\B), \A}$ and $\Dd = \scal{\B , \Rb_\B(\A)}$.
}
\end{lemma}

\begin{remark}\cite[Remark 2.8]{Kuz}
In case the triangulated subcategory $\A\subset \Dd$ is generated by an exceptional object $E$, there are explicit formulas for adjoint functors $i^!, i_{\ast}$ of the embedding functor $i : \A\rightarrow \Dd$. Specifically, there are distinguished triangles:
$$
\RHom(E,F)\otimes E\rightarrow F\rightarrow \Lb _E(F), \quad \Rb _E(F)\rightarrow F\rightarrow \RHom(F,E)^{\ast}\otimes E.
$$
\end{remark}

\subsection{Exceptional collections in geometric categories} 
We first introduce some notation. From now on, our ground field will be the field of complex numbers $\C$. Let $S$ be a smooth complex projective variety, and let $\Dd^b(S)$ denote the bounded derived category of coherent sheaves on $S$. Given a morphism $f : S \to T$ between two smooth varieties, we usually write $f_*,f^*$ for the corresponding derived functors of push-forwards and pull-backs between $\Dd ^b(S)$ and $\Dd ^b(T)$, and $\mathrm{R}^i f_*$ (resp., $\mathrm{L}^i f^*$) for the $i$-th cohomology of the functor $f_*$ (resp., $f^*$). We denote ${\mathcal R}{\mathcal Hom}_S$ the derived functor of the bi--functor of local ${\mathcal Hom}$--groups on $S$, and write ${\mathcal Ext}^i_S$ for its cohomology. We denote $\RHom _S$ the derived functor of the bi--functor of global $\Hom$--groups on $S$, and write $\Ext ^i_S$ for its cohomology. Given an object $A\in \Dd ^b(S)$, we denote $A^{\vee}:={\mathcal R}{\mathcal Hom}_S(A,\Oo _S)$ its dual.
Given two objects $A,B\in \Dd ^b(S)$, we also abbreviate the graded vector space $\Hom^{\bullet}_{\Dd ^b(S)}(A,B):=\bigoplus_{i\in \Z}\Hom_{\Dd ^b(S)}(A,B[i])$ as $\Hom _S^{\bullet}(A,B)$. 

We denote $\tau _{\leq i}$ and $\tau _{\geq i}$ the truncation functors with respect to the standard $t$--structure on $\Dd^b(S)$, and ${\mathcal H}^i:=\tau _{\leq i}\circ \tau _{\geq i}\circ [i]$ the $i$--th cohomology functor.

Convenient references for the definitions below are \cite[Section 1.4]{Huyb} and \cite{Ku}.

\begin{defn}
An object $E \in \Dd^b(S)$ is said to be \emph{exceptional} if there is an isomorphism of graded $\C$-algebras 
\[	
	\Hom_{S}^\bullet(E,E) = \C.
\]
\end{defn}

\begin{defn}
A collection of exceptional objects $(E_0,\dots,E_n)$ in $\Dd^b(S)$ is called 
\emph{exceptional} if for $0 \leq i < j \leq n$ one has
\[	
	\Hom_{S}^\bullet(E_j,E_i) = 0.
\]

An exceptional collection $(E_0,\dots,E_n)$ in $\Dd^b(S)$ is called \emph{full} if the smallest full triangulated subcategory containing all $E_i$ is $\Dd^b(S)$.
\end{defn}

\begin{lemma}\cite[Lemma 1.58]{Huyb}\label{lem:Lemma 1.58_Huyb}
Given an exceptional object $E$ in $\Dd^b(S)$, let $\langle E\rangle$ be the triangulated subcategory of $\Dd^b(S)$ generated by $E$. Then $\langle E\rangle\subset \Dd^b(S)$ is an admissible subcategory.
\end{lemma}

\begin{remark}\label{rem:O_exceptional}
Given a smooth Fano variety $S$, the structure sheaf $\Oo _S$ is an exceptional object by the Kodaira vanishing theorem. Hence, by the above lemma the subcategory $\langle \Oo _S\rangle \subset \Dd ^b(S)$ is an admissible subcategory. Likewise, any line bundle $\mathcal L$ on $S$ gives rise to an admissible subcategory $\langle \mathcal L\rangle$ of $\Dd ^b(S)$.
\end{remark}

Let $\Ll$ be an ample line bundle on $S$. 

\begin{defn}\cite[Definition 1.3]{KuzSm}
A \emph{Lefschetz collection with respect to $\Ll$} is an exceptional collection which has a block structure
\[	
\underbrace{E_0 , E_1, \dots, E_{\sigma_0}}_{\text{block 0}}; \,
\underbrace{E_0\otimes \Ll, \dots, E_{\sigma_1}\otimes \Ll}_{\text{block 1}}; \,
\dots; \,
\underbrace{E_0\otimes \Ll ^{\otimes (m-1)}, \dots, E_{\sigma_{m-1}}\otimes \Ll ^{\otimes (m-1)}}_{\text{block $m-1$}}
\]


where $\sigma = (\sigma_0 \geq \sigma_1 \geq \dots \sigma_{m-1} \geq 0)$ is a non--increasing sequence of non--negative integers. Semicolons are used to separate the blocks. The block $(E_0,E_1,\dots, E_{\sigma_0})$ is called the \emph{starting block}. If $\sigma_0=\sigma_1=\dots = \sigma_{m-1}$, then the corresponding Lefschetz collection is called \emph{rectangular}.
\end{defn}

Given a smooth variety $S$, let $\omega_S$ denote the canonical line bundle on $S$. 

\begin{lemma}\cite[Lemma 2.11]{Kuz}\label{lem:mutations_canonical_class}
For $S$ a smooth variety, assume given a semiorthogonal decomposition $\Dd ^b(S) = \scal{\A,\B}$. Then 
\[
	\Lb_\A(\B)=\B \otimes \omega_S \qquad \text{and} \qquad  \Rb_\A(\B)= \A \otimes \omega_S^{-1}.
\]
\end{lemma} 

\subsection{Exceptional collections on horospherical varieties of Picard rank one}\label{subsec:exc_coll_geom}

Some of the varieties in Pasquier's classification have been known to have full Lefschetz exceptional collections. Specifically, the variety $X^2$, which is isomorphic to a smooth hyperplane section of the orthogonal Grassmannian $\OG(5,10)$, has been treated in \cite[Section 6.2]{Ku}. The derived category of the variety $X^3(n,2)$, which is isomorphic to the odd symplectic Grassmannian $\IG(2,2n+1)$, has been computed in \cite{pech:q}. More recently, a full rectangular Lefschetz exceptional collection on $\IG(3,7)$ (i.e. on $X^3(3,3)$) has been constructed in \cite{Fon}. 

Below, in Theorem \ref{th:Lefexccoll}, we construct a full rectangular Lefschetz exceptional collection on the variety $X^5$.

Using the framework of Lefschetz semiorthogonal decompositions,
Kuznetsov and Smirnov have recently introduced in \cite{KuzSm} the notion of residual category. In particular, if a variety $X$ has a full rectangular Lefschetz exceptional collection, then the residual category of $X$ is trivial. 
Conjecture 1.12 of {\it loc.cit.} states that if the small quantum cohomology ring of a Fano variety of Picard rank one is generically semisimple, then its residual category should be of some prescribed form, and in the simplest possible case the residual category is trivial. Section \ref{subsec:4.5} 
establishes semisimplicity of the small quantum cohomology for the variety $X^5$ and together with Theorem \ref{th:Lefexccoll} this confirms the conjecture of \cite{KuzSm} for that case.


\section{A construction of exceptional vector bundles on horospherical varieties}
\label{section:exc-gen}

We start with two general statements concerning the descent of objects along a blow--up morphism.

\begin{prop}\label{prop:vecbun_blow-up_descent}
Let $j: S\subset T$ be a closed embedding of a smooth subvariety $S$ into a smooth variety $T$, and $\pi :\tildeT \to T$ be the blow-up of $S$ at $T$. Denote by $i : E\hookrightarrow \tildeT$ the embedding of the exceptional divisor, and let $p : E \to S$ be the projection. Assume given an object $\Ee \in \Dd^b(\tildeT)$. Then $\Ee$ is pulled back from $T$, i.e., $\Ee = \pi^*\Eb$ for some $\Eb \in \Dd^b(T)$, if and only if the restriction of $\Ee$ to $E$ is trivial along the fibers of $p$, i.e., $i^*\Ee \in p^*\Dd^b(S)$.
\end{prop}

\begin{proof}
Let $d$ be the codimension of $S$ in $T$. By~\cite{Or}, there is a semiorthogonal decomposition
\[
	\Dd^b(\tildeT) = \scal{i_*(p^*\Dd^b(S) \otimes \Oo_p(-d+1)), \dots, i_*(p^*\Dd^b(S) \otimes \Oo_p(-1)),\pi^*\Dd^b(T)}.
\]
  In other words, the objects of $\pi^*\Dd^b(T)$ are characterized by the fact that they are left orthogonal to the subcategory $\scal{i_*(p^*\Dd^b(S) \otimes \Oo_p(-d+1)), \dots, i_*(p^*\Dd^b(S) \otimes \Oo_p(-1)}$. Thus, $\Ee = \pi^* \Eb$ for some $\Eb \in \Dd^b(T)$ if and only if
\begin{equation}\label{eq:pullbackchar}
  	\Hom^{\bullet}_{\tildeT}(\Ee,i_*(p^*\Dd^b(S) \otimes \Oo_p(-k))) = \Hom^{\bullet}_E(i^*\Ee,p^*\Dd^b(S) \otimes \Oo_p(-k))=0,
\end{equation}

  for $k=1,\dots,d-1$.
  By \emph{loc.cit.}, $\Dd^b(E)$ has a semiorthogonal decomposition 
  \[
  	\Dd^b(E) = \scal{p^*\Dd^b(S) \otimes \Oo_p(-d+1), \dots, p^*\Dd^b(S) \otimes \Oo_p(-1), p^*\Dd^b(S)}.
  \]
  Thus, the equality (\ref{eq:pullbackchar}) holds if and only if $i^*\Ee \in p^*\Dd^b(S)$, proving the statement.
\end{proof}

We keep the notation and assumptions of Proposition \ref{prop:vecbun_blow-up_descent}.

\begin{prop}\label{prop:pullback_vecb}
Let $\Ee$ be an object of $\Dd^b(\tildeT)$, such that $\Ee = \pi^*\Eb$ where $\Eb \in \Dd^b(T)$. If $\Ee$ is a pure object, i.e. is a coherent sheaf on $\tildeT$, then so is the object $\Eb$ on $T$. If moreover $\Ee$ is locally free, then so is the coherent sheaf $\Eb$.
\end{prop}

\begin{proof}
Since $\pi$ is the blow--up, there is an isomorphism $\pi _{\ast}\Oo _{\tildeT}=\Oo _T$, and the functor $\pi ^{\ast}$ is full and faithful. In particular, $\pi ^{\ast}$ sends non--zero objects to non--zero objects. Applying the functor $\pi _{\ast}$ to 
$\pi^*\Eb$ and using projection formula, one obtains isomorphisms $\pi _{\ast} \pi^*\Eb=\Eb\otimes \pi _{\ast}\Oo _{\tildeT}=\Eb=\pi _{\ast}\Ee$.
One sees that $\Eb\in \Dd ^{\geq 0}(T)$, since in particular $\Ee\in \Dd ^{\geq 0}(\tildeT)$ and the functor $\pi _{\ast}$ is left $t$--exact.
Assuming that $\Eb$ had a non--trivial sheaf cohomology outside the degree zero, consider the distinguished triangle $\dots \rightarrow \tau _{<m}\Eb\rightarrow \Eb\rightarrow {\mathcal H}^{m}\Eb[-m]\rightarrow \dots$ associated with the canonical truncation, the positive integer $m$ being the maximal non--trivial cohomology degree of $\Eb$. Applying then the functor $\pi ^{\ast}$ to that triangle, one sees that $\pi ^{\ast}\Eb$ would also have a non--trivial sheaf cohomology on $\tildeT$ in the same degree, the functor $\pi ^{\ast}$ being right $t$--exact. In other words, there is an isomorphism ${\mathcal H}^{m}(\pi ^{\ast}\Eb)={\rm L}^0\pi ^{\ast}{\mathcal H}^{m}\Eb$, and the latter coherent sheaf is non--trivial. But the object $\pi ^{\ast}\Eb=\Ee$ being pure by assumption, it does not have non--trivial sheaf cohomology outside the degree zero. Therefore, $m=0$ and $\Eb$ is a coherent sheaf. 

Assuming that $\Ee$ is a locally free sheaf, the property of being locally free for $\Eb$ can be checked against the derived local ${\mathcal Hom}$--groups: a coherent sheaf $\Eb$ on $T$ is locally free if and only if ${\mathcal RHom}_{T}(\Eb,\Oo _T)$ is pure as an object of $\Dd ^b(T)$. Now $\pi ^{\ast}\Eb=\Ee$, and $\pi ^{\ast}{\mathcal RHom}_{T}(\Eb,\Oo _T)={\mathcal RHom}_{\tildeT}(\pi ^{\ast}\Eb,\pi ^{\ast}\Oo _T)={\mathcal RHom}_{\tildeT}(\Ee,\Oo _{\tildeT})$, while the latter object is pure, the sheaf $\Ee$ being locally free. We can repeat the same argument as above showing that ${\mathcal RHom}_{T}(\Eb,\Oo _T)$ is then a coherent sheaf, that is a pure object. This forces $\Eb$ to be locally free.
\end{proof}

\subsection{Bott's theorem}
Our main cohomological tool is Bott's theorem. Consider a weight $\lambda \in X(T)$ and let $\Ll_\lambda$ be the corresponding line bundle on $G/B$. The weight $\lambda$ is called \emph{singular} if it lies on a wall of some Weyl chamber defined by $\scal{-,\alpha^\vee} = 0$ for some coroot $\alpha^\vee \in \Ri^\vee$. Weights which are not singular are called \emph{regular}. The Weyl group $W = N_G(T)/T$ acts on $\XX(T)$ via the dot--action: if $w\in W$ and $\lambda \in \XX(T)$, then $w\cdot \lambda = w(\lambda + \rho) - \rho$, where $\rho$ is the sum of fundamental weights. 

\begin{thm}\cite[Theorem 2]{Dem}\label{th:Bott-Demazure_th}

\begin{enumerate}
	\item If $\lambda+\rho$ is singular, then $\Hr^i(G/B,\Ll_\lambda) = 0$ for all $i$.
	\item If $\lambda+\rho$  is regular and dominant, then $\Hr^i(G/B,\Ll_\lambda) = 0$ for $i>0$.
	\item If $\lambda+\rho$  is regular, then $\Hr^i(G/B,\Ll_\lambda)\neq 0$ for a unique degree $i$ which coincides with $\ell(w)$, the length of a Weyl group element $w$ that takes $\chi$ to the dominant chamber, i.e. $w \cdot \lambda \in X_+(T)$. The cohomology group $\Hr^{\ell(w)}(G/B,\Ll_\lambda)$ is the irreducible $G$-module $V(w \cdot \lambda)$ of highest weight $w \cdot \lambda$.
\end{enumerate}
\end{thm}

\subsection{The bundles $\Ft_Y$ and $\Ft_Z$}\label{sec:Hecke_transf}
We refer to Section~\ref{ss:blowups} for the notation used in this section.

Let $X$ be a horospherical variety from Pasquier's list. To construct exceptional objects on such an $X$, we start with appropriate exceptional objects on $\Xt_{YZ}$ (the blow--up of of $Y\cup Z$ in $X$), and modify\footnote{This can be thought of as a higher--dimensional analogue of the classical Hecke modifications on curves.}
 it across the divisors $E_Z$ and $E_Y$ in a way that the modified bundles be trivial when restricted to the fibers of $p$ and $q$. Proposition \ref{prop:vecbun_blow-up_descent} ensures that in this case the modified objects are pulled back from $X$. If on top of that an exceptional object we started with is an exceptional vector bundle, Proposition  \ref{prop:pullback_vecb} ensures that the modified coherent sheaf on $\Xt_{YZ}$ is the pull--back of an exceptional vector bundle on $X$.

 Recall that both maps $p : E_Z \to Y$ and $q : E_Y \to Z$ are projective bundles associated to vector bundles $N_Y = N_{Y/X}$ and $N_Z = N_{Z/X}$ on $Y$ and $Z$, respectively. We set $F_Y = N_{Y/X}^\vee = p_*\L_{\oZ - \oY}$ and $F_Z = N_{Z/X}^\vee = q_*\L_{\oY - \oZ}$, where $\oY$ and $\oZ$ are the fundamental weights associated to $Y$ and $Z$, respectively. Let $\cO_p(1)$ (resp., $\cO_q(1)$) denote the line bundle of relative degree one along projection $p$ (resp., along projection $q$). Note that we also have $\cO_p(1) = \L_{\oZ - \oY}$ and $\cO_q(1) = \L_{\oY - \oZ}$. In particular, we have the relative Euler sequences on $E_Z$

\begin{equation}\label{eq:Eulerseq_Y}
	0 \to \Omega^1_{E_Z/Y} \otimes \cO_p(1) \to p^*F_Y \to \cO_p(1) \to 0
\end{equation}

and on $E_Y$

\[	 
	 0 \to \Omega^1_{E_Y/Z} \otimes \cO_q(1)  \to q^*F_Z \to \cO_q(1) \to 0. 
\]

Recall that the composed morphisms $\xi \circ j_Z: E_Z \to E$ and $\xi \circ j_Y: E_Y \to E$ are both identity morphisms $\id_E: E\to E$ (and both divisors $E_Z$ and $E_Y$ are isomorphic to $E=G/(P_Y\cap P_Z)$).
Thus, $j_Z^* \xi^* = j_Y^*\xi^* = \id_E^*$, and the above map $p^*F_Y \to \cO_p(1)$ (resp., $q^*F_Z \to \cO_q(1)$) defines by adjunction a map of coherent sheaves $\xi^* p^*F_Y \to {j_Z}_*\cO_p(1)$ (resp., $\xi^* q^* F_Z \to {j_Y}_* \cO_q(1)$). We define coherent sheaves $\Ft_Y$ and $\Ft_Z$ 
on $\Xt_{YZ}$ via the exact sequences:

\begin{equation}\label{eq:seqdef_tilde_F_Y}
	0 \to \Ft_Y \to \xi^* p^*F_Y \to {j_Z}_*\cO_p(1) \to 0
\end{equation}
	
and
\[			 
	0 \to \Ft_Z \to \xi^* q^*F_Z \to {j_Y}_*\cO_q(1) \to 0.
\]

\begin{prop}\label{fact-ll}
Both $\Ft_Y$ and $\Ft_Z$ are locally free sheaves.
\end{prop}

\begin{proof}
Given a coherent sheaf $\Ee$ on a smooth variety $S$, we have, as in the proof of Proposition \ref{prop:pullback_vecb}, that $\Ee$ is locally free if and only if ${\mathcal Ext}_S^i(\Ee,\Oo_S)=0$ for $i>0$. The statement then follows from applying ${\mathcal Hom}_{\Xt_{YZ}}(-,\Oo_{\Xt_{YZ}})$ to the above sequences and taking into account that $j_Z$ and $j_Y$ are divisorial embeddings.
\end{proof}

\begin{prop}\label{prop:desc}
For all the varieties of Pasquier's list except the variety $X^2$, the bundle $\Ft_Y$ is the pull-back of a vector bundle on $X$.
\end{prop}

\begin{proof}
Proposition \ref{prop:pullback_vecb} ensures that if the vector bundle $\Ft_Y$ is the pull-back of an object of $\Dd^b(X)$, then that object is also a vector bundle on $X$. By Proposition \ref{prop:vecbun_blow-up_descent}, in order to prove that $\Ft_Y$ is pulled back from $X$, 
it is sufficient to check that $j_Y^* \Ft_Y$ and $j_Z^* \Ft_Y$ are trivial on the fibers of $p$ and $q$, respectively. The first statement is clear: since $Y$ and $Z$ are disjoint we have $j_Y^* \Ft_Y = j_Y^* \xi^* p^* F_Y = p^* F_Y$. We now compute $j_Z^* \Ft_Y$. To this end, let us restrict sequence (\ref{eq:seqdef_tilde_F_Y}) to $E_Z$, i.e. apply  $j_Z^{\ast}$. We compute $\textrm{Tor}_1^{\Xt_{YZ}}({j_Z}_*\cO_p(1),{j_Z}_*\cO_{E_Z}) = {j_Z}_*(\cO_{E_Z}(-E_Z) \otimes \cO_p(1)) = {j_Z}_*(\L_{\oY-\oZ} \otimes \cO_p(1)) = {j_Z}_*\cO_{E_Z}$ and obtain the exact sequence, the sheaves $p^{\ast}F_Y$ and $\Ft_Y$ being locally free:
  \[
  	0 \to \cO_{E_Z}\to j_Z^* \Ft_Y \to p^* F_Y \to \cO_p(1) \to 0.
  \]

Breaking up this exact sequence into short exact sequences and comparing them to (\ref{eq:Eulerseq_Y}), we obtain the short exact sequence 

\begin{equation}\label{eq:tildeF_restricted}
0 \to \cO_{E_Z}\to j_Z^* \Ft_Y \to \Omega^1_{E_Z/Y} \otimes \cO_p(1) \to 0.
\end{equation}
  
It is sufficient to prove that $\Omega^1_{E_Z/Y} \otimes \cO_p(1)$ is the pullback of a vector bundle on $Z$. There is an obvious necessary condition given by the top Chern class, namely, that $\Lambda^\mathrm{top}(\Omega^1_{E/Y} \otimes \cO_p(1)) =  \Lambda^\mathrm{top} p^* F_Y \otimes \cO_p(1)^{\otimes(-1)}$ should be pulled back from $Z$. Since $F_Y = N_{Y/X}^\vee$, we can compute
\[
	\Lambda^\mathrm{top} p^* F_Y = p^* \cO_Y(1)^{\otimes(c_1(Y)-c_1(X))} = \L_{(c_1(Y) - c_1(X))\oY}
\]	
and therefore
\[
	\Lambda^\mathrm{top}(\Omega^1_{E_Z/Y} \otimes \cO_p(1)) = \L_{(c_1(Y)+1-c_1(X))\oY - \oZ}.
\]
  Now if $X$ is not isomorphic to $X^2$ we have $c_1(X) = c_1(Y) + 1$, so this necessary condition is satisfied (this also proves that $\Ft_Y$ is never pulled back from $X$ if $X$ is isomorphic to $X^2$). If furthermore $\Omega^1_{E_Z/Y} \otimes \cO_p(1)$ is of rank one, the above condition on the top Chern class is sufficient (since $\Lambda^\mathrm{top}(\Omega^1_{E_Z/Y} \otimes \cO_p(1)) = \Omega^1_{E_Z/Y} \otimes \cO_p(1)$). This occurs if and only if $\rk(\Omega^1_{E_Z/Y} \otimes \cO_p(1)) = \codim_X Y - 1 = 1$, therefore for $\codim_X Y = 2$. This is true for $X^1(n)$ and $X^5$. We are left with $X^3(n,m)$ and $X^4$. For $X=X^3(n,m)$, using the description of $X$ as an odd symplectic Grassmannian we will easily check that $p^* F_Y$ and $\Omega^1_{E_Z/Y} \otimes \cO_p(1)$ are pullbacks of the tautological subbundles in $Y$ and $Z$, respectively, proving the result. We show this using representations. Indeed, both $p^*F_Y$ and $\L_{\oZ-\oY}$ are homogeneous bundles on $E_Z$ coming from representations of $P_Y \cap P_Z$. The bundle $p^* F_Y$ comes from the $P_Y$--representation $V_{P_Y}(\oZ-\oY)$ while $\L_{\oZ-\oY}$ comes from the one--dimensional $P_Y\cap P_Z$-representation $\C_{\oZ-\oY}$ of weight $\oZ-\oY$. The exact sequence comes from an exact sequence of representations
\[
	0 \to M \to V_{P_Y}(\oZ-\oY) \to \C_{\oZ-\oY} \to 0.
\]
We only need to check that $M$ is a $P_Z$--representation (and not only a $P_Y \cap P_Z$--representation). For this we only need to define the action of $\g_{\a_Y}$, where ${\a_Y}$ is the simple root associated to $Y$. The only possible action is the trivial action, which is compatible with the $P_Y \cap P_Z$--action if and only if the cocharacter $\a_Y^\vee$ acts trivially. This is an easy check if $X=X^3(n,m)$ or $X=X^4$.
\end{proof}

\begin{remark}
The above necessary condition on top Chern classes proves that $\Ft_Z$ is never pulled back from $X$ with one possible exception being the case of $X=X^3(n,m)$ (in which $c_1(X) = c_1(Z) + 1$). An easy computation with weights shows that even in that case $\Ft_Z$ is not pulled back from $X$.
\end{remark}

\begin{defn}
Assume $X$ is not isomorphic to $X^2$. We denote by $\cF_Y$ the vector bundle on $X$ such that $\pi_{YZ}^*\cF_Y = \Ft_Y$.
\end{defn}

\subsection{The bundle $\cF_Y$ is exceptional}

Assume as above that $X\neq X^2$.

\begin{prop}\label{prop:exc}
  The bundles $\Ft_Y$ and $\cF_Y$ are exceptional.
\end{prop}

\begin{proof}
Since $X$ is smooth and $\pi _{YZ}$ is the blow--up morphism, we have $\pi_{YZ*} \cO_{\Xt_{YZ}} = \cO_X$. Thus, the pull--back functor $\pi_{YZ}^{\ast}:\Dd ^b(X)\rightarrow \Dd ^b(\Xt_{YZ})$ is fully faithful. In what follows, we will be using this fact throughout when computing $\Ext$--groups: given a blow--up morphism $\pi: T\rightarrow S$, for any two objects ${\mathcal E}_1,{\mathcal E}_2$ of $\Dd ^b(S)$ there is an isomorphism $\Hom _S^{\bullet}({\mathcal E}_1,{\mathcal E}_2)=\Hom _T^{\bullet}(\pi ^{\ast}{\mathcal E}_1,\pi ^{\ast}{\mathcal E}_2)$. 

In particular, it is sufficient to prove that the bundle $\Ft_Y$ on $\Xt_{YZ}$ is exceptional. Let us apply $\Hom _{\Xt_{YZ}}(-,\Ft_Y)$ to sequence (\ref{eq:seqdef_tilde_F_Y}). 
We need to compute $\Ext^i _{\Xt_{YZ}}(\xi^* p^* F_Y,\Ft_Y)$ and $\Ext^i _{\Xt_{YZ}}({j_Z}_* \cO_p(1),\Ft_Y)$. First, by adjunction, we have $\Ext^i _{\Xt_{YZ}}(\xi^* p^* F_Y,\Ft_Y)=\Ext^i _{E_Z}(p^* F_Y,\xi_* \Ft_Y)$. Applying $\xi_*$ to (\ref{eq:seqdef_tilde_F_Y}), we see that $\xi_* \Ft_Y = \Omega^1_{E_Z/Y}\otimes \cO_p(1)$. To compute $\Ext^i_{E_Z}(p^* F_Y,\Omega^1_{E_Z/Y}\otimes \cO_p(1))$, we thus need to compute $\Ext^i_{E_Z}(p^* F_Y,p^* F_Y)$ and $\Ext^i_{E_Z}(p^* F_Y,\cO_p(1))$. But $F_Y = p_*\cO_p(1)$ and $p_{\ast}\Oo _{E_Z}=\Oo _Y$, therefore the first group is isomorphic to $\Ext^i_Y(F_Y,F_Y)$ by the remark above and, by adjunction of $p^*$ and $p_*$, the second group is also isomorphic to $\Ext^i_Y(F_Y,F_Y)$. This proves that $\Ext^i _{\Xt_{YZ}}(\xi^* p^* F_Y,\Ft_Y) = 0$ for any $i$.

We have $\Ext^i_{\Xt_{YZ}}({j_Z}_* \cO_p(1),\Ft_Y) = \Ext^i_{E_Z}(\cO_p(1),j_Z^! \Ft_Y)$ where $j_Z^!$ is a right adjoint functor to ${j_Z}_*$. For the smooth divisor $E_Z \subset X$, there is an isomorphism of functors $j_Z^!(-) = j_Z^*(-) \otimes N_{E/X}[-1] = j_Z^*(-) \otimes \L_{\oZ- \oY}[-1]$ giving isomorphisms $\Ext^i_{\Xt_{YZ}}({j_Z}_* \cO_p(1),\Ft_Y)\! =\! \Ext^{i-1}_{E_Z}(\L_{\oZ-\oY},j_Z^* \Ft_Y \otimes \L_{\oZ-\oY}) = \Ext^{i-1}_{E_Z}(\cO _{E_Z},j_Z^* \Ft_Y)$. Using sequence (\ref{eq:tildeF_restricted}), it remains to compute $\Ext^{i-1}_{E_Z}(\cO_{E_Z},\cO_{E_Z})$ and $\Ext^{i-1}_{E_Z}(\cO_{E_Z},\Omega^1_{E_Z/Y} \otimes \cO_p(1))$.

The former $\Ext$--group is isomorphic to ${\rm H}^{i-1}(E_Z,\Oo _{E_Z})$, which is isomorphic to $\C$ for $i=1$ and is trivial otherwise. To compute the latter $\Ext$--group, we need to consider $\Ext^{i-1}_{E_Z}(\cO_{E_Z},p^* F_Y)={\rm H}^{i-1}(Y,F_Y)$ and $\Ext^{i-1}_{E_Z}(\cO_{E_Z},\cO_p(1))$. Since $F_Y = p_* \cO_p(1)$, both groups are isomorphich to each other, and since $\cO_p(1) = \L_{\oZ-\oY}$ these groups are trivial for all $i$ by Theorem \ref{th:Bott-Demazure_th}.
\end{proof}

\subsection{An exceptional pair of vector bundles} 

Assume as above that $X$ is not isomorphic to $X^2$. We now construct a two--term exceptional sequence (an exceptional pair) of vector bundles on $X$.

\begin{prop}\label{prop:esq}
  The sequence $\scal{\cF_Y,\cO_X}$ is exceptional on $X$.
\end{prop}

\begin{proof}
By Remark \ref{rem:O_exceptional}, the sheaf $\cO_X$ is exceptional, and by the previous proposition so is the bundle $\cF_Y$. We only need to check the vanishing of the groups $\Ext^i_X(\cO_X,\cF_Y)$ for all $i$. This is equivalent to the vanishing of $\Ext^i_{\Xt_{YZ}}(\cO_{\Xt_{YZ}},\Ft_Y)$ for all $i$. Using sequence (\ref{eq:seqdef_tilde_F_Y}), we see that it is equivalent to the vanishing of
$\Ext^i_{\Xt_{YZ}}(\cO_{\Xt_{YZ}},\xi^* p^* F_Y)$ and $\Ext^i_{\Xt_{YZ}}(\cO_{\Xt_{YZ}},{j_Z}_* \cO_p(1))$ for $i\geq 0$.

\quad Since $p_{\ast}\xi_{\ast}\cO_{\Xt_Z}=\Oo _Y$, we obtain $\Ext^i_{\Xt_{YZ}}(\cO_{\Xt_{YZ}},\xi^* p^* F_Y)$=${\rm H}^i(\Xt_{YZ},\xi^* p^* F_Y)$ = ${\rm H}^i(Y,F_Y)$, and by adjunction there are isomorphisms $\Ext^i_{\Xt_{YZ}}(\cO_{\Xt_{YZ}},{j_Z}_* \cO_p(1))$ = $\Ext^i_{E_Z}(j_Z^* \cO_{\Xt_Z},\cO_p(1))$ = $\Ext^i_{E_Z}(\cO_{E_Z},\cO_p(1))$. But in the previous proposition we have shown that both groups in question vanish, hence the statement.
\end{proof}

\section{A full rectangular Lefschetz exceptional collection on the $G_2$--variety}\label{sec:g2}

In this section, we construct, extending the results of previous 
Section \ref{section:exc-gen}, a full rectangular Lefschetz exceptional collection on the variety $X^5$.

Let $G_2$ be the exceptional group of rank two.
We denote by $\alpha$ and $\beta$ denote the two simple roots in $R_{+}$, the root $\beta$ being the long root.


\subsection{The flag variety of $G_2$}

The group $G_2$ has two standard parabolic subgroups $\Pa$ and $\Pb$ which correspond to the simple roots $\alpha$ and $\beta$. Let $\oa, \ob \in X(T)$ be the two fundamental weights. Recall that given a dominant weight $\lambda \in X(T)$, we denote $V(\lambda)$ the irreducible representation with highest weight $\lambda$. The homogeneous spaces $G/\Pa$ and  $G/\Pb$ are isomorphic respectively to the 5-dimensional quadric $\Qs_5\subset \Pp (V(\oa))$ (the projectivization of $G_2$--orbit of the lowest weight vector in the irreducible $G_2$--representation $V(\oa)$)\footnote{In the notation concerning the Borel subgroups that we have chosen, the grassmannian $\Pa$ (resp., $G/\Pb$) is embedded into $\Pp (V(\oa)^{\ast})$ (resp., into $\Pp (V(\ob)^{\ast})$ via the very ample line bundle $\oa$ (resp., $\ob$). Still, both $V(\oa)$ and $V(\ob)$ are self--dual as $G_2$--modules, so we skip the dualization superscript at $V(\oa)$ and $V(\ob)$.}, and to the 5-dimensional variety $\Gad \subset \Pp (V(\ob))$ (the projectivization of $G_2$--orbit of the lowest weight vector in $V(\ob)$, the adjoint representation of $G_2$). The Levi subgroups of $\Pa$ and $\Pb$ have a component isomorphic to $\SL_2$; its tautological representation in each case gives rise to a homogeneous rank 2 vector bundle on $G/\Pa$ (resp., on $G/\Pb$). Denote by $\pa$ and $\pb$ the two projections of $G_2/B$ onto $\Qs_5 \subset \Pp (V(\oa))$ and $\Gad \subset \Pp (V(\ob))$, respectively. By \cite[Lemma 8.3]{Ku}, the projection $\pa$ is the projective bundle associated to the stable indecomposable rank two bundle $\Kk$ on $\Qs_5$ with $\det \Kk = \Ll_{-3\oa}$, and the projection $\pb$ is the projective bundle over $\Gad$ associated to the tautological rank two vector bundle $\Uu_2$ over $\Gad \subset \Gr(2,V(\oa))$. The latter embedding is related to an isomorphism of $G_2$--modules $\Lambda ^2V(\oa)=V(\oa)\oplus V(\ob)$: the grassmannian ${\rm Gr}(2,V(\oa))$ is embedded into 
${\mathbb P}(\Lambda ^2V(\oa))$ via the Pl\"ucker embedding, and the embedding $\Gad \subset \Pp (V(\ob))$ is obtained via restriction along 
the natural embedding $\Pp (V(\ob))\subset \Pp(\Lambda ^2V(\oa))$
induced by the above direct sum decomposition.

By \cite{ottaviani}, there is a short exact sequence 
\[
0 \to \Ss \to W \otimes \Oo_{\Qs_5} \to \Ss^\vee \to 0,
\]
where $\Ss$ is the spinor bundle on $\Qs_5$, and $W$ is the spinor representation regarded as a representation of $G_2$ via the restriction $G_2 \subset \mathrm{Spin}_7$. Next, there is a short exact sequence 
\begin{equation}\label{eq:relEulerseqforK}
	0 \to \L_{-\ob} \to  \pa^* \Kk \to \Ll_{\beta-\ob} \to 0.
\end{equation}
We have $\Ss \otimes \Ll_{\oa} = \Ss^\vee$. Denote by $\Psi^{\oa}_1$ the vector bundle on $\Qs_5$ which is the pullback of $\Omega_{\Pp (V(\oa))}^1(1)$ along the embedding $\Qs_5 \subset \Pp (V(\oa))$. The bundles $\Ss$ and $\Kk$ are related via the short exact sequence
\begin{equation}\label{eq:seqforK}
	0 \to \Kk \to (\Psi_1^{\oa})^\vee \otimes \Ll_{-\oa} \to \Ss\to 0,
\end{equation}
see~\cite[Appendix B]{Ku}.

The weights of $V(\oa)$ are $\oa,-\oa+\ob,2\oa-\ob,0,-2\oa+\ob,\oa-\ob,-\oa$, while we deduce from~\eqref{eq:relEulerseqforK} that the weights of $\Kk$ are $-\ob,\ob-3\oa$. Thus the weights of $\Ss$ are $0,-2\oa+\ob,\oa-\ob,-\oa$. We have the tautological short exact sequence sequence on $\Gad$:
\begin{equation}\label{eq:tautological_seq_U}
	0 \to \Uu_2 \to V(\oa) \otimes \Oo_{\Gad} \to V(\oa)/\Uu_2 \to 0.
\end{equation}
Define $\Uu_2^\perp: = (V(\oa)/\Uu_2)^\vee$. There is a short exact sequence for the pullback of $\Uu_2$ along the projection $\pb: G_2/B\rightarrow \Gad$:
\begin{equation}\label{eq:seqforU}
	0 \to \Ll_{-\oa} \to \pb^* \Uu_2 \to \Ll_{\oa-\ob} \to 0.
\end{equation}

\subsection{Horospherical $G_2$--variety}
In what follows, $X$ denotes the horospherical variety $X^5$. Recall that $X$ is embedded into $\Pp (V(\oa) \oplus V(\ob))$ and recall the blow-up construction from Subsection~\ref{ss:blowups}. In this section we will use the following notation which is better suited to the case of our study:
\[
	\xymatrix{
		E_1\ar@{^{(}->}[r]^{i_1}\ar@<-0.1ex>[d]^{\pa}&  \Xt \ar@<-0.4ex>[d]^\pi & E_2 \ar@{_{(}->}[l]_{i_2}\ar@<-0.4ex>[d]_{\pb}\\
		\Qs_5\ar@{^{(}->}[r]\ar[r]^{j_1} & X & \ar@{_{(}->}[l]_{j_2} \Gad
 	}
\]
In the above diagram, we have $Y = \Gad$, $Z = \Qs_5$. The middle map $\pi$ is the blow-up of $X$ at $\Qs_5 \cup \Gad$. The exceptional divisors over $\Qs_5$ and $\Gad$ are denoted $E_1$ and $E_2$, respectively. The maps $\pi_\a$ and $\pb$ are, respectively, the projections $q$ and $p$ from the incidence variety $E=E_1=E_2$ to $Z$ and $Y$. In our case, the incidence variety $E$ is isomorphic to the flag variety $G_2/B$, and the blow--up $\Xt$ has a $\Pp ^1$--bundle structure over $G_2/B$, namely $\Xt = \Pp_{G_2/B}(\Ll_{-\oa} \oplus \Ll_{-\ob})$. The exceptional divisors $E_1$ and $E_2$ are embedded linearly into $\Xt$ with respect to this $\Pp ^1$--bundle structure, and we can identify them as $E_1 = \Pp_{G_2/B}(\Ll_{-\oa})=G_2/B$ and $E_2 = \Pp_{G_2/B}(\Ll_{-\ob})=G_2/B$. Denote by $\xi$ the projection of $\Xt$ onto $G_2/B$.

\subsection{Line bundles on $\Xt$}

There are relations among line bundles on $\Xt$. Firstly, the index of $X$ being equal to 4, the canonical bundle satisfies $\omega_X = \Oo_X(-4)$. Thus, $\omega_\Xt = \pi^* \Oo_X(-4) \otimes \Oo_\Xt(E_1) \otimes \Oo_\Xt(E_2)$. The adjunction formula gives
\[
	\omega_{E_1} = (\omega_\Xt \otimes \Oo_\Xt(E_1)) \otimes \Oo_{E_1} = \Ll_{-4\oa}\otimes \Oo_{E_1}(2E_1).
\]
Indeed, $\Oo_{E_1}(E_2) = \Oo_{E_1}$ (resp., $\Oo_{E_2}(E_1) = \Oo_{E_2}$) since intersection of $E_1$ and $E_2$ is empty. On the other hand, $\omega_{E_1}=\Ll_{-2\rho}$, where $\rho := 5\alpha+3\beta$ is the half-sum of all positive roots, implying that $\Oo_{E_1}(E_1) = N_{E_1/\Xt}=\Ll_{\oa-\ob}$.

Similarly, the normal bundle $\Oo_{E_2}(E_2) = N_{E_2/\Xt}$ is isomorphic to $\Ll_{\ob-\oa}$, via the isomorphisms $\omega_{E_2}=(\omega_\Xt \otimes \Oo_\Xt(E_2))\otimes \Oo_{E_2} = \Ll_{-4\ob} \otimes \Oo_{E_2}(2E_2)$ and $\omega_{E_2} =\Ll_{-2\rho}$, implying that $\Oo_{E_2}(E_2) = N_{E_2/\Xt} = \Ll_{\ob-\oa}$. Finally, recall that $\xi: \Xt \to G_2/B$ is the $\Pp^1$-bundle associated to $\Ll_{-\oa} \oplus \Ll_{-\ob}$. This gives $\omega_\Xt = \cO_\xi(-2) \otimes \xi^* \omega_{G_2/B} \otimes \xi^* \Ll_{\oa+\ob} = \cO_\xi(-2) \otimes \xi^* \Ll_{-\rho}$. Pulling back $\cO_X(1)$ along the blow-up of $\Qs_5$ or the blow-up of $\Gad$, and using the above identifications of normal bundles $\Oo_{E_1}(E_1)$ and $\Oo_{E_2}(E_2)$, we get $\pi^* \cO_X(1) = \cO_\Xt(E_1) \otimes \xi^* \Ll_{\ob} = \cO_\Xt(E_2) \otimes \xi^*\Ll_{\oa}$. Comparing the two expressions for $\omega_\Xt$, we obtain an isomorphism $\Oo_\xi(-1) = \pi^* \Oo_X(-1)$, and
$\xi_* \pi^* \Oo_X(1) = \xi_*\Oo_\xi(1)=\Ll_{\oa} \oplus \Ll_{\ob}$.

Consider the distinguished triangles for the divisors $E_1$ and $E_2$:
\[
	\dots \to \Oo_\Xt(-E_1) \to \Oo_\Xt \to {i_1}_*\Oo_{E_1} \to \Oo_\Xt(-E_1)[1] \to \dots ,
\]
\[
	\dots \to \Oo_\Xt(-E_2) \to \Oo_\Xt \to {i_2}_*\Oo_{E_2} \to \Oo_\Xt(-E_2)[1] \to \dots 
\]

\begin{prop}\label{prop:rel_Euler_seq_rho}
There is a distinguished triangle in $\Dd^b(\Xt)$:
\begin{equation}\label{eq:rho-decomp_of_O(E_1)}
	\dots \to \xi^* \Ll_{\oa} \otimes \Oo_\xi(-1) \to \Oo_\Xt \oplus \xi^* \Ll_{\oa-\ob} \to \Oo_\Xt(E_1) \to \dots 
\end{equation}
\end{prop}

\begin{proof}
Consider a natural adjunction morphism $\xi^* \xi_* \Oo_\Xt(E_1) \to \Oo_\Xt(E_1)$. From the discussion above we know that $\cO_\Xt(E_1) = \pi^*\cO_X(1) \otimes \xi^* \Ll_{-\ob} = \cO_\xi(1) \otimes \xi^* \Ll_{-\ob}$. This gives
  \begin{equation}\label{eq:twisted_normal_seq_i_1}
    \xi_* \Oo_\Xt(E_1) = \Oo_{G_2/B} \oplus \Ll_{\oa-\ob}.
  \end{equation}
\qquad By \cite{Or}, there is a semiorthogonal decomposition $\Dd^b(\Xt)$ = $ \scal{\xi^* \Dd^b(G_2/B) \otimes \Oo_\xi(-1),\xi^* \Dd^b(G_2/B)}$. Using this, we can identify the cone of the natural morphism $\xi^* \xi_* \Oo_\Xt(E_1) \to \Oo_\Xt(E_1)$ with $\xi^*(?) \otimes \Oo_\xi(-1)$ where $?$ is an object of $\Dd^b(G_2/B)$:
	\[
		\dots \to \Oo_\Xt \oplus \xi^* \Ll_{\oa-\ob} \to \Oo_\Xt(E_1) \to \xi^* (?)\otimes \Oo_\xi(-1) \to \dots
	\]
To compute the object $?$ in the above triangle, we tensor it with $\Oo_\xi(-1)$ and apply $\xi_*$; we obtain an isomorphism $? \otimes \Ll_{-\rho}[-1] = \xi_*(\Oo_\Xt(E_1)\otimes \Oo_\xi(-1))$. Tensoring equation~\eqref{eq:twisted_normal_seq_i_1} with $\Oo_\xi(-1)$ and applying $\xi_*$, we obtain $\xi_* (\Oo_\Xt(E_1) \otimes \Oo_\xi(-1)) = \Ll_{-\ob}$. Finally, we get an isomorphism $? = \Ll_{\oa}[1]$, arriving at distinguished triangle ~\eqref{eq:rho-decomp_of_O(E_1)}.
\end{proof}

We see that triangle~\eqref{eq:rho-decomp_of_O(E_1)} is induced by the short exact sequence 
\begin{equation}\label{eq:Euler_seq_wrt_rho}
	0\to \xi^* \Ll_{\oa} \otimes \Oo_\xi(-1) \to \Oo_\Xt \oplus \xi^* \Ll_{\oa-\ob}\to \Oo_\Xt(E_1)\to 0.
\end{equation}
Since $\xi^* \Ll_{\oa} \otimes \Oo_\xi(-1) = \Oo_\Xt(-E_2)$, the sequence
\eqref{eq:Euler_seq_wrt_rho} is isomorphic to 
\[
	0 \to \Oo_\Xt(-E_2)\to \Oo_\Xt \oplus \xi^* \Ll_{\oa-\ob} \to \Oo_\Xt(E_1) \to 0.
\]
Observe that since $\Oo_\Xt(E_1) = \xi^* \Ll_{-\ob} \otimes \Oo_\xi(1)$, the above extension corresponds to a unique non-trivial extension in the group $\Ext^1_{G_2/B}(\xi^* \Ll_{-\ob} \otimes \Oo_\xi(1),\xi^* \Ll_{\oa} \otimes \Oo_\xi(-1)) = \Hr^1(G_2/B,\Ll_{\rho} \otimes \Ll_{-\rho}[-1]) = \Hr^0(G_2/B,\Oo_{G_2/B}) = \C$. By duality, we obtain:
\[
	0 \to \Oo_\Xt(-E_1) \to \Oo_\Xt \oplus \xi^* \Ll_{\ob-\oa} \to \Oo_\Xt(E_2) \to 0,
\]
with $(\Oo_\Xt \oplus \xi^* \Ll_{\oa-\ob}) = (\Oo_\Xt \oplus \xi^* \Ll_{\oa-\ob})^\vee \otimes \xi^* \Ll_{\oa-\ob}$.

Using the isomorphisms $\pi_* \Oo_\Xt(E_1) = \pi_* \Oo_\Xt(E_2) = \pi_* \Oo_\Xt(E_1+E_2) = \Oo_X$, we can sum up the results as follows:

\begin{lemma}\label{lem:E_1-E_2-O(rho)-lemma}
We have the following isomorphisms:
\[
	\Oo_\Xt(E_1) = \xi^* \Ll_{-\ob} \otimes \pi^* \Oo_X(1) \text{ and } \Oo_\Xt(E_2) = \xi^* \Ll_{-\oa} \otimes \pi^*\Oo_X(1);
\]
\[
	\pi_* \xi^* \Ll_{-\ob} = \Oo_X(-1) = \pi_* \xi^* \Ll_{-\oa} \text{ and } \quad \pi_* \xi^* \Ll_{-\rho} = \Oo_X(-2).
\]
There are also short exact sequences:
\[
	0\to \xi^* \Ll_{\oa} \otimes \pi^* \Oo_X(-1) \to \Oo_\Xt \oplus \xi^* \Ll_{\oa-\ob} \to \xi^* \Ll_{-\ob} \otimes \pi^* \Oo_X(1) \to 0,
\]
\[
	0 \to \xi^* \Ll_{\ob} \otimes \pi^* \Oo_X(-1) \to \Oo_\Xt \oplus \xi^* \Ll_{\ob-\oa} \to \xi^* \Ll_{-\oa} \otimes \pi^* \Oo_X(1) \to 0.
\]
\end{lemma}

\subsection{Vector bundles on $X$}
In order to construct exceptional vector bundles on $X$, we are going to use the approach from Section \ref{sec:Hecke_transf}. It turns out that in the case which is being considered (the variety $X^5$), that approach allows to completely describe the derived category.

We start with the tautological vector bundle $\Uu_2$ on $\Gad$ and its pullback $\pb^* \Uu_2$ to $E_2 \cong G_2/B$. Recall that $\xi \circ i_2 = \xi \circ i_1 = \id_{G_2/B}$, where the divisors $E_1$ and $E_2$ are identified with $G_2/B$. Thus, $i_1^* \xi^* \pb^* \Uu_2 = \pb^* \Uu_2$, and we have a short exact sequence (cf. sequence~\eqref{eq:seqforU}):
\[
	0 \to \Ll_{-\oa} \to i_1^* \xi^* \pb^* \Uu_2 \to \Ll_{\oa-\ob} \to 0,
\]
and by adjunction of $i_1^*$ and ${i_1}_*$ we obtain a surjective morphism $\xi^*\pb^* \Uu_2 \to {i_1}_* \Ll_{\oa-\ob}$  of coherent sheaves on $\Xt$.

\begin{defn}\label{def:tilde_U-def}
  	Define $\Uut$ to be the kernel of the above surjective map:
  	\begin{equation}\label{eq:defseq_tilde_U}
		0 \to \Uut \to \xi^* \pb^* \Uu_2 \to {i_1}_* \Ll_{\oa-\ob} \to 0,
	\end{equation}
\end{defn}

\begin{prop}\label{prop:tilde_U-prop}

Let $\Uut$ be as above.
\begin{enumerate}
	\item The coherent sheaf $\Uut$ is locally free of rank two on $\Xt$.
	\item $\Uut$ is the pullback of a vector bundle $\Ub$ of rank two on $X$.
	\item The bundle $\Ub$ is exceptional.
  	\item The bundles $\scal{\Ub,\Oo_X}$ form an exceptional pair.
	\item We have $\det(\Uut) = \pi^* \det(\Ub) = \pi^* \Oo_X(-1)$.
\end{enumerate}
\end{prop}

\begin{proof}
In the notation of Section~\ref{section:exc-gen}, we have $\Uu_2 = F_Y$, $\Ll_{\oa-\ob} = \cO_p(1)$ and $\Ll_{-\oa} = \Omega^1_{E/Y} \otimes \cO_p(1)$. In particular, $\Uut = \Ft_Y$ and the results follow from Propositions~\ref{fact-ll}, \ref{prop:desc}, \ref{prop:exc} and \ref{prop:esq}. We have $\Ub = \cF_Y$, and the isomorphism in item (5) is obtained by passing to determinants in the short exact sequence~\eqref{eq:defseq_tilde_U}. Note that we have $\xi_* \Uut = \xi_* \Ft_Y = \Ll_{-\oa}$.
\end{proof}

We have obtained the bundle $\Ub$ on $X$ starting from the bundle $\Uu_2$ which is pulled back via $\pb^*$ to $G_2/B$ from the grassmannian $\Gad$. One can also apply the above construction to vector bundles that are pulled back from $\Qs_5$ via the projection $\pa$. We will see next that when applied to the spinor bundle $\Ss$ on $\Qs_5$, this gives an {\it almost} exceptional bundle (cf. \cite{BK} and Proposition \ref{prop:tilde_S-prop} below) $\Sb$ on $X$. To this end we need an auxiliary fact.

\begin{prop}\label{prop:spinor_U_2_sequence}
	On $G_2/B$, we have an exact sequence of vector bundles 
	\begin{equation}\label{eq:spinor_U_2_sequence}
		0\to \pb^* \Uu_2 \to \pa^* \Ss \to \pb^* \Uu_2^\vee\otimes \Ll_{-\oa} \to 0.
	\end{equation}
\end{prop}

\begin{proof}
This was shown in \cite[Proposition 3 and Lemma 4]{Kuz2}. For convenience of the reader, we give a sketch of its proof, providing details of some cohomological calculations which are paramount to this section.

We first verify that the graded vector space $\Hom ^{\bullet}_{G_2/B}(\pb^* \Uu_2^\vee\otimes \Ll_{-\oa},\pb^* \Uu_2)=\Hom ^{\bullet}_{G_2/B}(\pb^* \Uu_2^\vee,\pb^* \Uu_2\otimes \Ll_{\oa})$ is isomorphic to $\C [-1]$. To see this, recall that the bundle $\pb^* \Uu_2^\vee$ has a two--step filtration by the line bundles $\Ll _{\ob-\oa}$ and $\Ll _{\oa}$, as can be seen from sequence (\ref{eq:seqforU}). Similarly, the bundle $\pb^* \Uu_2\otimes \Ll_{\oa}$ has a two--step filtration by the line bundles $\Ll _{2\oa-\ob}$ and $\Oo _{G_2/B}$. Applying Theorem \ref{th:Bott-Demazure_th}, we immediately see that the pairwise $\Ext$--groups between all line bundles above are trivial, except one which is isomorphic to $\C$ in degree 1. Specifically, the weights of line bundles in question are $\oa-\ob,-\oa,3\oa-2\ob,\oa-\ob$, and all these weights except the weight $3\oa-2\ob=-\beta$ are singular, since $\langle \oa-\ob, \beta ^{\vee}\rangle =\langle -\oa,\alpha ^{\vee}\rangle =-1$ (recall that $\langle \rho ,\gamma ^{\vee}\rangle =1$ for any simple root $\gamma$). Hence, by Theorem \ref{th:Bott-Demazure_th}, (1) the line bundles 
$\Ll _{\oa-\ob}$ and $\Ll _{-\oa}$ are acyclic. As for the weight $3\oa-2\ob=-\beta$, after applying the simple reflection $s_{\beta}$ to it, we obtain $s_{\beta}\cdot (-\beta)=0$, and Theorem \ref{th:Bott-Demazure_th}, (3) gives that ${\rm H}^1(G_2/B,\Ll _{-\beta})=\Ext ^1_{G_2/B}(\pb^* \Uu_2^\vee\otimes \Ll_{-\oa},\pb^* \Uu_2)=\C$. 

Associated to a unique non--trivial extension is a short exact sequence:
\begin{equation}\label{eq:spinor_U_2_sequence-prop}
0\rightarrow \pb^* \Uu_2\rightarrow ?\rightarrow \pb^* \Uu_2^\vee\otimes \Ll_{-\oa}\rightarrow 0.
\end{equation}

The coherent sheaf $?$ in the middle, being an extension of two locally free sheaves, is also locally free. We check that it is the pullback of a vector bundle on ${\sf Q}_5$. To ensure that, it is sufficient to show that ${\pi _{\a}}_{\ast}(?\otimes \Ll _{-\ob})=0$. Indeed, if $?=\pi _{\a}^{\ast}\Ee$ for a vector bundle $\Ee$ on ${\sf Q}_5$, then by the projection formula 
${\pi _{\a}}_{\ast}(?\otimes \Ll _{-\ob})={\pi _{\a}}_{\ast}(\pi _{\a}^{\ast}\Ee\otimes \Ll _{-\ob})=\Ee \otimes {\pi _{\a}}_{\ast}\Ll _{-\ob}=0$, since the line bundle $\Ll _{-\ob}$ has the relative degree $-1$ with respect to the projection $\pi _{\a}$. It follows from \cite{Or} that this sufficient condition is also necessary that a vector bundle on ${G_2/B}$ be pulled back 
from ${\sf Q}_5$. Using the above filtrations for the bundles $\pb^* \Uu_2$ and $\pb^* \Uu_2^\vee\otimes \Ll_{-\oa}$, tensoring (\ref{eq:spinor_U_2_sequence-prop}) with $\Ll _{-\ob}$, and finally applying ${\pi _{\a}}_{\ast}$ to the resulting sequence, one may check that ${\pi _{\a}}_{\ast}(?\otimes \Ll _{-\ob})$ fits into an exact triangle $\dots \rightarrow {\pi _{\a}}_{\ast}(?\otimes \Ll _{-\ob})\rightarrow \Ll _{-2\oa}\rightarrow \Ll _{-2\oa}\rightarrow \dots$ which forces ${\pi _{\a}}_{\ast}(?\otimes \Ll _{-\ob})$ to be zero.

Thus, there is a vector bundle $\Ee$ on ${\sf Q}_5$, such that $\pi ^{\ast}_{\a}\Ee =?$. By \cite[Theorem 4.19]{Kap}, the category $\Dd ^b({\sf Q}_5)$ has a full exceptional collection $\langle \Ll _{-3\oa}, \Ll _{-2\oa},\Ll _{-\oa},\Ss,\Oo _{{\sf Q}_5},\Ll _{\oa}\rangle$. 
To ensure that $\Ee$ is isomorphic to the spinor bundle $\Ss$, we check first that $\Ee$ is right orthogonal to line bundles $\Oo _{{\sf Q}_5}$ and $\Ll _{\oa}$ and is left ortogonal to line bundles $\Ll _{-3\oa}, \Ll _{-2\oa}$, and $\Ll _{-\oa}$. This easy cohomological computation which we are now skipping gives that $\Ee$ is a direct sum of copies of the bundle $\Ss$. However, the rank of $\Ss$ is equal to four which is the same as the rank of $\Ee$ as is seen from the sequence (\ref{eq:spinor_U_2_sequence-prop}). Hence, $\Ee$ is isomorphic to $\Ss$, and the statement follows.
\end{proof}

\begin{defn}
Define $\Sst$ as the kernel of the natural map
\begin{equation}\label{eq:defseq_tilde_S}
	0 \to \Sst \to \xi^* \pa^* \Ss \to {i_2}_* (\pb^* \Uu_2^\vee \otimes \Ll_{-\oa}) \to 0.
\end{equation}
\end{defn}

\begin{prop}\label{prop:tilde_S-prop}
\begin{enumerate}
	\item The coherent sheaf $\Sst$ is locally free of rank four on $\Xt$.
	\item $\Sst$ is the pullback of a vector bundle $\Sb$ of rank four on $X$.
	\item We have $\Hom _X^{\bullet}(\Sb,\Sb) = \C \oplus \C$.
    \item We have $\det(\Sst) = \pi^* \det(\Sb) = \pi^* \Oo_X(-2) = \Oo_\xi(-2)$.
\end{enumerate}
\end{prop}

\begin{proof}
Item (1) is similiar to Proposition~\ref{prop:tilde_U-prop}, $(1)$ since $E_2$ is a smooth divisor in $\Xt$ and $\pb^* \Uu_2^\vee \otimes \Ll_{-\oa}$ is a locally free sheaf on $E_2$.

For (2), note that the restriction of $\Sst$ to $E_1$ is isomorphic to 
$\pa^* \Ss$. To compute the restriction of $\Sst$ to $E_2$, apply $i_2^*$ to~\eqref{eq:defseq_tilde_S}:
\[
	0 \to \Le^1 i_2^* {i_2}_* (\pb^* \Uu_2^\vee \otimes \Ll_{-\oa}) \to i_2^*\Sst \to \pa^* \Ss \to \pb^* \Uu_2^\vee \otimes \Ll_{-\oa} \to 0.
\]
We have $\Le^1 i_2^* {i_2}_* (\pb^* \Uu_2^\vee \otimes \Ll_{-\oa}) = \pb^* \Uu_2^\vee \otimes \Ll_{-\oa} \otimes \Ll_{-\ob+\oa} = \pb^* \Uu_2$. We deduce that
\begin{equation}\label{eq:i_2-restriction_tilde_S}
	0 \to \pb^* \Uu_2 \to i_2^* \Sst \to \pb^* \Uu_2\to 0,
\end{equation}
hence $i_2^*\Sst$ is trivial when restricted to fibers of $\pb$. As the bundle $\pb^* \Uu_2$ is exceptional, the above short exact sequence splits. Thus, $i_2^* \Sst = \pb^* \Uu_2 \oplus \pb^* \Uu_2$.

To prove (3), we apply $\Hom_\Xt(-,\Sst)$ to~\eqref{eq:defseq_tilde_S}. 
We first compute $\Ext ^i _\Xt(\xi^* \pa^* \Ss,\Sst)$. Since $\xi_* \Sst = \pb^* \Uu_2$, this group is isomorphic to $\Ext ^i_{G_2/B}(\pa^* \Ss,\pb^* \Uu_2)$. To compute the latter, apply $\Hom_{G_2/B}(\pa^* \Ss,-)$ to sequence \eqref{eq:seqforU}. We obtain $\Ext ^i_{G_2/B}(\pa^* \Ss,\Ll_{-\oa}) = \Hr^i(G_2/B,\pa^* \Ss)=0$ and $\Ext ^i_{G_2/B}(\pa^* \Ss,\Ll_{\oa-\ob}) = 0$; the latter group is trivial since ${\pa}_* \Ll_{-\ob} = 0$.

\quad Next, using the right adjoint for ${i_2}_{\ast}$, we obtain an isomorphism $\Ext ^i_\Xt({i_2}_* (\pb^* \Uu_2^\vee \otimes \Ll_{-\oa}),\Sst)=\Ext ^{i-1}_{G_2/B}(\pb^* \Uu_2^\vee \otimes \Ll_{-\oa},i_2^* \Sst \otimes \Ll_{\ob-\oa})$. From \eqref{eq:i_2-restriction_tilde_S}, we see that the latter group is isomorphic to $\Ext ^{i-1}_{G_2/B}(\pb^* \Uu_2^\vee,\pb^* \Uu_2^\vee \oplus \pb^* \Uu_2^\vee)$. This gives $\Hom _\Xt^{\bullet}({i_2}_* (\pb^* \Uu_2^\vee \otimes \Ll_{-\oa}),\Sst)=(\C\oplus \C)[-1]$. From the previous paragraph and from the long exact sequence associated to ~\eqref{eq:defseq_tilde_S}, we obtain $\Hom_\Xt^{\bullet}(\Sst,\Sst)=\C\oplus \C$.
This proves (3).

Finally, item (4) is an easy computation.
\end{proof}

The top row in the following diagram represents the bundle $\Sst$ as an extension of two locally free sheaves on $\Xt$. The middle column is short exact sequence ~\eqref{eq:defseq_tilde_S}, while the last column is the tensor product of $\xi^* \pb^* (\Uu_2^\vee \otimes \Ll_{-\oa})$ with the short exact sequence
\[
	0 \to \Oo_\xi(-E_2) \to \Oo_\Xt \to {i_2}_* \Oo_{E_2} \to 0,
\] 
where we use Proposition \ref{prop:rel_Euler_seq_rho} and Lemma \ref{lem:E_1-E_2-O(rho)-lemma} to identify $\xi^* (\pb^* \Uu_2^\vee \otimes \Ll_{-\oa}) \otimes \Oo_\Xt(-E_2)$ with $\xi^* \pb^* \Uu_2^\vee \otimes \Oo_\xi(-1)$):

\[
	\xymatrix{
  		& & 0 \ar[d] & 0 \ar[d] & \\
  		0 \ar[r] & \xi^*\pb^* \Uu_2 \ar[r] \ar@{=}[d] & \Sst \ar[r]\ar[d] & \xi^* \pb^* \Uu_2^\vee \otimes \Oo_\xi(-1) \ar[r] \ar[d]& 0 \\
  		0 \ar[r] & \xi^* {\pb}^* \Uu_2 \ar[r]& \xi^* \pi_\a^* \Ss \ar[r] \ar[d] & \xi^* (\pb^* \Uu_2^\vee \otimes \Ll_{-\oa}) \ar[r] \ar[d] & 0 \\
 		& & {i_2}_*(\pb^* \Uu_2^\vee \otimes \Ll_{-\oa}) \ar@{=}[r] \ar[d] & {i_2}_* (\pi_{\beta}^* \Uu_2^\vee \otimes \Ll_{-\oa}) \ar[d] & \\
  		& & 0 & 0 & \\
	}
\]

The first row represents a non--trivial element in $\Ext^1_\Xt(\xi^* \pb^* \Uu_2^\vee \otimes \Oo_\xi(-1),\xi^* \pb^* \Uu_2) = \C$. The group $\Ext^2_\Xt({i_2}_* (\pb^* \Uu_2^\vee \otimes \Ll_{-\oa}),\xi^* \pb^* \Uu_2)$ is trivial, and the bundle $\xi^* \pa^* \Ss$ in the center of the above diagram trivializes this extension.

\begin{prop}\label{p:diag}
As $\C$-algebras, we have $\End_\Xt(\Sst) \simeq \C[t]/t^2$.
\end{prop}

\begin{proof}
By \cite{At}, the vector bundle $\Sst$ is uniquely decomposed, up to a reordering, into a direct sum of indecomposable vector bundles. If $\Sst$ were decomposable, it would be the direct sum of two vector bundles, namely, $\Sst = \Sst_1 \oplus \Sst_2$. Recall that the restriction $i_1^* \Sst$ of $\Sst$ to $E_1$ is isomorphic to $\pa^* \Ss$, hence it is an indecomposable bundle on $E_1$; if the bundle $\Sst$ was split as above, then applying $i_1^*$ to the direct sum decomposition, one would obtain an isomorphism $i_1^* \Sst = i_1^* \Sst_1 \oplus i_1^* \Sst_2$, with each of the summands being necessarily non-trivial. That would contradict the indecomposability of $i_1^* \Sst$.

Therefore the bundle $\Sst$ is indecomposable, hence the algebra $\End_\Xt(\Sst)$ is local. By Proposition \ref{prop:tilde_S-prop}, (3), $\dim_\C \End_\Xt(\Sst) = 2$, proving the claim.
\end{proof}

\begin{prop}\label{prop:S_to_U_to_S}
We have $\Hom_\Xt^{\bullet}(\Sst,\Uut) = \Hom_\Xt^{\bullet}(\Uut,\Sst) = \C$.
\end{prop}

\begin{proof}
The proof is similar to that of Proposition \ref{prop:tilde_S-prop}.
To compute the first group, apply $\Hom_\Xt(-,\Uut)$ to~\eqref{eq:defseq_tilde_S}. We need to compute $\Ext^i_\Xt(\xi^* \pa^* \Ss,\Uut)$ and $\Ext^i_\Xt({i_2}_* (\pb^* \Uu_2^\vee \otimes \Ll_{-\oa}),\Uut)$. We see that the former group is trivial as  
  \[
  	\Ext^i_\Xt(\xi^* \pa^* \Ss,\Uut)=\Ext^i_{G_2/B}(\pa^* \Ss,\xi_ *\Uut) = \Ext^i_{G_2/B}(\pa^* \Ss,\Ll_{-\oa}) = \Hr^i(\Qs_5,\Ss)=0.
  \]
  For the latter group, we have an isomorphism 
  \[
  	\Ext^i_\Xt({i_2}_* (\pb^* \Uu_2^\vee \otimes \Ll_{-\oa}),\Uut)=\Ext_{G_2/B}^{i-1}(\pb^* \Uu_2^\vee \otimes \Ll_{-\oa},i_2^* \Uut \otimes \Ll_{\ob-\oa}).
  \]

From \eqref{eq:defseq_tilde_U}, we see that the last group is isomorphic to $\Ext_{G_2/B}^{i-1}(\pb^* \Uu_2^\vee,\pb^* \Uu_2^\vee)$.
The bundle $\pb^* \Uu_2^\vee$ being exceptional, we obtain from the long exact sequence associated to \eqref{eq:defseq_tilde_S} an isomorphism $\Hom_\Xt^{\bullet}(\Sst,\Uut)=\C$. This proves the first statement. For the $\Hom$--groups in the opposite direction, apply $\Hom_\Xt(-,\Sst)$ to \eqref{eq:defseq_tilde_U}. We get

 \[
  	\Hom^{\bullet}_\Xt({i_1}_* \Ll_{\oa-\ob},\Sst)=\Hr^{\bullet}(G_2/B,\pa^* \Ss)=0
  \]

  and 
   \[
  	\Hom_\Xt^{\bullet}(\xi^* \pb^* \Uu_2,\Sst) = \Hom_\Xt^{\bullet}(\pb^* \Uu_2,\xi_* \Sst) = \Hom_\Xt^{\bullet}(\pb^* \Uu_2,\pb^* \Uu_2)  = \C.
  \]
  
We conclude as above invoking the long exact sequence associated to sequence \eqref{eq:defseq_tilde_U}. \end{proof}

We now use Proposition \ref{prop:S_to_U_to_S} to construct, out of the pair of bundles $\Ub$ and $\Sb$, another genuine exceptional object $\Sbh$ on $X$. Together with $\Ub$, these two objects will form an exceptional pair (Proposition \ref{prop:hatS_except} and Corollary \ref{cor:StoUorthog} below).

\begin{defn}\label{def:def_hat_S}
  Let $\Ssh$ denote $\Cone(\Sst \otimes \Hom^{\bullet}_\Xt(\Sst,\Uut) \to \Uut)$, where the map is the canonical evaluation map. Identifying $\Hom^{\bullet}_\Xt(\Sst,\Uut)$ with $\C$ via Proposition \ref{prop:S_to_U_to_S}, we can also write $\Ssh$ as $\Cone(\Sst \to \Uut)$.
\end{defn}

We have a distinguished triangle
\begin{equation}\label{eq:cone_tilde_S_to_tilde_U}
	\dots \to \Sst \to \Uut \to \Ssh \to \Sst[1] \to \dots 
\end{equation}

Passing to cohomology, we obtain an exact sequence, both $\Sst$ and $\Uut$  being vector bundles:

$$0 \to \Hh^{-1}(\Ssh) \to \Sst \to \Uut \to \Hh^0(\Ssh) \to 0$$.

\begin{prop}\label{prop:cohomology_hat_S}
We have $\Hh^{-1}(\Ssh) = \xi^* \pb^* \Uu_2$ and $\Hh^0(\Ssh) = {i_1}_* \Ll_{-\oa}$.
\end{prop}

\begin{proof}
By the previous proof, we have $\Ext_\Xt^1({i_2}_* (\pb^* \Uu_2^\vee \otimes \Ll_{-\oa}),\Uut) = \C$. Consider the unique non-trivial extension:
\begin{equation}\label{eq:def_seq_hat_U}
	0 \to \Uut \to \Uuh \to {i_2}_* (\pb^* \Uu_2^\vee \otimes \Ll_{-\oa}) \to 0.
\end{equation}

We claim that the middle term $\Uuh$ fits into a short exact sequence:
\begin{equation}\label{eq:hat_U_as_cokernel-1}
	0 \to \xi^* \Ll_{-\oa} \otimes \Oo_\Xt(-E_1) \to \xi^* (\Ll_{-\oa} \oplus \pb^* \Uu_2^\vee \otimes \Ll_{-\oa}) \to \Uuh \to 0.
\end{equation}

Indeed, start with a distinguished triangle 
\begin{equation}\label{eq:rho-decomp_of_hatU}
	\dots \to \xi^* \xi_* \Uuh \to \Uuh \to \xi^*(?) \otimes \Oo_\xi(-1) \to \xi^* \xi_* \Uuh[1] \to \dots 
\end{equation}

which is obtained by completing the natural adjunction morphism $\xi^* \xi_* \Uuh \to \Uuh$ to a distinguished triangle and by identifying the cone of that morphism with $\xi^*(?) \otimes \Oo_\xi(-1)$ via semiorthogonal decomposition from \cite{Or} for the $\Pp ^1$--bundle $\xi :\Xt\rightarrow G_2/B$, as 
in Proposition \ref{prop:rel_Euler_seq_rho}. Here $?$ is an object $\Dd ^b(G_2/B)$. To compute $\xi_*\Uuh$ and $?$, apply $\xi_*$ to (\ref{eq:def_seq_hat_U}):

\[
	\dots \to \xi_* \Uut \to \xi_* \Uuh \to \pb^* \Uu_2^* \otimes \Ll_{-\oa}\to \dots.
\]
Remembering that $\xi_* \Uut = \Ll_{-\oa}$, the above triangle reduces to 
\[
	0 \to \Ll_{-\oa} \to \xi_* \Uuh \to \pb^* \Uu_2^\vee \otimes \Ll_{-\oa} \to 0,
\]
which splits since $\Ext_{G_2/B}^1(\pb^* \Uu_2^\vee \otimes \Ll_{-\oa},\Ll_{-\oa}) = \Hr^1(G_2/B,\pb^* \Uu_2) = 0$. Thus, $\xi_* \Uuh = \Ll_{-\oa} \oplus \pb^* \Uu_2^\vee \otimes \Ll_{-\oa}$. To compute $?$, tensor~\eqref{eq:rho-decomp_of_hatU} with $\Oo_\xi(-1)$ and apply $\xi_*$. As $\xi_* \Oo_\xi(-1) = 0$, by projection formula we obtain $\xi_* (\xi^* \xi_*\Uuh \otimes \Oo_\xi(-1)) = 0$. Further, $\xi_* (\xi^* (?) \otimes \Oo_\xi(-2)) = ? \otimes \xi_* \Oo_\xi(-2) = ? \otimes \Ll_{-\rho}[-1]$. Thus, $? \otimes \Ll_{-\rho}[-1] = \xi_* (\Uuh \otimes \Oo_\xi(-1))$, and $? = \xi_* (\Uuh \otimes \Oo_\xi(-1)) \otimes \Ll_{\rho}[1]$.

To compute $\xi_* (\Uuh \otimes \Oo_\xi(-1))$, tensor~\eqref{eq:def_seq_hat_U} with $\Oo_\xi(-1)$:
\[
	0 \to \Uut \otimes \Oo_\xi(-1) \to \Uuh \otimes \Oo_\xi(-1)\to {i_2}_* (\pb^* \Uu_2^\vee \otimes \Ll_{-\oa-\ob}) \to 0.
\]
Applying $\xi_*$ to the above sequence, we get
\[
	0\to \Ri^0\xi_*(\Uut \otimes \cO_\xi(-1))\to \Ri^0\xi_*(\Uuh \otimes \Oo_\xi(-1))\to \pb^*\Uu_2^\vee \otimes \Ll_{-\rho} \to \Ri^1\xi_*(\Uut \otimes \cO_\xi(-1))\to \dots 
\]
To compute $\xi_*(\Uut \otimes \Oo_\xi(-1))$, tensor~\eqref{eq:defseq_tilde_U} with $\Oo_\xi(-1)$ and apply $\xi_*$, obtaining an isomorphism $\xi_*(\Uut \otimes \Oo_\xi(-1))=\xi_*{i_1}_*(\Ll_{\oa-\ob-\oa})[-1]$ (using again isomorphisms $\xi_*\Oo_\xi(-1)=0$ and ${i_1}^*\Oo_\xi(-1)={i_1}^*\pi^*\Oo_X(-1)=\Ll_{-\oa}$). Thus, $\xi_*(\Uut \otimes \Oo_\xi(-1))=\Ll_{-\ob}[-1]$.
From the above sequence we see that 
\[
	0\to \Ri^0\xi_*(\Uuh \otimes \Oo_\xi(-1))\to \pb^*\Uu_2^\vee \otimes \Ll_{-\rho}\to \Ll_{-\ob}\to 0,
\]
which is the sequence~\eqref{eq:seqforU} tensored with $\Ll_{-\oa}$. We conclude that $\xi_*(\Uuh \otimes \Oo_\xi(-1))=\Ll_{-2\oa}$. Returning to $?$, we obtain an isomorphism $?=\xi_*(\Uuh \otimes \Oo_\xi(-1)) \otimes \Ll_{\rho}[1]=\Ll_{\rho-2\oa}[1]=\Ll_{\ob-\oa}[1]$. Thus, the triangle~\eqref{eq:rho-decomp_of_hatU} gives an exact sequence:
\begin{equation}\label{eq:hat_U_as_cokernel-2}
	0 \to \xi^*\Ll_{\ob-\oa} \otimes \Oo_\xi(-1)\to \xi^*(\Ll_{-\oa}\oplus \pb^*\Uu_2^\vee \otimes \Ll_{-\oa}) \to \Uuh \to 0.
\end{equation}
Lemma~\ref{lem:E_1-E_2-O(rho)-lemma} implies that $\xi^*\Ll_{\ob-\oa} \otimes \Oo_\xi(-1)=\xi^*\Ll_{-\oa} \otimes \Oo_\Xt(-E_1)$. The morphism $\xi^*\Ll_{-\oa} \otimes \Oo_\Xt(-E_1)\to \xi^*\Ll_{-\oa}$ in~\eqref{eq:hat_U_as_cokernel-2} is obtained by tensoring the map 
$\Oo_\Xt(-E_1)\to \Oo_\Xt$ with $\xi^*\Ll_{-\oa}$, thus it is injective. The morphism $\xi^*\Ll_{\ob-\oa} \otimes \Oo_\xi(-1)\to \xi^*(\pb^*\Uu_2^\vee \otimes \Ll_{-\oa})$ is also injective since its source is an invertible sheaf and its target is a locally free sheaf (we also compute $\Hom_\Xt^{\bullet}(\xi^*\Ll_{\ob-\oa} \otimes \Oo_\xi(-1),\xi^*(\pb^*\Uu_2^\vee \otimes \Ll_{-\oa}))= V(\ob)$). This proves~\eqref{eq:hat_U_as_cokernel-1}.

Consider the push-out of the extension~\eqref{eq:defseq_tilde_S} along the map $\Sst\to \Uut$ in~\eqref{eq:cone_tilde_S_to_tilde_U}. It follows from the above that the push-out extension is isomorphic to $\Uuh$:
\[
	\xymatrix{
		& 0\ar[r]&\Sst \ar[r]\ar[d] & \xi^*\pa^*\Ss \ar[r]\ar[d]^{\widetilde f} & {i_2}_*(\pb^*\Uu_2^\vee \otimes \Ll_{-\oa})\ar[r]\ar@{=}[d] & 0\\
		& 0 \ar[r]& \Uut \ar[r] & \Uuh \ar[r] & {i_2}_*(\pb^*\Uu_2^\vee \otimes \Ll_{-\oa}) \ar[r] & 0.\\
	}
\]
Our next step is to pull back short exact sequence ~\eqref{eq:hat_U_as_cokernel-1} along the map $\tilde f: \xi^*\pa^*\Ss\to \Uuh$. Denoting $\xi^*(\Ll_{-\oa}\oplus \pb^*\Uu_2^\vee \otimes \Ll_{-\oa})$ by $\A$, we obtain 
\[
\xymatrix{
0\ar@{->}[r]&\xi^*\Ll_{-\oa} \otimes \Oo_\Xt(-E_1) \ar@{->}[r]\ar@{=}[d] & ? \ar@{->}[r]\ar@{->}[d] & \xi^*\pa^*\Ss\ar@{->}[r]\ar@{->}[d]^{\tilde f} & 0\\
0 \ar@{->}[r]& \xi^*\Ll_{-\oa} \otimes \Oo_\Xt(-E_1) \ar@{->}[r] & \A
\ar@{->}[r] & \Uuh \ar@{->}[r]& 0.\\
}
\]
We compute $\Ext^1_\Xt(\xi^*\pa^*\Ss,\xi^*\Ll_{-\oa} \otimes \Oo_\Xt(-E_1))=0$. Thus, the top row in the above diagram splits and we obtain an isomorphism $? = \xi^*\pa^*\Ss \oplus \xi^*\Ll_{-\oa} \otimes \Oo_\Xt(-E_1)$. We finally see that the cone $\Ssh$ is isomorphic to the cone of the morphism 

\begin{equation}\label{eq:cone_hatS_explicit}
\xi^*\pa^*\Ss \oplus \xi^*\Ll_{-\oa} \otimes \Oo_\Xt(-E_1)\to \xi^*(\Ll_{-\oa}\oplus \pb^*\Uu_2^\vee \otimes \Ll_{-\oa})
\end{equation}

in which $\Hom^{\bullet}_\Xt(\xi^*\pa^*\Ss,\xi^*\Ll_{-\oa})=0$ and $\Hom^{\bullet}_\Xt(\xi^*\pa^*\Ss,\xi^*\pb^*\Uu_2^\vee \otimes \Ll_{-\oa})={\C}$. The unique non-trivial element of the latter group gives rise to sequence~\eqref{eq:spinor_U_2_sequence} from Proposition~\ref{prop:spinor_U_2_sequence}. The map $\xi^*\Ll_{-\oa} \otimes \Oo_\Xt(-E_1)\to \xi^*\Ll_{-\oa}$ is obtained by tensoring the injective map $\Oo_\Xt(-E_1)\to \Oo_\Xt$ with $\xi^*\Ll_{-\oa}$. Finally, $\Hom^{\bullet}_\Xt(\xi^*\Ll_{-\oa} \otimes \Oo_\Xt(-E_1),\xi^*(\pb^*\Uu_2^\vee \otimes \Ll_{-\oa}))= V(\ob)$ as was computed above. It follows that the kernel 
$\Hh^{-1}(\Ssh)$ is isomorphic to $\xi^*\pb^*\Uu_2$ and that the cokernel of~\eqref{eq:cone_hatS_explicit} is isomorphic to $\Hh^{0}(\Ssh)=\xi^*\Ll_{-\oa} \otimes {i_1}_*\Oo_{E_1}={i_1}_*\Ll_{-\oa}$, and Proposition~\ref{prop:cohomology_hat_S} follows. 
\end{proof}

\begin{lemma}\label{lem:tildeS_is_pullback}
  We have $\Ssh =\pi^*\Sbh$ for some $\Sbh \in \Dd^b(X)$.
\end{lemma}

\begin{proof}
Using Propositions \ref{prop:tilde_U-prop} and \ref{prop:tilde_S-prop}, and recalling that the functor $\pi ^{\ast}:\Dd ^b(X)\rightarrow \Dd ^b(\Xt)$ is fully faithful, we see that $\Ssh$ is the pull--back under $\pi$ of the cone of a unique non--zero map $\Sb \to \Ub$. Hence, $\Ssh =\pi^*\Sbh$ where $\Sbh=\Cone(\Sb \to \Ub)\in \Dd^b(X)$. 

Let us also give a direct proof which will be useful for what follows. By Proposition~\ref{prop:vecbun_blow-up_descent}, it is sufficient to check that the restrictions $i_1^*\Ssh$ and $i_2^*\Ssh$ are pulled--back from $\Qs_5$ and $\Gad$, respectively. Applying $i_2^*$ to \eqref{eq:cohomology_triangle_hatS}, we immediately see that $i_2^*\Ssh\in \pb^*\Dd^b(\Gad)$. Applying $i_1^*$, we obtain:
\begin{equation}\label{eq:i_1*to_cohomology_triangle_hatS1}
	\dots \to \pb^*\Uu_2[1]\to i_1^*{\Ssh}\to i_1^*{i_1}_*\Ll_{-\oa}\to \pb^*\Uu_2[2]\rightarrow \dots 
\end{equation}
Applying to the above triangle the cohomology functor $\Hh^0$, we compute $\Hh^0(i_1^*{\Ssh})=\Ll_{-\oa}$, while $\Hh^{-1}(i_1^*\Ssh)$ fits into a short exact sequence:
\begin{equation}\label{eq:i_1*to_cohomology_triangle_hatS2}
	0\to \pb^*\Uu_2 \to \Hh^{-1}(i_1^*\Ssh)\to \Ll_{\ob-2\oa} \to 0,
\end{equation}
as $\Hh^{-1}(i_1^*{i_1}_*\Ll_{-\oa})={\Le}^1i_1^*{i_1}_*\Ll_{-\oa}=\Ll_{\ob-2\oa}$. To see that $\Hh^{-1}(i_1^*{\Ssh})$ belongs to $\pa^*\Dd^b(\Qs_5)$, it is sufficient to show that ${\pa}_*( \Hh^{-1}(i_1^*{\Ssh}) \otimes \Ll_{-\ob})=0$. Tensoring~\eqref{eq:i_1*to_cohomology_triangle_hatS2} with $\Ll_{-\ob}$ and applying ${\pa}_*$ we see that $\Ri^0{\pa}_*\Ll_{-2\oa}=
\Ri^1{\pa}_*(\pb^*\Uu_2 \otimes \Ll_{-\ob})=\Ll_{-2\oa}$, hence ${\pa}_*( \Hh^{-1}(i_1^*{\Ssh}) \otimes \Ll_{-\ob})=0$.
\end{proof}

\begin{remark}
Up to tensoring with $\Ll_{-\oa}$, the dual extension to~\eqref{eq:i_1*to_cohomology_triangle_hatS1} is the push-out of~\eqref{eq:spinor_U_2_sequence} along the map $\pb^*\Uu_2\to \Ll_{\oa-\ob}$ from~\eqref{eq:tautological_seq_U}:
\[
	\xymatrix{
		& 0\ar@{->}[r]&\pb^*\Uu_2\ar@{->}[r]\ar@{->}[d] & \pa^*\Ss \ar@{->}[r]\ar@{->}[d] & \pb^*\Uu_2^\vee \otimes \Ll_{-\oa}\ar@{->}[r]\ar@{=}[d] & 0\\
		& 0 \ar@{->}[r]& \Ll_{\oa-\ob} \ar@{->}[r] & {\Gg} \ar@{->}[r] & \pb^*\Uu_2^\vee \otimes \Ll_{-\oa} \ar@{->}[r] & 0.\\
	}
\]
That is, we have an isomorphism $\Hh^{-1}(i_1^*{\Ssh}) = {\Gg}^\vee \otimes \Ll_{-\oa}$.
\end{remark}

\begin{prop}\label{prop:hatS_except}
	The object $\Ssh$ is exceptional.
\end{prop}

\begin{proof}
By Proposition~\ref{prop:cohomology_hat_S}, there is a distinguished triangle:
\begin{equation}\label{eq:cohomology_triangle_hatS}
	\dots \to \xi^*\pb^*\Uu_2[1]\to {\Ssh}\to {i_1}_*\Ll_{-\oa}\to \xi^*\pb^*\Uu_2[2]\to \dots 
\end{equation}
The coboundary map in this triangle lies in the group 
\[
	\Hom_\Xt({i_1}_*\Ll_{-\oa},\xi^*\pb^*\Uu_2[2])=\Ext_\Xt^2({i_1}_*\Ll_{-\oa},\xi^*\pb^*\Uu_2)={\C}
\]
and corresponds to a unique non-trivial extension in the latter group. Applying the functor $\Hom_\Xt(-,{\Ssh})$ to~\eqref{eq:cohomology_triangle_hatS}, and using~\eqref{eq:cone_tilde_S_to_tilde_U} to compute the groups $\Hom^{\bullet}_\Xt(\xi^*\pb^*\Uu_2,{\Ssh})$ and $\Hom^{\bullet}_\Xt({i_1}_*\Ll_{-\oa},{\Ssh})$, we arrive at the following:

\begin{equation}\label{eq:8.13-1}
\Hom^{\bullet}_\Xt(\xi^*\pb^*\Uu_2[1],\Sst) = \Hom^{\bullet}_{G_2/B}(\pb^*\Uu_2[1],\pb^*\Uu_2)={\C}[-1];
\end{equation}

\begin{equation}\label{eq:8.13-2}
\Hom^{\bullet}_\Xt(\xi^*\pb^*\Uu_2,\Uut)= 0; \ \Hom^{\bullet}_\Xt({i_1}_*\Ll_{-\oa},\Uut)= 0; \ \Hom^{\bullet}_\Xt({i_1}_*\Ll_{-\oa},\Sst) = 0.
\end{equation}

We therefore obtain $\Hom^{\bullet}_\Xt({\Ssh},{\Ssh})$ = $\Hom_\Xt^{\bullet}(\xi^*\pb^*\Uu_2[1],\Sst[1])={\C}$, as claimed.
\end{proof}

The first two $\Hom$--vanishings of (\ref{eq:8.13-2}) furnish an important orthogonality (see Corollary \ref{cor:StoUorthog} below) that will allow us 
to complete the construction of the starting block of our would--be Lefschetz exceptional collection. For convenience of the reader, we recall the details of this calculation.

For the first, applying $\xi_{\ast}$ to (\ref{eq:defseq_tilde_U}), we see that $\xi_{\ast}\Uut=\Ll_{-\oa}$. By adjunction, $\Hom^{\bullet}_\Xt(\xi^*\pb^*\Uu_2,\Uut)=\Hom^{\bullet}_{G_2/B}(\pb^*\Uu_2,\Ll_{-\oa})={\rm H}^{\bullet}(G_2/B,\pb^*{\Uu_2^{\ast}}\otimes \Ll_{-\oa})$. The last group is zero, by the projection formula and since ${\pb}_{\ast}\Ll_{-\oa}=0$.

For the second, recall that the right adjoint to ${i_1}_{\ast}(-)$ is isomorhic to $i_1^{\ast}(-)\otimes \Ll _{\oa-\ob}[-1]$ (recalling also that $\Oo _{E_1}(E_1)=\Ll _{\oa-\ob}$). Thus, $\Hom^{\bullet}_\Xt({i_1}_*\Ll_{-\oa},\Uut)=\Hom^{\bullet}_\Xt(\Ll_{-\oa},i_1^{\ast}\Uut\otimes \Ll _{\oa-\ob}[-1])$. Recalling that ${i_1}^*\Uut=\Oo _{E_1}\oplus \Ll _{-\oa}$, we see that by Theorem \ref{th:Bott-Demazure_th} both line bundles $\Ll _{2\oa-\ob}$ and $\Ll _{\oa-\ob}$ are acyclic. We conclude that  
$\Hom^{\bullet}_\Xt({i_1}_*\Ll_{-\oa},\Uut)=0$.

\begin{cor}\label{cor:StoUorthog}
We have $\Hom_\Xt^{\bullet}(\Ssh,\Uut)=0$. 
\end{cor}

\begin{proof}
Applying $\Hom _\Xt^{\bullet}(-,\Uut)$ to (\ref{eq:cohomology_triangle_hatS}) and using the first two $\Hom$--vanishings from (\ref{eq:8.13-2}) of Proposition \ref{prop:hatS_except}, we obtain the claim. 
\end{proof}

We note here for the record that $\Hom_{\Xt}^{\bullet}(\Uut,\Ssh) = \C \oplus \C[1]$. Its computation is similar to that of (\ref{eq:8.13-2}): using sequence (\ref{eq:defseq_tilde_U}) from Definition \ref{def:tilde_U-def}
and Proposition \ref{prop:tilde_U-prop}, (3), together with the isomorphism ${i_1}^*\Uut=\Oo _{E_1}\oplus \Ll _{-\oa}$, we obtain
$$
\Hom_{\Xt}^{\bullet}(\Uut,{i_1}_*\Ll_{-\oa})= \C \quad ; \quad \Hom_{\Xt}^{\bullet}(\Uut,\xi^*\pb^*\Uu_2) = \C.
$$

We conclude by applying $\Hom_\Xt(\Uut,-)$ to (\ref{eq:cohomology_triangle_hatS}) and using the above isomorphisms.

\begin{prop}\label{lem:exceptional_triple_on_X}
	The triple $\scal{\Ub,\Sbh, \Oo_X}$ of $\Dd^b(X)$ is exceptional.
\end{prop}

\begin{proof}
By Proposition \ref{prop:tilde_U-prop}, the pair $\scal{\Ub,\Oo_X}$ is exceptional. By Lemma \ref{lem:tildeS_is_pullback} we have $\pi^*\Sbh = \Ssh$, and by Proposition \ref{prop:hatS_except} the object $\Ssh$ is exceptional. Hence, so is $\Sbh$. Corollary \ref{cor:StoUorthog} gives that 
$\Hom_\Xt^{\bullet}(\Ssh,\Uut)=\Hom_X^{\bullet}(\Sbh,\Ub)=0$. Finally, from distinguished triangle \eqref{eq:cone_tilde_S_to_tilde_U} and from the vanishings $\Hom_\Xt^{\bullet}(\cO_\Xt,\Uut) = 0 = \Hom_\Xt^{\bullet}(\cO_\Xt,\Sst)$, we get $\Hom_\Xt^{\bullet}(\cO_\Xt,\Ssh) = \Hom_X^{\bullet}(\cO_X,\Sbh)=0$.
\end{proof}

\subsection{A rectangular Lefschetz exceptional collection on $X$}
Denote by $\A_0$ the subcategory $\scal{\Ub, \Sbh, \Oo_X}$ from Proposition \ref{lem:exceptional_triple_on_X}. By Lemma \ref{lem:Lemma 1.58_Huyb}, $\A_0$ is an admissible subcategory of $\Dd^b(X)$.
For $i \in [1,3]$, let $\A_0(-i)$ denote the subcategory generated by $\scal{\Ub \otimes \Oo_X(-i),\Sbh \otimes \Oo_X(-i), \Oo_X(-i)}$. The subcategories $\A_0(-i)$ are generated by exceptional sequences and are also admissible subcategories of $\Dd^b(X)$.

Recall that the index of $X$ is equal to 4. Proposition~\ref{lem:exceptional_triple_on_X} thus suggests to consider the following sequence of admissible subcategories:
\begin{equation}\label{eq:Lefschetz_exc_coll_G_2hor}
\A = \scal{\A_0(-3), \A_0(-2), \A_0(-1), \A_0}.
\end{equation}

\comment{
\begin{equation}\label{eq:Lefschetz_exc_coll_G_2hor}
  \left\langle
  \begin{array}{cccccc}
    \Ub(-3), & \Sbh(-3), & \Oo_X(-3), & \Ub (-2), & \Sbh(-2), & \Oo_X(-2), \\
	\Ub (-1), & \Sbh(-1), & \Oo_X(-1), & \Ub, & \Sbh, & \Oo_X
  \end{array} \right\rangle
\end{equation}
}

\begin{lemma}\label{lem:exceptional_lemma}
We have the following vanishing results.
\begin{enumerate}[(a)]
	\item $\Hom_X^{\bullet}(\Ub,\Oo_X(-i))=0$ for $i \in [1,3]$.
	\item $\Hom_X^{\bullet}(\Sb,\Oo_X(-i))=0$ for $i \in [1,3]$.
	\item $\Hom_X^{\bullet}(\Ub,\Sb \otimes \Oo_X(-1))=0$.
	\item $\Hom_X^{\bullet}(\Sb,\Sb \otimes \Oo_X(-1)) = 0.$
	\item $\Hom_X^{\bullet}(\Sb,\Ub \otimes \Oo_X(-1))=0$.
	\item $\Hom_X^{\bullet}(\Oo_X,\Ub \otimes \Oo_X(-i))=\Hom_X^{\bullet}(\Oo_X,\Sb \otimes \Oo_X(-i))=0$ for $i \in [1,3]$.
	\item $\Hom_X^{\bullet}(\Ub,\Ub \otimes \Oo_X(-1)))=0$.
	\item $\Hom_X^{\bullet}(\Sb,\Sb \otimes \Oo_X(-2))=0$.
	\item $\Hom_X^{\bullet}(\Ub,\Ub \otimes \Oo_X(-2))=0$.
	\item $\Hom_X^{\bullet}(\Ub,\Sb \otimes \Oo_X(-2))=0$.
	\item $\Hom_X^{\bullet}(\Sb,\Ub \otimes \Oo_X(-2)))= 0$.
	\item $\Hom_X^{\bullet}(\Sb,\Ub \otimes \Oo_X(-3)))= 0$.
  	\item $\Hom_X^{\bullet}(\Sb,\Sb \otimes \Oo_X(-3))= 0$.
  	\item $\Hom_X^{\bullet}(\Ub,\Sb \otimes \Oo_X(-3))=0$.
  	\item $\Hom_X^{\bullet}(\Ub,\Ub \otimes \Oo_X(-3)))=0$.
  \end{enumerate}
\end{lemma}

\begin{proof}
As before, the functor $\pi^*:\Dd^b(X)\to \Dd^b(\Xt)$ being fully faithful, the above groups can be computed on $\Xt$. Recall the isomorphism $\pi^*\Oo_X(-1)=\Oo_\xi(-1)$.

(a) Applying $\Hom_\Xt(-,\Oo_\xi(-i))$ to~\eqref{eq:defseq_tilde_U}, we need to compute the following groups: $\Hom_\Xt(\xi^*\pb^*\Uu_2,\Oo_\xi(-i))$ and $\Hom_\Xt({i_1}_*\Ll_{\oa-\ob},\Oo_\xi(-i))$. We have 
\begin{equation}\label{eq:rho_O(-i)}
  \begin{array}{l}
    \xi_*\Oo_\xi(-1)=0, \quad \xi_*\Oo_\xi(-2)=\Ll_{-\rho}[-1], \quad \text{and} \\
    \xi_*\Oo_\xi(-3)= (\Ll_{-\oa-\rho}\oplus \Ll_{-\ob-\rho})[-1].\\
    \end{array}
\end{equation}
Using this and~\eqref{eq:seqforU}, we get
$\Hom_\Xt^{\bullet}(\xi^*\pb^*\Uu_2,\Oo_\xi(-i))=0$ for $i \in [1,3]$. Moreover
\[
	\Hom_\Xt^{\bullet}({i_1}_*\Ll_{\oa-\ob},\Oo_\xi(-i)) = \Hr^{\bullet}(\Qs_5,\Ll_{-i\oa}[-1]) = 0
\] for $i \in [1,3]$.

(b) Applying $\Hom_\Xt(-,\Oo_\xi(-1))$ to~\eqref{eq:defseq_tilde_S}, we need to compute the following groups: $\Hom_\Xt^{\bullet}(\xi^*\pa^*\Ss,\Oo_\xi(-i))$ and $\Hom_\Xt^{\bullet}({i_2}_*(\pb^*\Uu_2^\vee \otimes \Ll_{-\oa}),\Oo_\xi(-i))$.

For $i=1$, we have $\Hom_\Xt^{\bullet}(\xi^*\pa^*\Ss,\Oo_\xi(-1))=0$. For $i \in [2,3]$, use~\eqref{eq:defseq_tilde_S} and~\eqref{eq:rho_O(-i)}. Recall that the weights of $\Ss^\vee$ are equal to $0,2\oa-\ob,\ob-\oa,\oa$. We see that the cohomology of $\pa^*\Ss^\vee \otimes \Ll_{-\rho}$ and of $\pa^*\Ss^\vee \otimes (\Ll_{-\oa-\rho}\oplus \Ll_{-\ob-\rho})$ is trivial. Indeed, their respective weights are equal to 
\[
	\begin{array}{l}
   	-\rho, \quad \oa-2\ob, \quad -2\oa, \quad -\ob, \quad {\rm and} \\
 	-\oa-\rho, \quad -2\ob, \quad -3\oa, \quad
-\rho \ 
 	-\ob-\rho, \quad \oa-3\ob, \quad -\oa -\rho, \quad -2\ob.
	\end{array}
\]
We have $\Hom_\Xt^{\bullet}({i_2}_*(\pb^*\Uu_2^\vee \otimes \Ll_{-\oa}),\Oo_\xi(-i))= \Hr^{\bullet}(\Gad, \Uu_2^\vee  \otimes \Ll_{-i\ob}[-1])=0$ for $i \in [1,3]$.

(c) Applying $\Hom_\Xt(-,\Sst \otimes \Oo_\xi(-1))$ to~\eqref{eq:defseq_tilde_U}, we need to compute the $\Hom$--groups $\Hom_{G_2/B}^{\bullet}(\pb^*\Uu_2,\xi_*(\Sst \otimes \Oo_\xi(-1))$ and $\Hom_\Xt^{\bullet}({i_1}_*\Ll_{\oa-\ob},\Sst \otimes \Oo_\xi(-1))$. Tensoring~\eqref{eq:defseq_tilde_S} with $\Oo_\xi(-1)$ and applying $\xi_*$, we see that $\xi_*(\Sst \otimes \Oo_\xi(-1))=\pb^*\Uu_2 \otimes \Ll_{-\oa}[-1]$, and $\Hom_{G_2/B}^{\bullet}(\pb^*\Uu_2,\pb^*\Uu_2 \otimes \Ll_{-\oa}[-1])=0$. Finally, $\Hom_\Xt^\bullet({i_1}_*\Ll_{\oa-\ob},\Sst \otimes \Oo_\xi(-1)) = \Hr^\bullet(G_2/B,i_1^*\Sst \otimes \Ll_{-\oa}[-1])$. Since $i_1^*\Sst=\pa^*\Ss$, we conclude by $\Hr^{\bullet}(G_2/B,\pa^*\Ss \otimes \Ll_{-\oa})=0$.

(d) Applying $\Hom_\Xt(-,\Sst \otimes \Oo_\xi(-1))$ to~\eqref{eq:defseq_tilde_S}, we need to compute $\Hom_\Xt^{\bullet}(\xi^*\pa^*\Ss,\Sst \otimes \Oo_\xi(-1))$ and $\Hom_\Xt^{\bullet}({i_2}_*(\pb^*\Uu_2^\vee \otimes \Ll_{-\oa}),\Sst \otimes \Oo_\xi(-1))$. As above, we have
$\xi_*(\Sst \otimes \Oo_\xi(-1))=\pb^*\Uu_2 \otimes \Ll_{-\oa}[-1]$. By adjunction of $\xi^*$ and $\xi_*$, the first group is isomorphic to 
\[
	\Hom_{G_2/B}^{\bullet}(\pa^*\Ss,\pb^*\Uu_2 \otimes \Ll_{-\oa}[-1])=\Hr^{\bullet}(G_2/B,\pa^*\Ss^\vee \otimes \pb^*\Uu_2 \otimes \Ll_{-\oa}[-1]).
\]

It vanishes since the weights of $\pa^*\Ss^\vee \otimes \pb^*\Uu_2 \otimes \Ll_{-\oa} = \pa^* \Ss  \otimes \pb^*\Uu_2$ are equal to 

\[
	\begin{array}{l}
 		-2\oa, \quad -\ob, \quad \ob-3\oa, \quad -\oa, \
 		-\ob, \quad 2\oa-2\ob, \quad -\oa, \quad \oa-\ob. \ 
	\end{array}
\]

The second group is $\Hom_{G_2/B}^{\bullet}(\pb^*\Uu_2^\vee \otimes \Ll_{-\oa},\pb^*\Uu_2^{\oplus 2} \otimes \Ll_{-\ob} \otimes \Ll_{\ob-\oa}[-1])$ and vanishes since $\Hr^{\bullet}(G_2/B,\pb^*\Uu_2^{ \otimes 2})=0$.

(e) Apply $\Hom_\Xt(-,\Uut \otimes \Oo_\xi(-1))$ to~\eqref{eq:defseq_tilde_S}. We need to compute $\Hom_\Xt^{\bullet}(\xi^*\pa^*\Ss,\Uut \otimes \Oo_\xi(-1))$ and $\Hom_\Xt^{\bullet}({i_2}_*(\pb^*\Uu_2^\vee \otimes \Ll_{-\oa}),\Uut \otimes \Oo_\xi(-1))$. Tensoring~\eqref{eq:defseq_tilde_U} with $\Oo_\xi(-1)$ and applying $\xi_*$ we get $\xi_*(\Uut \otimes \Oo_\xi(-1))=\Ll_{-\ob}[-1]$, and 
\[
	\Hom_\Xt^{\bullet}(\xi^*\pa^*\Ss,\Uut \otimes \Oo_\xi(-1))=\Hr^{\bullet-1}(G_2/B,\pa^*\Ss^\vee \otimes \Ll_{-\ob})=0.
\]
The second group is isomorphic to $\Hom_{G_2/B}^{\bullet}(\pb^*\Uu_2^\vee,\pb^*\Uu_2[-1])=0$. 

(f) For $i=0$, the statement follows from Proposition~\ref{lem:exceptional_triple_on_X}. By (d) and (e), we have $\xi_*(\Sst \otimes \Oo_\xi(-1))=\pb^*\Uu_2 \otimes \Ll_{-\oa}[-1]$ and $\xi_*(\Uut \otimes \Oo_\xi(-1))=\Ll_{-\ob}[-1]$, which gives the statement for $i=1$. Note that $\Uut^\vee = \Uut \otimes \pi^*\Oo_X(1)$ therefore $\Ub^\vee = \Ub \otimes \cO_X(1)$. By Serre duality, we have $\Hr^{\bullet}(X,\Ub \otimes \Oo_X(-i)) = \Hr^{7-\bullet}(X,\Ub \otimes \Oo_X(i-3))^\vee$, proving the result for $\Ub  \otimes \Oo_X(-i)$ and $i=2,3$. Tensoring~\eqref{eq:defseq_tilde_S} with $\Oo_\xi(-i)$ for $i=2,3$, and applying $\xi_*$, we get that $\xi_*(\Sst \otimes \Oo_\xi(-i))$ is isomorphic to the cone of $\pa^*\Ss \otimes \xi_*\Oo_\xi(-i)\to \pb^*\Uu_2^\vee \otimes \Ll_{-\oa} \otimes \Ll_{-i\ob}$. Since ${\pb}_*\Ll_{-\oa}=0$, we have $\Hr^{\bullet}(G_2/B,\pb^*\Uu_2^\vee \otimes \Ll_{-\oa} \otimes \Ll_{-i\ob})=0$. Taking into account the fact that $\xi_*\Oo_\xi(-i)={\sf S}^{i-2}(\Ll_{-\oa}\oplus \Ll_{-\ob}) \otimes \Ll_{-\rho}[-1]$, we are reduced to showing the vanishing of the cohomology of $\pa^*\Ss \otimes \Ll_{-\rho}$ and of $\pa^*\Ss \otimes (\Ll_{-\oa}\oplus \Ll_{-\ob}) \otimes \Ll_{-\rho}$. We have $\Hr^{\bullet}(G_2/B,\pa^*\Ss \otimes \Ll_{-\rho})=\Hr^{\bullet}(G_2/B,\pa^*\Ss \otimes \Ll_{-\oa} \otimes \Ll_{-\rho})=0$. Remembering that ${\pa}_*\Ll_{-2\ob}=\Ll_{-3\oa}[-1]$, we compute $\Hr^{\bullet}(G_2/B,\pa^*\Ss \otimes \Ll_{-\ob} \otimes \Ll_{-\rho})=\Hr^{\bullet}(\Qs_5,\Ss  \otimes \Ll_{-4\oa})=\Hr^{5-\bullet}(\Qs_5,\Ss^\vee \otimes \Ll_{-\oa})^\vee=\Hr^{\bullet}(\Qs_5,\Ss)=0$, by Serre duality.

(g) Apply $\Hom_\Xt(-,\Uut \otimes \Oo_\xi(-1))$ to~\eqref{eq:defseq_tilde_U}. We need to compute the group 
\[
	\Hom_\Xt^{\bullet}(\xi^*\pb^*\Uu_2, \Uut \otimes \Oo_\xi(-1))=\Hom_{G_2/B}^{\bullet-1}(\pb^*\Uu_2,\Ll_{-\ob})=0
\]
and the group
\[
	\Hom_\Xt^{\bullet}({i_1}_*\Ll_{\oa-\ob},\Uut \otimes \Oo_\xi(-1))=\Hr^{\bullet -1}(G_2/B,\Ll_{-\oa} \oplus \Ll_{-2\oa})=0.
\] 

(h) Tensoring~\eqref{eq:defseq_tilde_S} with $\Oo_\xi(-2)$ and applying $\xi_*$, we get
\[
  0\to \pb^*\Uu_2^\vee \otimes \Ll_{-\oa-2\ob}\to \Ri^1\xi_*(\Sst \otimes \Oo_\xi(-2))\to \pa^*\Ss \otimes \Ll_{-\rho}\to  0.
\]
The extension is trivial as $\Ext^1_{G_2/B}(\pa^*\Ss,\pb^*\Uu_2)=0$. Thus, 
\[
	\xi_*(\Sst \otimes \Oo_\xi(-2)) = (\pb^*\Uu_2 \otimes \Ll_{-\rho}\oplus \pa^*\Ss \otimes \Ll_{-\rho})[-1].
\]
Applying $\Hom_\Xt(-,\Sst \otimes \Oo_\xi(-2))$ to~\eqref{eq:defseq_tilde_S}, we need to compute $\Hom_\Xt^{\bullet}({i_2}_*(\pb^*\Uu_2^\vee\otimes \Ll_{-\oa}),\Sst \otimes \Oo_\xi(-2))$ and $\Hom_\Xt^{\bullet}(\xi^*\pa^*\Ss,\Sst \otimes \Oo_\xi(-2))=\Hom_{G_2/B}^{\bullet-1}(\pa^*\Ss,\pb^*\Uu_2 \otimes \Ll_{-\rho}\oplus \pa^*\Ss \otimes \Ll_{-\rho})$.

Since ${\pa}_*\Ll_{-\ob}=0$, we get $\Hom_{G_2/B}^{\bullet}(\pa^*\Ss,\pa^*\Ss \otimes \Ll_{-\rho})=0$. From~\eqref{eq:seqforU}, we get that ${\pa}_*(\pb^*\Uu_2 \otimes \Ll_{-\ob})=\Ll_{-2\oa}[-1]$, hence 
\[
	\begin{array}{ll}
		\Hom_{G_2/B}^{\bullet}(\pa^*\Ss,\pb^*\Uu_2 \otimes \Ll_{-\rho})& =\Hom_{G_2/B}^{\bullet}(\pa^*\Ss,\pb^*\Uu_2 \otimes \Ll_{-\ob - \oa}) \\
		& =\Hom_{\Qs_5}^{\bullet-1}(\Ss,\Ll_{-3\oa})=0.\\
	\end{array}
\]
Finally, we have $\Hom_\Xt^{\bullet}({i_2}_*(\pb^*\Uu_2^\vee \otimes \Ll_{-\oa}),\Sst \otimes \Oo_\xi(-2)) = \Hom_{G_2/B}^{\bullet -1}(\pb^*\Uu_2^\vee,\pb^*\Uu_2^{\oplus 2} \otimes \Ll_{-\ob})=
\Hr^{\bullet-1}(G_2/B,\pb^*(\Uu_2 \otimes \Uu_2)^{\oplus 2} \otimes \Ll_{-\ob})$. The latter group is isomorphic to 
\[
	\Hr^{\bullet-1}(G_2/B,\Ll_{-2\ob})\oplus \Hr^{\bullet-1}(\Gad,{\sf S}^2\Uu_2 \otimes \Ll_{-\ob})=0
\]
as $\omega_{\Gad}=\Ll_{-3\ob}$ and $\mathrm{S}^2\Uu_2 \otimes \Ll_{-\ob}[-1]={\pb}_*\Ll_{-4\oa}$ that is also acyclic, since $\omega_{\Qs_5}=\Ll_{-5\oa}$.

(i) Tensoring~\eqref{eq:defseq_tilde_U} with $\Oo_\xi(-2)$ and applying $\xi_*$, we get
\[
	0\to \Ll_{-\oa-\ob}\to \Ri^1\xi_*(\Uut \otimes \Oo_\xi(-2)) \to \pb^*\Uu \otimes \Ll_{-\rho}\to  0.
\]
The extension is trivial since $\Ext^1_{G_2/B}(\pb^*\Uu \otimes \Ll_{-\rho},\Ll_{-\rho})=0$. Thus, we have an isomorphism $\xi_*(\Uut \otimes \Oo_\xi(-2))=(\pb^*\Uu \otimes \Ll_{-\rho}\oplus \Ll_{-\rho})[-1]$. Applying $\Hom_\Xt(-,\Uut \otimes \Oo_\xi(-2))$ to~\eqref{eq:defseq_tilde_U}, we get that the group $\Hom_\Xt^{\bullet}({i_1}_*\Ll_{\oa-\ob},\Uut \otimes \Oo_\xi(-2))$ equals to 
$$\Hom_{G_2/B}^{\bullet -1}(\Ll_{\oa-\ob},(\Oo_{G_2/B}\oplus \Ll_{-\oa}) \otimes \Ll_{-2\oa} \otimes \Ll_{\oa-\ob}) = 0 \quad \textrm{and}$$
$$\Hom_\Xt^{\bullet}(\xi^*\pb^*\Uu_2,\Uut \otimes \Oo_\xi(-2)) =\Hom_{G_2/B}^{\bullet-1}(\pb^*\Uu_2,\pb^*\Uu \otimes \Ll_{-\rho}\oplus \Ll_{-\rho}) =0.$$

(j) Apply $\Hom_\Xt(-,\Sst \otimes \Oo_\xi(-2))$ to~\eqref{eq:defseq_tilde_U}. We need to compute the groups 
\[
	\Hom_\Xt^{\bullet}(\xi^*\pb^*\Uu_2,\Sst \otimes \Oo_\xi(-2)) = \Hom_{{G_2/B}}^{\bullet-1}(\pb^*\Uu_2,\pb^*\Uu_2 \otimes \Ll_{-\rho}\oplus \pa^*\Ss \otimes \Ll_{-\rho})
\]
and 
\[
	\Hom_\Xt^{\bullet}({i_1}_*\Ll_{\oa-\ob},\Sst \otimes \Oo_\xi(-2)) = \Hr^{\bullet -1}(E_1,\pa^*\Ss \otimes \Ll_{-2\oa}).
\]
We have $\Hom_{G_2/B}^{\bullet}(\pb^*\Uu_2,\pb^*\Uu_2 \otimes \Ll_{-\rho})=0$ as in (i), and $\Hom_{G_2/B}^{\bullet}(\pb^*\Uu_2,\pa^*\Ss  \otimes \Ll_{-\rho})=0$ by applying $\Hom_{G_2/B}(-,\pa^*\Ss \otimes \Ll_{-\rho})$ to~\eqref{eq:seqforU}. Moreover, we have $\Hr^{\bullet}(G_2/B,\pa^*\Ss \otimes \Ll_{-2\oa})=0$.

For (k)--(o), use Serre duality and the previous results:

\begin{itemize}
	\item $\Hom_X^{\bullet}(\Sb,\Ub \otimes \Oo_X(-2)))=\Hom_X^{\bullet}(\Ub,\Sb \otimes \Oo_X(-2)[7])^\vee=0$ by (j);
	\item $\Hom_X^{\bullet}(\Sb,\Ub \otimes \Oo_X(-3)))=\Hom_X^{\bullet}(\Ub,\Sb \otimes \Oo_X(-1)[7])^\vee=0$ by (c);
	\item $\Hom_X^{\bullet}(\Sb,\Sb \otimes \Oo_X(-3))= \Hom_X^{\bullet}(\Sb,\Sb \otimes \Oo_X(-1)[7])^\vee = 0$ by (d);
	\item $\Hom_X^{\bullet}(\Ub,\Sb \otimes \Oo_X(-3))=\Hom_X^{\bullet}(\Sb,\Ub \otimes \Oo_X(-1)[7])^\vee=0$ by (e);
	\item $\Hom_X^{\bullet}(\Ub,\Ub \otimes \Oo_X(-3)))=
  \Hom_X^{\bullet}(\Ub,\Ub \otimes \Oo_X(-1))[7])^\vee=0$ by (g).
\end{itemize}
This completes the proof.
\end{proof}

\begin{cor}
	The sequence $\A = \scal{\A_0(-3), \A_0(-2), \A_0(-1), \A_0}\subset \Dd ^b(X)$ is semiorthogonal.
\end{cor}

\begin{proof}
By Proposition \ref{lem:exceptional_triple_on_X}, each block $\A_0(-i)$ for $i=0,1,2,3$ is generated by an exceptional sequence $\scal{\Ub \otimes \Oo_X(-i),\Sbh \otimes \Oo_X(-i), \Oo_X(-i)}$. By the Kodaira vanishing, we have $\Hom _X^{\bullet}(\Oo _X,\Oo _X(-i))=\Hr^{\bullet}(X,\Oo _X(-i))=0$ for $i=1,2,3$, the index of $X$ being equal to 4. Next, using Lemma~\ref{lem:exceptional_lemma} and triangle~\eqref{eq:cone_tilde_S_to_tilde_U}, we obtain the following vanishing of $\Hom$--groups:

\begin{enumerate}[(1)]
	\item $\Hom_X^{\bullet}(\Sbh,\Oo_X(-i)) = 0$  for $i=1,2,3$, using 
(a) and (b) of Lemma~\ref{lem:exceptional_lemma}.
	\item $\Hom_X^{\bullet}(\Oo_X,\Sbh(-i))=0$ for $i=0,1,2,3$, using (f) of Lemma~\ref{lem:exceptional_lemma}.
	\item $\Hom_X^{\bullet}(\Sbh,\Ub(-i))=0$ for $i=1,2,3$, using (c), (g), (i), (k), (l), and (o) of Lemma~\ref{lem:exceptional_lemma}.
	\item $\Hom_X^{\bullet}(\Sbh,\Sbh(-i))=0$ for $i=1,2,3$, using (c), (d), (h), (j), (l), and (m) of Lemma~\ref{lem:exceptional_lemma}.\qedhere
\end{enumerate}
\end{proof}

The above statements imply:

\begin{thm}\label{th:Lefexccoll}
	The sequence ~\eqref{eq:Lefschetz_exc_coll_G_2hor} forms a rectangular Lefschetz exceptional collection on $X$ whose length is equal to the rank of $K^0(X)$.
\end{thm}

\subsection{Fullness of the exceptional collection in Theorem~\ref{th:Lefexccoll}}

\begin{thm}
	The semiorthogonal sequence $\A = \scal{\A_0(-3), \A_0(-2), \A_0(-1), \A_0}$ is a rectangular Lefschetz semiorthogonal decomposition of $\Dd^b(X)$ of length 4 with the starting block $\A_0$. 
\end{thm}

To unburden the notation, in what follows we will suppress the pullback symbol $\xi^*(-) \in \Dd^b(\Xt)$ at an object $(-) \in \Dd^b(G_2/B)$, and whenever it doesn't lead to confusion we will also be suppressing the pullback symbols $\pa^*$ and $\pb^*$ at objects of $\Dd^b(\Qs_5)$ and $\Dd^b(\Gad)$, respectively. Thus, typically an object such as $\Uu_2$ means $\xi^*\pb^*\Uu_2$ on $\Xt$, or $\pb^*\Uu_2$ on $G_2/B$, depending on the variety the object is being considered.

\begin{proof}
The proof is a bit involved computationally, but conceptually it is rather straightforward. We start with finding a special semiorthogonal decomposition of $\Dd ^b(\Xt)$. Assuming then that the semiorthogonal sequence $\A$ is not a semiorthogonal decomposition of $\Dd ^b(X)$, that is, that there exists a nonzero object $E\in {^\perp\A}$, we show in Lemma \ref{lem:reductionlemma} that orthogonality conditions imposed on $E$ force the pullback $\pi ^{\ast}E$ to belong to a subcategory $\B$ of $\Dd ^b(\Xt)$ that is generated by three exceptional objects (see collection (\ref{eq:step_7}) and Lemma \ref{lem:reductionlemma} below). The category $\B$ is rather manageable, and 
in this last step we show that the question of fullness of the sequence $\A$ on the triangulated level can be reduced to the level of the group $K^0(X)$ where it is already proven by the above Theorem \ref{th:Lefexccoll}.

Thus, our first step is to find the aforementioned special semiorthogonal decomposition of $\Dd ^b(\Xt)$. By \cite[6.4]{Ku}, the category $\Dd^b(\Gad)$ has a full exceptional collection 
\[
	\scal{\Uu_2 \otimes \Ll_{-2\ob}, \Ll_{-2\ob}, \Uu_2 \otimes \Ll_{-\ob}, \Ll_{-\ob}, \Uu_2, \Oo_{\Gad}}.
\]

\quad By \cite{Or}, there is a semiorthogonal decomposition $\Dd^b(G_2/B)= \scal{\pb^*\Dd^b(\Gad) \otimes \Ll_{-\oa},\pb^*\Dd^b(\Gad)}$, and combining these two statements we obtain a full exceptional collection in $\Dd^b(G_2/B)$:

\begin{eqnarray}\label{eq:G_2/B_exc_coll_end}
& \Uu _2\otimes \Ll _{-2\omega _{\beta}}, \ \Ll _{-2\omega _{\beta}}, \ \Uu _2\otimes \Ll _{-\omega _{\beta}}, \ \Ll _{-\omega _{\beta}}, \ \Uu _2, \ \Oo _{\tilde X}\rangle \\ \nonumber 
& \langle \Uu _2\otimes \Ll _{-2\omega _{\beta}-\omega _{\alpha}}, \ \Ll _{-2\omega _{\beta}-\omega _{\alpha}},\Uu _2\otimes \Ll _{-\rho},\ \Ll _{-\rho},\ \Uu _2\otimes \Ll _{-\omega _{\alpha}}, \ \Ll _{-\omega _{\alpha}},
\end{eqnarray}

When it comes to an exceptional collection as above, our convention is that it is ordered as usual for objects in the same line, and the order on lines is from top to bottom. Thus, in the above example the objects of the upper line are all left orthogonal to the objects of the lower line, and each of the two lines forms an exceptional collection. The brackets $''\langle''$ and $''\rangle''$ indicate the margins of a collection.

\ Remembering that $\xi$ is a $\P^1$-bundle and using once again \cite{Or}, we obtain a semiorthogonal decomposition $\Dd^b(\Xt)= \scal{\xi^*\Dd^b(G_2/B) \otimes \Oo_\xi(-1),\xi^*\Dd^b(G_2/B)}$. Taking the collection \eqref{eq:G_2/B_exc_coll_end} in $\Dd^b(G_2/B)$ and tensoring it with $\Oo_\xi(-1)$, we obtain a full exceptional collection in $\Dd^b(\Xt)$. We first perform a series of mutations of the collection thus obtained inside $\Dd^b(\Xt) = \scal{\xi^*\Dd^b(G_2/B) \otimes \Oo_\xi(-1),\xi^*\Dd^b(G_2/B)}$. Throughout, we use Lemma~\ref{lem:mutations_canonical_class}.

\subsubsection*{Step 1} Starting from decomposition \eqref{eq:G_2/B_exc_coll_end} and recalling that $\omega^{-1}_{\Gad}=\Ll_{3\ob}$, we mutate $\Uu_2 \otimes \Ll_{-2\ob}$ to the right inside $\pb^*\Dd^b(\Gad)$, obtaining a full exceptional collection in $\xi^*\Dd^b(G_2/B)=\xi^* \scal{\pb^*\Dd^b(\Gad) \otimes \Ll_{-\oa}, \pb^*\Dd^b(\Gad)}$:

\begin{eqnarray}\label{eq:step_1}
&  \Ll _{-2\omega _{\beta}}, \ \Uu _2\otimes \Ll _{-\omega _{\beta}}, \ \Ll _{-\omega _{\beta}}, \ \Uu _2, \ \Oo _{\tilde X}, \ \Uu _2^{\vee} \rangle \\ \nonumber 
& \langle \Ll _{-2\omega _{\beta}-\omega _{\alpha}}, \ \Uu _2\otimes \Ll _{-\rho},\ \Ll _{-\rho},\ \Uu _2\otimes \Ll _{-\omega _{\alpha}}, \ \Ll _{-\omega _{\alpha}},\ \Uu _2^{\vee}\otimes \Ll _{-\omega _{\alpha}},
\end{eqnarray}

\subsubsection*{Step 2} Recall that $\omega^{-1}_{G_2/B}=\Ll_{2\rho}$ and mutate $\Ll_{-2\ob-\oa}$ to the right in $\xi^*\Dd^b(G_2/B)$, we obtain

\begin{eqnarray}\label{eq:step_2}
&  \Ll _{-2\omega _{\beta}}, \ \Uu _2\otimes \Ll _{-\omega _{\beta}}, \ \Ll _{-\omega _{\beta}}, \ \Uu _2, \ \Oo _{\tilde X}, \ \Uu _2^{\vee}, \ \Ll _{\omega _{\alpha}}\rangle \\ \nonumber 
& \langle \Uu _2\otimes \Ll _{-\rho},\ \Ll _{-\rho},\ \Uu _2\otimes \Ll _{-\omega _{\alpha}}, \ \Ll _{-\omega _{\alpha}},\ \Uu _2^{\vee}\otimes \Ll _{-\omega _{\alpha}},
\end{eqnarray}

\subsubsection*{Step 3} Consider $\Dd^b(\Xt) = \scal{\xi^*\Dd^b(G_2/B) \otimes \Oo_\xi(-1),\xi^*\Dd^b(G_2/B)}$, and let the full exceptional collection in $\xi^*\Dd^b(G_2/B)$ be the one from~\eqref{eq:step_2}, and the collection in $\xi^*\Dd^b(G_2/B) \otimes \Oo_\xi(-1)$ be the one from \eqref{eq:step_1} tensored with $\Oo_\xi(-1)$. 

\subsubsection*{Step 4} 
In $\xi^*\Dd^b(G_2/B) \otimes \Oo_\xi(-1)$, consider $\Ll_{-2\ob} \otimes \Oo_\xi(-1)\in \xi^*\pb^*\Dd^b(\Gad)\otimes \Oo_\xi(-1)$ and mutate it to the right inside $\xi^*\pb^*\Dd^b(\Gad) \otimes \Oo_\xi(-1)\subset \xi^*\Dd^b(G_2/B) \otimes \Oo_\xi(-1)$, obtaining $\Ll_{\ob} \otimes \Oo_\xi(-1)$ and a full exceptional collection in $\xi^*\Dd^b(G_2/B) \otimes \Oo_\xi(-1)$:

\begin{eqnarray}\label{eq:step_4}
&  \Ll _{-\omega _{\beta}}\otimes \Oo _{\xi}(-1), \Uu _2\otimes \Oo _{\xi}(-1), \Oo _{\xi}(-1), \Uu _2^{\vee}\otimes \Oo _{\xi}(-1), \Ll _{\omega _{\beta}}\otimes \Oo _{\xi}(-1)\rangle \\
& \Uu _2\otimes \Ll _{-\omega _{\alpha}}\otimes \Oo _{\xi}(-1), \Ll _{-\omega _{\alpha}}\otimes \Oo _{\xi}(-1), \Uu _2^{\vee}\otimes \Ll _{-\omega _{\alpha}}\otimes \Oo _{\xi}(-1), \Uu _2\otimes \Ll _{-\omega _{\beta}}\otimes \Oo _{\xi}(-1), \nonumber \\
& \langle \Ll _{-2\omega _{\beta}-\omega _{\alpha}}\otimes \Oo _{\xi}(-1),\Uu _2\otimes \Ll _{-\rho}\otimes \Oo _{\xi}(-1), \Ll _{-\rho}\otimes \Oo _{\xi}(-1), \nonumber  
\end{eqnarray}

\subsubsection*{Step 5} In \eqref{eq:step_4}, mutate $\Ll_{\ob} \otimes \Oo_\xi(-1)$ to the left inside $\xi^*\Dd^b(G_2/B) \otimes \Oo_\xi(-1)$, obtaining $\Ll_{-\ob-2\oa} \otimes \Oo_\xi(-1)$ and  a full exceptional collection in $\xi^*\Dd^b(G_2/B) \otimes \Oo_\xi(-1)$:

\begin{eqnarray}\label{eq:step_5}
&  \Uu _2\otimes \Ll _{-\omega _{\beta}}\otimes \Oo _{\xi}(-1), \Ll _{-\omega _{\beta}}\otimes \Oo _{\xi}(-1), \Uu _2\otimes \Oo _{\xi}(-1), \Oo _{\xi}(-1), \Uu _2^{\vee}\otimes \Oo _{\xi}(-1)\rangle  \nonumber \\
& \Ll _{-\rho}\otimes \Oo _{\xi}(-1), \Uu _2\otimes \Ll _{-\omega _{\alpha}}\otimes \Oo _{\xi}(-1), \Ll _{-\omega _{\alpha}}\otimes \Oo _{\xi}(-1), \Uu _2^{\vee}\otimes \Ll _{-\omega _{\alpha}}\otimes \Oo _{\xi}(-1), \nonumber \\
& \langle \Ll _{-\omega _{\beta}-2\omega _{\alpha}}\otimes \Oo _{\xi}(-1), \Ll _{-2\omega _{\beta}-\omega _{\alpha}}\otimes \Oo _{\xi}(-1),\Uu _2\otimes \Ll _{-\rho}\otimes \Oo _{\xi}(-1),
\end{eqnarray}

Note that $\Hom^{\bullet}_{\Xt}(\Ll_{-\ob-2\oa} \otimes \Oo_\xi(-1),\Ll_{-2\ob-\oa} \otimes \Oo_\xi(-1))=\Hr^{\bullet}(G_2/B,\Ll_{\oa-\ob})=0$ and also $\Hom^{\bullet}_{\Xt}(\Ll_{-2\ob-\oa} \otimes \Oo_\xi(-1),\Ll_{-\ob-2\oa} \otimes \Oo_\xi(-1))=\Hr^{\bullet}(G_2/B,\Ll_{\ob-\oa})=0$. Thus, mutating $\Ll_{-2\ob-\oa} \otimes \Oo_\xi(-1)$ in \eqref{eq:step_5} to the left through $\Ll_{-\ob-2\oa} \otimes \Oo_\xi(-1)$ amounts to just interchanging these two objects. We obtain 

\begin{eqnarray}\label{eq:step_5-1/2}
&  \Uu _2\otimes \Ll _{-\omega _{\beta}}\otimes \Oo _{\xi}(-1), \Ll _{-\omega _{\beta}}\otimes \Oo _{\xi}(-1), \Uu _2\otimes \Oo _{\xi}(-1), \Oo _{\xi}(-1), \Uu _2^{\vee}\otimes \Oo _{\xi}(-1)\rangle  \nonumber \\
& \Ll _{-\rho}\otimes \Oo _{\xi}(-1), \Uu _2\otimes \Ll _{-\omega _{\alpha}}\otimes \Oo _{\xi}(-1), \Ll _{-\omega _{\alpha}}\otimes \Oo _{\xi}(-1), \Uu _2^{\vee}\otimes \Ll _{-\omega _{\alpha}}\otimes \Oo _{\xi}(-1), \nonumber \\
& \langle \Ll _{-2\omega _{\beta}-\omega _{\alpha}}\otimes \Oo _{\xi}(-1),\Ll _{-\omega _{\beta}-2\omega _{\alpha}}\otimes \Oo _{\xi}(-1), \Uu _2\otimes \Ll _{-\rho}\otimes \Oo _{\xi}(-1),
\end{eqnarray}

\subsubsection*{Step 6} In~\eqref{eq:step_5-1/2}, we mutate $\Ll_{-2\ob-\oa} \otimes \Oo_\xi(-1)$ to the right inside $\xi^*\Dd^b(G_2/B) \otimes \Oo_\xi(-1)$, obtaining $\Ll_{\oa} \otimes \Oo_\xi(-1)$ and a full exceptional collection in $\xi^*\Dd^b(G_2/B) \otimes \Oo_\xi(-1)$:

\begin{eqnarray}\label{eq:step_6}
&  \Ll _{-\omega _{\beta}}\otimes \Oo _{\xi}(-1), \Uu _2\otimes \Oo _{\xi}(-1), \Oo _{\xi}(-1), \Uu _2^{\vee}\otimes \Oo _{\xi}(-1), \Ll _{\omega _{\alpha}}\otimes \Oo _{\xi}(-1)\rangle  \\
&  \Uu _2\otimes \Ll _{-\omega _{\alpha}}\otimes \Oo _{\xi}(-1), \Ll _{-\omega _{\alpha}}\otimes \Oo _{\xi}(-1), \Uu _2^{\vee}\otimes \Ll _{-\omega _{\alpha}}\otimes \Oo _{\xi}(-1), \Uu _2\otimes \Ll _{-\omega _{\beta}}\otimes \Oo _{\xi}(-1), \nonumber \\
& \langle \Ll _{-\omega _{\beta}-2\omega _{\alpha}}\otimes \Oo _{\xi}(-1), \Uu _2\otimes \Ll _{-\rho}\otimes \Oo _{\xi}(-1), \Ll _{-\rho}\otimes \Oo _{\xi}(-1), \nonumber 
\end{eqnarray}

\subsubsection*{Step 7} Finally, recalling that $\omega^{-1}_\Xt =\xi^*\Ll_{\rho} \otimes \Oo_\xi(2)$ we mutate $\Ll_{-\ob-2\oa} \otimes \Oo_\xi(-1)\in \Dd^b({\Xt})$ to the right inside the whole of $\Dd^b({\Xt})$, obtaining $\Ll_{-\ob-2\oa} \otimes \Oo_\xi(-1) \otimes \xi^*\Ll_{\rho} \otimes \Oo_\xi(2)=\Ll_{-\oa} \otimes \Oo_\xi(1)=\Oo_\Xt(E_2)$ and a full exceptional collection in $\Dd^b(\Xt)$:

\begin{eqnarray}\label{eq:step_7}
&  \Ll _{-2\omega _{\beta}}, \ \Uu _2\otimes \Ll _{-\omega _{\beta}}, \ \Ll _{-\omega _{\beta}}, \ \Uu _2, \ \Oo _{\tilde X}, \ \Uu _2^{\vee}, \ \Ll _{\omega _{\alpha}}, \Oo _{\tilde X}(E_2)\rangle \\ \nonumber 
&  \Uu _2^{\vee}\otimes \Oo _{\xi}(-1), \Ll _{\omega _{\alpha}}\otimes \Oo _{\xi}(-1), \Uu _2\otimes \Ll _{-\rho},\ \Ll _{-\rho},\ \Uu _2\otimes \Ll _{-\omega _{\alpha}}, \ \Ll _{-\omega _{\alpha}},\ \Uu _2^{\vee}\otimes \Ll _{-\omega _{\alpha}}, \\
&  \Uu _2^{\vee}\otimes \Ll _{-\omega _{\alpha}}\otimes \Oo _{\xi}(-1), \Uu _2\otimes \Ll _{-\omega _{\beta}}\otimes \Oo _{\xi}(-1), \Ll _{-\omega _{\beta}}\otimes \Oo _{\xi}(-1), \Uu _2\otimes \Oo _{\xi}(-1), \Oo _{\xi}(-1),  \nonumber \\
& \langle \Uu _2\otimes \Ll _{-\rho}\otimes \Oo _{\xi}(-1),\Ll _{-\rho}\otimes \Oo _{\xi}(-1), \Uu _2\otimes \Ll _{-\omega _{\alpha}}\otimes \Oo _{\xi}(-1), \Ll _{-\omega _{\alpha}}\otimes \Oo _{\xi}(-1),  \nonumber
\end{eqnarray}

We now assume that the semiorthogonal sequence $\A$ is not a semiorthogonal decomposition, that is, that there exists a nonzero object $E\in {^\perp\A}$. 

\begin{lemma}\label{lem:reductionlemma}
Let $\B$ denote the admissible subcategory $\scal{ \Uu_2^{\vee}, \Ll_{\oa}, \Oo_\Xt(E_2)} \subset \Dd^b(\Xt)$ of~\eqref{eq:step_7}. Consider a non--zero object $E\in {^\perp\A}$. Then $\pi^*E\in \B$. 
\end{lemma}

\begin{proof}
Since $\B$ is admissible, we have a semiorthogonal decomposition $\Dd^b(\Xt) = \scal{\B^\perp,\B}$. The statement of lemma amounts to showing that $\pi^*E$ is left orthogonal to $\B^\perp$, that is to all the objects to the left from $\scal{\Uu_2^{\vee}, \Ll_{\oa}, \Oo_\Xt(E_2)}$ in (\ref{eq:step_7}). We break up the proof into several steps. We first recall the basic short exact sequences that will be used throughout:

\begin{equation}\label{eq:seqforU-1}
	0\to \Ll_{-\oa}\to \Uu_2\to \Ll_{\oa-\ob}\to 0.
\end{equation}

\begin{equation}\label{eq:defseq_tilde_U-1}
	0\to \Uut \to \Uu_2\to {i_1}_{\ast}\Ll_{\oa-\ob}\to 0,
\end{equation}

\begin{equation}\label{eq:spinor_U_2_sequence-1}
	0\to \Uu_2\to \Ss\to \Uu_2^* \otimes \Ll_{-\oa}\to 0.
\end{equation}

\begin{equation}\label{eq:defseq_tilde_S-1}
	0\to \Sst \to \Ss\to {i_2}_{*}(\Uu_2^* \otimes \Ll_{-\oa})\to 0.
\end{equation}

Finally, recall also the distinguished triangle that defines the object $\Ssh$ (cf. (\ref{eq:cone_tilde_S_to_tilde_U})):

\begin{equation}\label{eq:cone_tilde_S_to_tilde_U-1}
\dots \to \Sst\to \Uut \to \Ssh \to \Sst[1] \to \dots 
\end{equation}

Denote $\At_0$ a full triangulated subcategory of $\Dd^b(X)$ generated by 
$\scal{\Sb, \Ub, \Oo_X}$, and let $\At$ denote the minimal full triangulated subcategory containing the collection $\scal{ \At_0 \otimes \Oo_X(-3), \At_0 \otimes \Oo_X(-2), \At_0 \otimes \Oo_X(-1), \At_0} \subset \Dd^b(X)$. Then, by virtue of (\ref{eq:cone_tilde_S_to_tilde_U-1}), we have $\At = \A$. 
Given a non--zero object $E \in {^\perp\A}$, we see that $E\in {^\perp\At}$,
and taking the pull--back $\pi^*E$, we see that $\pi^*E\in {^\perp{\pi^{\ast}\At}}$ (and conversely). We show that the latter orthogonality condition implies that $\pi^*E\in \B$.

\comment{
the left orthogonality of $E$ to  
the admissible subcategory $\A\subset \Dd ^b(X)$ is equivalent to the left orthogonality of $E$ to $\in {^\perp\A}$
\\
the triangulated subcategory $\pi^*\A_0= \scal{\Ssh, \Uut ,\Oo_\Xt} \subset \Dd^b(\Xt)$ is equivalent to a subcategory $\pi^*\At_0\subset \Dd^b(\Xt)$, where $\At_0\subset \Dd^b(X)$ is a subcategory generated by 
$\scal{\Sbt, \Ub, \Oo_X}$ (recall Definition~\ref{def:def_hat_S} of the object $\Sbh$). Write $\At = \scal{ \At_0 \otimes \Oo_X(-3), \At_0 \otimes \Oo_X(-2), \At_0 \otimes \Oo_X(-1), \At_0} \subset \Dd^b(X)$; we have $\At = \A$ and $E\in ^\perp\At, \pi^*E\in \pi^* {^\perp\At}$. We show that the latter orthogonality condition implies that $\pi^*E\in \B$.
}

\subsubsection*{Step 1}
We prove the inclusion
$$\pi^*E \in {^\perp \scal{ \pi^*\At ,\pi^*\At  \otimes \Oo_\Xt(E_1),\pi^*\At  \otimes \Oo_\Xt(E_2),\pi^*\At  \otimes \Oo_\Xt(E_1+E_2)}}.$$ Indeed, isomorphisms $\pi_*\Oo_\Xt({\rm E}_1)=\pi_*\Oo_\Xt({\rm E}_2)=\pi_*\Oo_\Xt({\rm E}_1+{\rm E}_2) = \Oo_X$ imply that $\Hom^{\bullet}_{\Xt}(\pi^*E,\pi^*\At  \otimes \Oo_\Xt({\rm E}_1))=\Hom^{\bullet}_{\Xt}(\pi^*E,\pi^*\At  \otimes \Oo_\Xt({\rm E}_2))=\Hom^{\bullet}_{\Xt}(\pi^*E,\pi^*\At  \otimes \Oo_\Xt({\rm E}_1+{\rm E}_2))=\Hom^{\bullet}_X(E,\At)=0$.

\subsubsection*{Step 2}

Taking $\Oo_\xi(-i)=\pi^*\Oo_X(-i)$ for $i=0,1,2,3$ and tensoring it with $\Oo_X(E_1)=\Ll_{-\ob} \otimes \Oo_\xi(1)$ and $\Oo_X(E_2)=\Ll_{-\oa} \otimes \Oo_\xi(1)$ (resp., with  $\Oo_X(E_1+E_2)=\Ll_{-\rho} \otimes \Oo_\xi(2)$), we have $\pi^*E \in ^\perp\scal{ \Oo_\xi(-i)}$ for $i=0,1,2,3$, $\pi^*E\in ^\perp\scal{ \Ll_{-\oa} \otimes \Oo_\xi(-i), \Ll_{-\ob} \otimes \Oo_\xi(-i)}$ for $ i=-1,0,1,2$, $\pi^*E \in ^\perp\scal{\Ll_{-\rho} \otimes \Oo_\xi(-i)}$ for $i=-2,-1,0,1$.

Taking into account that $\pi_*{i_1}_*\Ll_{\oa-\ob}=\pi_*{i_2}_*(\Uu_2^{\vee} \otimes \Ll_{-\oa})=0$ we see from~\eqref{eq:defseq_tilde_U-1} and~\eqref{eq:defseq_tilde_S-1} that $\pi^*E \in {^\perp\scal{ \Uu_2, \Ss}}$.
Indeed, applying $\Hom _{\tilde X}(\pi ^{\ast}E,-)$ to the first sequence ~\eqref{eq:defseq_tilde_U-1}, we obtain an isomorphism $\Hom _{\tilde X}(\pi ^{\ast}E,{\tilde \Uu})=\Hom _{\tilde X}(\pi ^{\ast}E,\Uu _2)$, since $\Hom _{\tilde X}(\pi ^{\ast}E,{i_1}_{\ast}\Ll _{\omega _{\alpha}-\omega _{\beta}})=\Hom _{\tilde X}(E,\pi _{\ast}{i_1}_{\ast}\Ll _{\omega _{\alpha}-\omega _{\beta}})=0$, the last isomorphism being a consequence of $\pi _{\ast}{i_1}_{\ast}\Ll _{\omega _{\alpha}-\omega _{\beta}}=0$. Now $\Hom _{\tilde X}(\pi ^{\ast}E,{\tilde \Uu})=0$ by the assumption on $E$: it is supposed to be a non--trivial object of the left orthogonal to $\mathcal A$ and $\tilde \Uu=\pi ^{\ast}\Ub\in \pi^{\ast}\mathcal A$. It follows that $\Hom _{\tilde X}(\pi ^{\ast}E,\Uu _2)=0$, as claimed.
 
Tensoring~\eqref{eq:defseq_tilde_U-1} and~\eqref{eq:defseq_tilde_S-1} with $\Oo_\xi(-i)$ for $i=1,2,3$, we obtain, respectively

\begin{equation}\label{eq:defseq_tilde_U_otimesO_rho(-i)}
	0\to \Uut  \otimes \Oo_\xi(-i)\to \Uu_2  \otimes \Oo_\xi(-i)\to {i_1}_*\Ll_{(1-i)\oa-\ob}\to 0,
\end{equation}

and
\begin{equation}\label{eq:defseq_tilde_S_otimesO_rho(-i)}
0\to \Sst  \otimes \Oo_\xi(-i)\to \Ss  \otimes \Oo_\xi(-i)\to {i_2}_*(\Uu_2^{\vee} \otimes \Ll_{-i\ob-\oa})\to 0.
\end{equation}

Taking into account that $\pi_*{i_1}_*\Ll_{(1-i)\oa-\ob}=\pi_*{i_2}_*(\Uu_2^{\vee} \otimes \Ll_{-i\ob-\oa})=0$ we obtain $\pi^*E \in {^\perp\scal{ \Uu_2 \otimes \Oo_\xi(-i),\Ss \otimes \Oo_\xi(-i)}}$ for $i=1,2,3$. Tensoring~\eqref{eq:spinor_U_2_sequence-1} with $\Oo_\xi(-i)$ for $i=0,1,2,3$ we obtain $\pi^*E \in ^\perp\scal{ \Uu_2^{\vee} \otimes \Ll_{-\oa} \otimes \Oo_\xi(-i)}$ for $i=0,1,2,3$.


Tensoring~\eqref{eq:defseq_tilde_U_otimesO_rho(-i)} with $\Oo_\Xt(E_2)=\Ll_{-\oa} \otimes \Oo_\xi(1)$ and~\eqref{eq:defseq_tilde_S_otimesO_rho(-i)} with $\Oo_\Xt(E_1)=\Ll_{-\ob} \otimes \Oo_\xi(1)$ we obtain, respectively:

\begin{eqnarray}\label{eq:defseq_tilde_U_otimesE_2}
& 0\rightarrow \tilde \Uu \otimes \Oo _{\xi}(-i)\otimes \Oo _{\tilde X}(E_2)\rightarrow \Uu _2\otimes \Ll _{-\omega _{\alpha}}\otimes \Oo _{\xi}(1-i)\rightarrow \nonumber \\
& {i_1}_{\ast}\Ll _{(1-i)\omega _{\alpha}-\omega _{\beta}}\otimes\Oo _{\tilde X}(E_2)\rightarrow 0,
\end{eqnarray}

and 

\begin{eqnarray}\label{eq:defseq_tilde_S_otimesE_1}
& 0\rightarrow \tilde {\mathcal S}\otimes \Oo _{\xi}(-i)\otimes \Oo _{\tilde X}(E_1)\rightarrow {\mathcal S}\otimes \Ll _{-\omega _{\beta}}\otimes \Oo _{\xi}(1-i) \rightarrow \nonumber \\
& {i_2}_{\ast}(\Uu _2^{\ast}\otimes \Ll _{-\omega _{\alpha}})\otimes \Oo _{\tilde X}(E_1)\rightarrow 0.
\end{eqnarray}

Taking into account isomorphisms ${i_1}_*\Ll_{(1-i)\oa-\ob} \otimes\Oo_\Xt(E_2) = {i_1}_*\Ll_{(1-i)\oa-\ob}$ and ${i_2}_*(\Uu_2^{\vee} \otimes \Ll_{-i\ob-\oa}) \otimes \Oo_\Xt(E_1)={i_2}_*(\Uu_2^{\vee} \otimes \Ll_{-i\ob-\oa})$, we obtain
\[
	\pi^*E \in {^\perp\scal{(\Uu_2 \otimes \Ll_{-\oa}) \otimes \Oo_\xi(1-i), (\Ss \otimes \Ll_{-\ob}) \otimes \Oo_\xi(1-i)}}
\]
for $i=0,1,2,3$. Moreover, tensoring the relative Euler sequence 
\begin{equation}\label{eq:relEulerseq_rho}
	0\to \Oo_\xi(-1)\to \Ll_{-\ob}\oplus \Ll_{-\oa}\to \Ll_{-\rho} \otimes \Oo_\xi(1)\to 0
\end{equation}
with $\Ll_{-\oa} \otimes \Oo_\xi(-1)$, we obtain
\[
0\to \Ll_{-\oa} \otimes \Oo_\xi(-2)\to (\Ll_{-\rho}\oplus \Ll_{-2\oa}) \otimes \Oo_\xi(-1)\to \Ll_{-\rho-\oa}\to 0.
\]

We have $\pi_*(\Ll_{-\oa} \otimes \Oo_\xi(-2))=\pi_*(\Ll_{-\rho} \otimes \Oo_\xi(-1))= \Oo_X(-3)$, while $\pi_*(\Ll_{-2\oa} \otimes \Oo_\xi(-1))$ can be found from the short exact sequence 

\begin{equation}\label{eq:U_2otimes-alpha}
	0\to \Ll_{-2\oa}\to \Uu_2 \otimes \Ll_{-\oa}\to \Ll_{-\ob}\to 0.
\end{equation}
We deduce that $\pi_*\Ll_{-2\oa}=\Cone(\Ub(-1)\to \Oo_X(-1))$. Thus, $\pi^*E\in {^\perp}\scal{ \Ll_{-\rho-\oa}}$. This also implies $\pi^*E \in {^\perp\scal{ \Uu_2 \otimes \Ll_{-\ob} \otimes \Oo_\xi(-i)}}$ for $i=0,1$ via tensoring~\eqref{eq:spinor_U_2_sequence-1} with $\Ll_{-\ob}$.

Tensoring \eqref{eq:relEulerseq_rho} with $\Ll_{\oa}$ we obtain:
\[
	0\to \Ll_{\oa} \otimes \Oo_\xi(-1)\to \Ll_{\oa-\ob}\oplus \Oo_\Xt\to \Ll_{-\ob} \otimes \Oo_\xi(1)\to 0,
\]
Using that $\Ll_{-\ob} \otimes \Oo_\xi(1)=\Oo_\Xt({\rm E}_1)$, while $\Ll_{\oa-\ob}\in \scal{ \Uu_2, \Ll_{-\oa}}$, we also get that $\pi^*E\in {^\perp\scal{ \Ll_{\oa} \otimes \Oo_\xi(-1)}}$. Tensoring~\eqref{eq:relEulerseq_rho} with $\Uu_2 \otimes \Oo_\xi(-i), i=1,2$, we obtain 

\begin{eqnarray}\label{eq:relEulerseq_rho_otimesU_2O_rho(-1)}
& 0\rightarrow \Uu _2\otimes \Oo _{\xi}(-i-1)\rightarrow (\Uu _2\otimes \Ll _{-\omega _{\alpha}}\oplus \Uu _2\otimes \Ll _{-\omega _{\beta}})\otimes \Oo _{\xi}(-i)\rightarrow \nonumber \\
& \Uu _2\otimes \Ll _{-\rho}\otimes \Oo _{\xi}(1-i)\rightarrow 0
\end{eqnarray}

We have seen that $\pi^*E\in {^\perp\scal{ \Uu_2 \otimes \Oo_\xi(-i-1)}}$ and $\pi^*E\in {^\perp\scal{\Uu_2 \otimes \Ll_{-\oa} \otimes \Oo_\xi(-i)}}$ for $ i=1,2$. For this we refer to the lines after~\eqref{eq:defseq_tilde_S_otimesO_rho(-i)} and~\eqref{eq:defseq_tilde_S_otimesE_1}. Tensoring~\eqref{eq:spinor_U_2_sequence-1} with $\Ll_{-\ob} \otimes \Oo_\xi(-i-1)$ and using again the lines after~\eqref{eq:defseq_tilde_S_otimesE_1} we obtain that $\pi^*E\in {^\perp\scal{ \Uu_2 \otimes \Ll_{-\ob}) \otimes \Oo_\xi(-i-1)}}$. Hence $\pi^*E\in {^\perp\scal{ \Uu_2 \otimes \Ll_{-\rho} \otimes \Oo_\xi(1-i)}}$ for $i=1,2$.

We conclude from Steps 1 and 2 that $\pi^*E$ is left orthogonal to the following objects:

\begin{eqnarray}\label{eq:Steps1-2}
&  \Uu _2^{\vee}\otimes \Ll _{-\omega _{\alpha}}, \ \Uu _2\otimes \Ll _{-\omega _{\beta}}, \ \Ll _{-\omega _{\beta}}, \ \Ll _{-\omega _{\alpha}}, \ \Uu _2, \ \Oo _{\tilde X} \rangle. \\
& \Oo _{\xi}(-1), \ \Ll _{\omega _{\alpha}}\otimes \Oo _{\xi}(-1), \ \Uu _2\otimes \Ll _{-\rho}, \Ll _{-\rho}, \ \Uu _2\otimes \Ll _{-\omega _{\alpha}},\nonumber \\
& \Uu _2^{\vee}\otimes \Ll _{-\omega _{\alpha}}\otimes \Oo _{\xi}(-1), \ \Uu _2\otimes \Ll _{-\omega _{\beta}}\otimes \Oo _{\xi}(-1), \ \Ll _{-\omega _{\beta}}\otimes \Oo _{\xi}(-1), \ \Uu _2\otimes \Oo _{\xi}(-1), \nonumber \\
& \langle \Uu _2\otimes \Ll _{-\rho}\otimes \Oo _{\xi}(-1), \ \Ll _{-\rho}\otimes \Oo _{\xi}(-1), \ \Uu _2\otimes \Ll _{-\omega _{\alpha}}\otimes \Oo _{\xi}(-1), \ \Ll _{-\omega _{\alpha}}\otimes \Oo _{\xi}(-1), \nonumber
\end{eqnarray}

\subsubsection*{Step 3} We check that $\pi^*E\in {^\perp\scal{  \Uu_2^{\vee} \otimes \Oo_\xi(-1)}}$. This follows immediately from the short exact sequence representing the vector bundle $\Sst$ as an extension of two vector bundles (cf. the diagram before Proposition~\ref{p:diag}):
\[
	0\to \Uu_2\to \Sst\to \Uu_2^{\vee} \otimes \Oo_\xi(-1)\to 0.
\]

\subsubsection*{Step 4} We check that $\pi^*E\in {^\perp\scal{ \Ll_{-2\ob}}}$. Tensoring~\eqref{eq:seqforU-1} with $\Ll_{-\rho}$, we get 
\[
	0\to \Ll_{-\rho -\oa}\to \Uu_2 \otimes \Ll_{-\rho}\to \Ll_{-2\ob}\to 0.
\]

We have $\pi^*E\in {^\perp\scal{ \Ll_{-\rho -\oa}}}$ (see the line after~\eqref{eq:U_2otimes-alpha}) and $\pi^*E\in {^\perp\scal{ \Uu_2 \otimes \Ll_{-\rho}}}$ (see the paragraph after~\eqref{eq:relEulerseq_rho_otimesU_2O_rho(-1)}). This implies the statement.
\end{proof}

Lemma~\ref{lem:reductionlemma} implies that $\pi^*E\in \B$ belongs to the full subcategory generated by $\scal{ \Uu_2^{\vee}, \Ll_{\oa}, \Oo_\Xt(E_2)}$. We conclude the proof by showing that an object of the form $\pi^*E$ belonging to the admissible subcategory $\scal{ \Uu_2^{\vee}, \Ll_{\oa}, \Oo_\Xt(E_2)}\subset \Dd^b(\Xt)$ is necessarily trivial, i.e. $E=0$. The first step is to mutate the subcategory $\scal{ \Uu_2^{\vee}, \Ll_{\oa}}$ to the right through $\scal{ \Oo_\Xt(E_2)}$.  One computes $\Hom^{\bullet}_{\Xt}(\Ll_{\oa},\Oo_\Xt(E_2))=\Hom^{\bullet}_{\Xt}(\Ll_{\oa}\Ll_{-\oa} \otimes \Oo_\xi(1))=\Hr^{\bullet}(\Xt,\Ll_{-2\oa} \otimes \Oo_\xi(1))=\Hr^{\bullet}(G_2/B,\Ll_{-\oa}\oplus \Ll_{-2\oa+\ob})=\C[-1]$, hence the right mutation $\Rb := \Rb_{\Oo_\Xt(E_2)}(\Ll_{\oa})$ of $\Ll_{\oa}$ through $\scal{ \Oo_\Xt(E_2)}$ is given by 

\begin{equation}\label{eq:lasteq1}
\dots \to \Ll_{\oa}[-1]\to \Oo_\Xt(E_2)\to \Rb \to \Ll_{\oa}\to \dots 
\end{equation}

which is in fact a short exact sequence $0\to \Oo_\Xt(E_2)\to \Rb \to \Ll_{\oa}\to 0$ corresponding to a unique non-trivial extension in $\Ext^1_{\Xt}(\Ll_{\oa},\Oo_\Xt(E_2))=\C$ above. Further, $\Hom_{\Xt}(\Uu_2^{\vee},\Oo_\Xt(E_2))=\Hr^{\bullet}(\Xt,\Uu_2 \otimes \Ll_{-\oa} \otimes \Oo_\xi(1))=\Hr^{\bullet}(G_2/B,\Uu_2\oplus \Uu_2^{\vee} \otimes \Ll_{-\oa})=0$. Thus, the right mutation of $\Uu_2^{\vee}$ through $\Oo_\Xt(E_2)$ is just a transposition, and we arrive at the collection $\scal{ \Oo_\Xt(E_2),\Uu_2^{\vee},\Rb}$.

We have seen that $\Hom^{\bullet}_{\Xt}(\pi^*E,\Oo_\Xt(E_2))=\Hom^{\bullet}_{X}(E,\pi_*\Oo_\Xt(E_2))=0$ which follows from $\pi_*\Oo_\Xt(E_2)$ = $\Oo_X$. Thus, the object $\pi^*E$ fits into an exact triangle 

\[
	\dots \to \Rb \otimes V_1^{\bullet}\to \pi^*E\to \Uu_2^{\vee} \otimes V_2^{\bullet}\to \dots  ,
\]

where $V_1^{\bullet}$ and $V_2^{\bullet}$ are graded vector spaces. Restricting that triangle to $E_1$, then tensoring it with $\Ll_{-\ob}$, and finally applying ${\pi_{\alpha}}_*$, we obtain an exact triangle

\begin{eqnarray}\label{eq:lasteq3}
& \dots \rightarrow {\pi _{\alpha}}_{\ast}(i_1^{\ast}{\bf R}\otimes \Ll _{-\omega _{\beta}})\otimes V_1^{\bullet}\rightarrow {\pi _{\alpha}}_{\ast}(i_1^{\ast}\pi ^{\ast}E\otimes \Ll _{-\omega _{\beta}})\rightarrow \nonumber \\
& {\pi _{\alpha}}_{\ast}(\Uu _2^{\vee}\otimes \Ll _{-\omega _{\beta}})\otimes V_2^{\bullet}\rightarrow \dots 
\end{eqnarray}

The middle term ${\pi_{\alpha}}_*(i_1^*\pi^*E \otimes \Ll_{-\ob})$ is equal to $0$, while the term ${\pi_{\alpha}}_*(i_1^* \Rb \otimes \Ll_{-\ob}) \otimes V_1^{\bullet}$ can be read off of the sequence~\eqref{eq:lasteq1}:

\begin{eqnarray}\label{eq:lasteq4}
& \dots \rightarrow {\pi _{\alpha}}_{\ast}\Ll _{\omega _{\alpha}-\omega _{\beta}}[-1]\rightarrow {\pi _{\alpha}}_{\ast}(\Oo _{E_1}\otimes \Ll _{-\omega _{\beta}})\rightarrow {\pi _{\alpha}}_{\ast}(i_1^{\ast}{\bf R}\otimes \Ll _{-\omega _{\beta}})\rightarrow \nonumber \\
& {\pi _{\alpha}}_{\ast}\Ll _{\omega _{\alpha}-\omega _{\beta}}\rightarrow \dots 
\end{eqnarray}

and since ${\pi_{\alpha}}_*\Ll_{\oa-\ob}={\pi_{\alpha}}_*(\Oo_{E_1} \otimes \Ll_{-\ob})=0$, we obtain ${\pi_{\alpha}}_*(i_1^* \Rb \otimes \Ll_{-\ob})=0$. On the other hand, ${\pi_{\alpha}}_*(\Uu_2^{\vee} \otimes \Ll_{-\ob})=\Ll_{-\oa}$ as can be seen from~\eqref{eq:seqforU-1}. From ~\eqref{eq:lasteq3} we conclude that $V_2^{\bullet}=0$. Thus, $\pi^*E = \Rb \otimes V_1^{\bullet}$. Recall that $\Rb$ is the middle term in the short exact sequence $0\to \Oo _{\tilde X}(E_2)\to \Rb \to \Ll_{\oa}\to 0$, hence $\Rb$ is a locally free sheaf having a non-trivial class in $K^0(\Xt)$. For consistency of calculations, let us check that $\bf R$ belongs to $\pi^*\Dd^b(X)$. Indeed, restricting $\Rb$ to $E_1$ we obtain a short exact sequence $0\to \Oo_{E_1}\to i_1^* \Rb \to \Ll_{\oa}\to 0$, hence $i_1^* \Rb$ is pulled back from $\Qs_5$ and has a trivial restriction to the fibers of $\pa$. Restricting $\Rb$ to $E_2$ we obtain a short exact sequence $0\to \Ll_{\ob-\oa}\to i_2^* \Rb \to \Ll_{\oa}\to 0$, and then $i_2^* \Rb =\Uu_2^{\vee}$, hence $i_2^* \Rb$ is pulled back from $\Gad$ and has a trivial restriction to the fibers of $\pb$.

Let $\R$ denote the vector bundle on $X$ such that $\pi^*\R = \Rb$. Now $E=\R  \otimes V_1^{\bullet} \in {^\perp\A}$, where $\A$ is a full subcategory of $\Dd^b(X)$ generated by a semiorthogonal sequence. We have $\Hom_{X}^{\bullet}(E, \A) =\Hom_{X}^{\bullet}(\R  \otimes V_1^{\bullet}, \A) = V_1^{\bullet}  \otimes \Hom_{X}^{\bullet}(\R,\A)=0$, where the second isomorphism holds since $E= \R  \otimes V_1^{\bullet}\in \Dd^b(X)$, and hence $V_1^{\bullet}$ is a bounded complex of finite-dimensional vector spaces. Assuming that $V_1^{\bullet}\neq 0$, we necessarily obtain $\Hom_{X}^{\bullet}(\R,\A)=0$, but this would contradict the fact that the terms of $\A$ generate the group $K^0(X) \cong \Z^{12}$. Thus $V_1^{\bullet}=0$ and $E=0$.

Therefore, ${^\perp\A}=0$, and the semiorthogonal sequence $\A$ is a semiorthogonal decomposition of $\Dd ^b(X)$.
\end{proof}

\begin{remark}
We finish with a (heuristic so far) remark about how collection (\ref{eq:Lefschetz_exc_coll_G_2hor}) of Theorem \ref{th:Lefexccoll} shows up. By \cite[Proposition 2.3]{BP}, the variety $X^5$ has a deformation to the orthogonal grassmannian ${\sf OGr}(2,7)$. By \cite[Section 7]{Kuzisotropic}, there is a full rectangular Lefschetz exceptional collection on ${\sf OGr}(2,7)$ which consists of four blocks with the starting block being $\langle \Uu ,{\mathcal S},\Oo _{{\sf OGr}(2,7)}\rangle$. Here $\Uu$ is the tautological rank two bundle and ${\mathcal S}$ is the spinor bundle on ${\sf OGr}(2,7)$. 
Exceptional objects can be extended over to the formal neighbourhood of the 
special fiber in a family, \cite[Lemma 2.10]{BK}, as all the obstructions to deformations vanish. We see that the shape of the starting block of collection (\ref{eq:Lefschetz_exc_coll_G_2hor}) agrees with that of the starting block on ${\sf OGr}(2,7)$. We will explore this connection in a subsequent work.
\end{remark}

\comment{
\section{Appendix}

In this appendix, we give the beginning (these are infinite graphs) of the quantum Hasse diagrams in case (1), $n = 3$, case (2), case (3) $n = 3 = m$ and case (5). Each column corresponds to a degree starting from the left with degree zero. The black vertices correspond to the classes $\s'_u$ and where there are two classes of the same degree, the classes are $\s'_u$ and $\s'_{u'}$ and the top class is always $\s'_u$ while the bottom one is $\s'_{u'}$. We use the same convention for blue vertices which correspond to $\tau_v$ classes. Black and blue edges are from the Hasse diagram of $Y$ and $Z$. Grey edges are from the Hasse diagram of $Y$ to the one of $Z$ while dotted edges correspond to quantum multiplication.

\subsection{Case (1), $n = 3$} We give the quantum hasse diagram in case (1) for $n = 3$.

\centerline{\begin{pspicture*}(0,0)(8.4,6.0)

\psellipse[fillstyle=solid,fillcolor=black](1.2,4.8)(0.086,0.086)
\psellipse[fillstyle=solid,fillcolor=black](1.8,4.8)(0.086,0.086)
\psellipse[fillstyle=solid,fillcolor=black](2.4,4.8)(0.086,0.086)
\psellipse[fillstyle=solid,fillcolor=black](2.4,4.2)(0.086,0.086)
\psellipse[fillstyle=solid,fillcolor=black](3.0,4.8)(0.086,0.086)
\psellipse[fillstyle=solid,fillcolor=black](3.0,4.2)(0.086,0.086)
\psellipse[linecolor=blue,fillstyle=solid,fillcolor=blue](3.0,3.6)(0.086,0.086)
\psellipse[fillstyle=solid,fillcolor=black](3.6,4.8)(0.086,0.086)
\psellipse[fillstyle=solid,fillcolor=black](3.6,4.2)(0.086,0.086)
\psellipse[linecolor=blue,fillstyle=solid,fillcolor=blue](3.6,3.6)(0.086,0.086)
\psellipse[fillstyle=solid,fillcolor=black](4.2,4.8)(0.086,0.086)
\psellipse[fillstyle=solid,fillcolor=black](4.2,4.2)(0.086,0.086)
\psellipse[linecolor=blue,fillstyle=solid,fillcolor=blue](4.2,3.6)(0.086,0.086)
\psellipse[fillstyle=solid,fillcolor=black](4.2,3.0)(0.086,0.086)
\psellipse[fillstyle=solid,fillcolor=black](4.8,4.8)(0.086,0.086)
\psellipse[linecolor=blue,fillstyle=solid,fillcolor=blue](4.8,4.2)(0.086,0.086)
\psellipse[linecolor=blue,fillstyle=solid,fillcolor=blue](4.8,3.6)(0.086,0.086)
\psellipse[fillstyle=solid,fillcolor=black](4.8,3.0)(0.086,0.086)
\psellipse[fillstyle=solid,fillcolor=black](5.4,4.8)(0.086,0.086)
\psellipse[linecolor=blue,fillstyle=solid,fillcolor=blue](5.4,4.2)(0.086,0.086)
\psellipse[fillstyle=solid,fillcolor=black](5.4,3.0)(0.086,0.086)
\psellipse[fillstyle=solid,fillcolor=black](5.4,2.4)(0.086,0.086)
\psellipse[linecolor=blue,fillstyle=solid,fillcolor=blue](6.0,4.2)(0.086,0.086)
\psellipse[fillstyle=solid,fillcolor=black](6.0,3.0)(0.086,0.086)
\psellipse[fillstyle=solid,fillcolor=black](6.0,2.4)(0.086,0.086)
\psellipse[linecolor=blue,fillstyle=solid,fillcolor=blue](6.0,1.8)(0.086,0.086)
\psellipse[linecolor=blue,fillstyle=solid,fillcolor=blue](6.6,4.2)(0.086,0.086)
\psellipse[fillstyle=solid,fillcolor=black](6.6,3.0)(0.086,0.086)
\psellipse[fillstyle=solid,fillcolor=black](6.6,2.4)(0.086,0.086)
\psellipse[linecolor=blue,fillstyle=solid,fillcolor=blue](6.6,1.8)(0.086,0.086)
\psellipse[fillstyle=solid,fillcolor=black](7.2,3.0)(0.086,0.086)
\psellipse[fillstyle=solid,fillcolor=black](7.2,2.4)(0.086,0.086)
\psellipse[linecolor=blue,fillstyle=solid,fillcolor=blue](7.2,1.8)(0.086,0.086)
\psellipse[fillstyle=solid,fillcolor=black](7.2,1.2)(0.086,0.086)
\psline(1.286,4.8)(1.714,4.8)
\psline(1.861,4.739)(2.339,4.261)
\psline(2.461,4.739)(2.939,4.261)
\psline(3.061,4.261)(3.539,4.739)
\psline(2.486,4.8)(2.914,4.8)
\psline(3.061,4.739)(3.539,4.261)
\psline(3.661,4.261)(4.139,4.739)
\psline(3.686,4.8)(4.114,4.8)
\psline(4.261,4.261)(4.739,4.739)
\psline(4.886,4.8)(5.314,4.8)
\psline(4.286,3.0)(4.714,3.0)
\psline(4.861,2.939)(5.339,2.461)
\psline(5.486,3.0)(5.914,3.0)
\psline(6.061,2.939)(6.539,2.461)
\psline(6.514,2.4)(6.514,2.4)
\psline(6.661,2.461)(7.139,2.939)
\psline(5.914,2.4)(5.914,2.4)
\psline(5.461,2.939)(5.939,2.461)
\psline(6.061,2.461)(6.539,2.939)
\psline(6.686,3.0)(7.114,3.0)
\psline(1.878,4.764)(2.322,4.764)
\psline(1.878,4.836)(2.322,4.836)
\psline(2.478,4.164)(2.922,4.164)
\psline(2.478,4.236)(2.922,4.236)
\psline(3.078,4.764)(3.522,4.764)
\psline(3.078,4.836)(3.522,4.836)
\psline(3.078,4.164)(3.522,4.164)
\psline(3.078,4.236)(3.522,4.236)
\psline(4.278,4.764)(4.722,4.764)
\psline(4.278,4.836)(4.722,4.836)
\psline(3.678,4.164)(4.122,4.164)
\psline(3.678,4.236)(4.122,4.236)
\psline(4.878,2.964)(5.322,2.964)
\psline(4.878,3.036)(5.322,3.036)
\psline(6.078,2.964)(6.522,2.964)
\psline(6.078,3.036)(6.522,3.036)
\psline(5.478,2.364)(5.922,2.364)
\psline(5.478,2.436)(5.922,2.436)
\psline(6.078,2.364)(6.522,2.364)
\psline(6.078,2.436)(6.522,2.436)
\psline(6.678,2.364)(7.122,2.364)
\psline(6.678,2.436)(7.122,2.436)
\psline[linecolor=blue](3.086,3.6)(3.514,3.6)
\psline[linecolor=blue](3.686,3.6)(4.114,3.6)
\psline[linecolor=blue](4.261,3.661)(4.739,4.139)
\psline[linecolor=blue](4.886,4.2)(5.314,4.2)
\psline[linecolor=blue](5.486,4.2)(5.914,4.2)
\psline[linecolor=blue](6.086,4.2)(6.514,4.2)
\psline[linecolor=blue](4.286,3.6)(4.714,3.6)
\psline[linecolor=blue](4.861,3.661)(5.339,4.139)
\psline[linecolor=blue](6.086,1.8)(6.514,1.8)
\psline[linecolor=blue](6.686,1.8)(7.114,1.8)
\psline[linecolor=gray](2.461,4.139)(2.939,3.661)
\psline[linecolor=gray](3.061,4.139)(3.539,3.661)
\psline[linecolor=gray](3.661,4.139)(4.139,3.661)
\psline[linecolor=gray](4.261,4.739)(4.739,4.261)
\psline[linecolor=gray](4.261,4.139)(4.739,3.661)
\psline[linecolor=gray](4.861,4.739)(5.339,4.261)
\psline[linecolor=gray](5.461,4.739)(5.939,4.261)
\psline[linecolor=gray](5.461,2.339)(5.939,1.861)
\psline[linecolor=gray](6.061,2.339)(6.539,1.861)
\psline[linecolor=gray](6.661,2.339)(7.139,1.861)
\psline[linestyle=dashed](3.661,3.539)(4.139,3.061)
\psline[linestyle=dashed](4.261,3.539)(4.739,3.061)
\psline[linestyle=dashed](4.861,3.539)(5.339,3.061)
\psline[linestyle=dashed](4.827,4.118)(5.373,2.482)
\psline[linestyle=dashed](5.427,4.118)(5.973,2.482)
\psline[linestyle=dashed](6.027,4.118)(6.573,2.482)
\psline[linestyle=dashed](6.627,4.118)(7.173,2.482)
\psline[linestyle=dashed](5.417,4.716)(5.983,1.884)
\psline[linestyle=dashed](6.021,4.117)(6.579,1.883)
\psline[linestyle=dashed](6.621,4.117)(7.179,1.883)
\psline[linestyle=dashed](6.58,4.116)(7.149,1.27)
\psline[linestyle=dashed](6.651,4.13)(7.22,1.284)

\end{pspicture*}
}

\subsection{Case (2)} We give the quantum hasse diagram in case (2).

\centerline{\begin{pspicture*}(0,0)(8.4,4.8)

\psellipse[fillstyle=solid,fillcolor=black](1.2,3.6)(0.082,0.082)
\psellipse[fillstyle=solid,fillcolor=black](1.8,3.6)(0.082,0.082)
\psellipse[fillstyle=solid,fillcolor=black](2.4,3.6)(0.082,0.082)
\psellipse[fillstyle=solid,fillcolor=black](3.0,3.6)(0.082,0.082)
\psellipse[linecolor=blue,fillstyle=solid,fillcolor=blue](3.0,3.0)(0.082,0.082)
\psellipse[fillstyle=solid,fillcolor=black](3.6,3.6)(0.082,0.082)
\psellipse[linecolor=blue,fillstyle=solid,fillcolor=blue](3.6,3.0)(0.082,0.082)
\psellipse[fillstyle=solid,fillcolor=black](4.2,3.6)(0.082,0.082)
\psellipse[linecolor=blue,fillstyle=solid,fillcolor=blue](4.2,3.0)(0.082,0.082)
\psellipse[linecolor=blue,fillstyle=solid,fillcolor=blue](4.8,3.0)(0.082,0.082)
\psellipse[linecolor=blue,fillstyle=solid,fillcolor=blue](4.8,2.4)(0.082,0.082)
\psellipse[linecolor=blue,fillstyle=solid,fillcolor=blue](5.4,2.4)(0.082,0.082)
\psellipse[fillstyle=solid,fillcolor=black](5.4,1.8)(0.082,0.082)
\psellipse[linecolor=blue,fillstyle=solid,fillcolor=blue](6.0,2.4)(0.082,0.082)
\psellipse[fillstyle=solid,fillcolor=black](6.0,1.8)(0.082,0.082)
\psellipse[linecolor=blue,fillstyle=solid,fillcolor=blue](6.6,2.4)(0.082,0.082)
\psellipse[fillstyle=solid,fillcolor=black](6.6,1.8)(0.082,0.082)
\psellipse[fillstyle=solid,fillcolor=black](7.2,1.8)(0.082,0.082)
\psellipse[linecolor=blue,fillstyle=solid,fillcolor=blue](7.2,1.2)(0.082,0.082)
\psline(1.282,3.6)(1.718,3.6)
\psline(1.882,3.6)(2.318,3.6)
\psline(3.082,3.6)(3.518,3.6)
\psline(3.682,3.6)(4.118,3.6)
\psline(5.482,1.8)(5.918,1.8)
\psline(6.082,1.8)(6.518,1.8)
\psline(6.674,1.764)(7.126,1.764)
\psline(6.674,1.836)(7.126,1.836)
\psline(2.474,3.564)(2.926,3.564)
\psline(2.474,3.636)(2.926,3.636)
\psline[linecolor=blue](3.082,3.0)(3.518,3.0)
\psline[linecolor=blue](3.682,3.0)(4.118,3.0)
\psline[linecolor=blue](4.282,3.0)(4.718,3.0)
\psline[linecolor=blue](4.858,2.942)(5.342,2.458)
\psline[linecolor=blue](5.482,2.4)(5.918,2.4)
\psline[linecolor=blue](6.082,2.4)(6.518,2.4)
\psline[linecolor=blue](4.258,2.942)(4.742,2.458)
\psline[linecolor=blue](4.882,2.4)(5.318,2.4)
\psline[linecolor=gray](2.458,3.542)(2.942,3.058)
\psline[linecolor=gray](3.058,3.542)(3.542,3.058)
\psline[linecolor=gray](3.658,3.542)(4.142,3.058)
\psline[linecolor=gray](4.258,3.542)(4.742,3.058)
\psline[linecolor=gray](6.658,1.742)(7.142,1.258)
\psline[linestyle=dashed](4.858,2.342)(5.342,1.858)
\psline[linestyle=dashed](5.458,2.342)(5.942,1.858)
\psline[linestyle=dashed](6.058,2.342)(6.542,1.858)
\psline[linestyle=dashed](6.658,2.342)(7.142,1.858)

\end{pspicture*}
}

\subsection{Case (3), $n = 3 = m$} We give the quantum hasse diagram in case (3) for $n = 3 = m$.

\centerline{\begin{pspicture*}(0,0)(8.4,7.2)

\psellipse[fillstyle=solid,fillcolor=black](1.2,6.0)(0.079,0.079)
\psellipse[fillstyle=solid,fillcolor=black](1.8,6.0)(0.079,0.079)
\psellipse[fillstyle=solid,fillcolor=black](2.4,6.0)(0.079,0.079)
\psellipse[linecolor=blue,fillstyle=solid,fillcolor=blue](2.4,4.8)(0.079,0.079)
\psellipse[fillstyle=solid,fillcolor=black](3.0,6.0)(0.079,0.079)
\psellipse[fillstyle=solid,fillcolor=black](3.0,5.4)(0.079,0.079)
\psellipse[linecolor=blue,fillstyle=solid,fillcolor=blue](3.0,4.8)(0.079,0.079)
\psellipse[fillstyle=solid,fillcolor=black](3.6,5.4)(0.079,0.079)
\psellipse[linecolor=blue,fillstyle=solid,fillcolor=blue](3.6,4.8)(0.079,0.079)
\psellipse[linecolor=blue,fillstyle=solid,fillcolor=blue](3.6,4.2)(0.079,0.079)
\psellipse[fillstyle=solid,fillcolor=black](4.2,5.4)(0.079,0.079)
\psellipse[linecolor=blue,fillstyle=solid,fillcolor=blue](4.2,4.8)(0.079,0.079)
\psellipse[linecolor=blue,fillstyle=solid,fillcolor=blue](4.2,4.2)(0.079,0.079)
\psellipse[fillstyle=solid,fillcolor=black](4.2,3.6)(0.079,0.079)
\psellipse[fillstyle=solid,fillcolor=black](4.8,5.4)(0.079,0.079)
\psellipse[linecolor=blue,fillstyle=solid,fillcolor=blue](4.8,4.8)(0.079,0.079)
\psellipse[linecolor=blue,fillstyle=solid,fillcolor=blue](4.8,4.2)(0.079,0.079)
\psellipse[fillstyle=solid,fillcolor=black](4.8,3.6)(0.079,0.079)
\psellipse[linecolor=blue,fillstyle=solid,fillcolor=blue](5.4,4.8)(0.079,0.079)
\psellipse[linecolor=blue,fillstyle=solid,fillcolor=blue](5.4,4.2)(0.079,0.079)
\psellipse[fillstyle=solid,fillcolor=black](5.4,3.6)(0.079,0.079)
\psellipse[linecolor=blue,fillstyle=solid,fillcolor=blue](5.4,2.4)(0.079,0.079)
\psellipse[linecolor=blue,fillstyle=solid,fillcolor=blue](6.0,4.8)(0.079,0.079)
\psellipse[fillstyle=solid,fillcolor=black](6.0,3.6)(0.079,0.079)
\psellipse[fillstyle=solid,fillcolor=black](6.0,3.0)(0.079,0.079)
\psellipse[linecolor=blue,fillstyle=solid,fillcolor=blue](6.0,2.4)(0.079,0.079)
\psellipse[linecolor=blue,fillstyle=solid,fillcolor=blue](6.6,4.8)(0.079,0.079)
\psellipse[fillstyle=solid,fillcolor=black](6.6,3.0)(0.079,0.079)
\psellipse[linecolor=blue,fillstyle=solid,fillcolor=blue](6.6,2.4)(0.079,0.079)
\psellipse[linecolor=blue,fillstyle=solid,fillcolor=blue](6.6,1.8)(0.079,0.079)
\psellipse[fillstyle=solid,fillcolor=black](7.2,3.0)(0.079,0.079)
\psellipse[linecolor=blue,fillstyle=solid,fillcolor=blue](7.2,2.4)(0.079,0.079)
\psellipse[linecolor=blue,fillstyle=solid,fillcolor=blue](7.2,1.8)(0.079,0.079)
\psellipse[fillstyle=solid,fillcolor=black](7.2,1.2)(0.082,0.082)
\psline(1.279,6.0)(1.721,6.0)
\psline(2.456,5.944)(2.944,5.456)
\psline(3.056,5.944)(3.544,5.456)
\psline(4.279,5.4)(4.721,5.4)
\psline(4.279,3.6)(4.721,3.6)
\psline(5.456,3.544)(5.944,3.056)
\psline(6.056,3.544)(6.544,3.056)
\psline(1.87,5.964)(2.33,5.964)
\psline(1.87,6.036)(2.33,6.036)
\psline(2.47,5.964)(2.93,5.964)
\psline(2.47,6.036)(2.93,6.036)
\psline(3.07,5.364)(3.53,5.364)
\psline(3.07,5.436)(3.53,5.436)
\psline(3.67,5.364)(4.13,5.364)
\psline(3.67,5.436)(4.13,5.436)
\psline(4.87,3.564)(5.33,3.564)
\psline(4.87,3.636)(5.33,3.636)
\psline(5.47,3.564)(5.93,3.564)
\psline(5.47,3.636)(5.93,3.636)
\psline(6.07,2.964)(6.53,2.964)
\psline(6.07,3.036)(6.53,3.036)
\psline(6.67,2.964)(7.13,2.964)
\psline(6.67,3.036)(7.13,3.036)
\psline[linecolor=blue](2.479,4.8)(2.921,4.8)
\psline[linecolor=blue](3.079,4.8)(3.521,4.8)
\psline[linecolor=blue](3.679,4.8)(4.121,4.8)
\psline[linecolor=blue](4.256,4.744)(4.744,4.256)
\psline[linecolor=blue](4.879,4.2)(5.321,4.2)
\psline[linecolor=blue](5.456,4.256)(5.944,4.744)
\psline[linecolor=blue](6.079,4.8)(6.521,4.8)
\psline[linecolor=blue](3.056,4.744)(3.544,4.256)
\psline[linecolor=blue](3.679,4.2)(4.121,4.2)
\psline[linecolor=blue](4.256,4.256)(4.744,4.744)
\psline[linecolor=blue](4.879,4.8)(5.321,4.8)
\psline[linecolor=blue](5.479,4.8)(5.921,4.8)
\psline[linecolor=blue](3.656,4.256)(4.144,4.744)
\psline[linecolor=blue](4.856,4.744)(5.344,4.256)
\psline[linecolor=blue](5.479,2.4)(5.921,2.4)
\psline[linecolor=blue](6.079,2.4)(6.521,2.4)
\psline[linecolor=blue](6.679,2.4)(7.121,2.4)
\psline[linecolor=blue](6.056,2.344)(6.544,1.856)
\psline[linecolor=blue](6.656,1.856)(7.144,2.344)
\psline[linecolor=blue](6.679,1.8)(7.121,1.8)
\psline[linecolor=blue](4.27,4.764)(4.73,4.764)
\psline[linecolor=blue](4.27,4.836)(4.73,4.836)
\psline[linecolor=blue](4.27,4.164)(4.73,4.164)
\psline[linecolor=blue](4.27,4.236)(4.73,4.236)
\psline[linecolor=gray](1.835,5.929)(2.365,4.871)
\psline[linecolor=gray](2.435,5.929)(2.965,4.871)
\psline[linecolor=gray](3.035,5.329)(3.565,4.271)
\psline[linecolor=gray](3.035,5.929)(3.565,4.871)
\psline[linecolor=gray](3.656,5.344)(4.144,4.856)
\psline[linecolor=gray](4.256,5.344)(4.744,4.856)
\psline[linecolor=gray](4.856,5.344)(5.344,4.856)
\psline[linecolor=gray](4.835,3.529)(5.365,2.471)
\psline[linecolor=gray](5.435,3.529)(5.965,2.471)
\psline[linecolor=gray](6.035,2.929)(6.565,1.871)
\psline[linecolor=gray](6.035,3.529)(6.565,2.471)
\psline[linecolor=gray](6.656,2.944)(7.144,2.456)
\psline[linecolor=gray](7.118,1.2)(7.118,1.2)
\psline[linestyle=dashed](3.635,4.729)(4.165,3.671)
\psline[linestyle=dashed](4.235,4.729)(4.765,3.671)
\psline[linestyle=dashed](4.835,4.729)(5.365,3.671)
\psline[linestyle=dashed](5.435,4.729)(5.965,3.075)
\psline[linestyle=dashed](6.025,4.725)(6.575,3.075)
\psline[linestyle=dashed](6.625,4.725)(7.175,3.075)
\psline[linestyle=dashed](4.825,4.125)(5.375,2.475)
\psline[linestyle=dashed](5.425,4.125)(5.975,2.475)
\psline[linestyle=dashed](6.015,4.723)(6.585,1.877)
\psline[linestyle=dashed](6.615,4.723)(7.185,1.877)
\psline[linestyle=dashed](5.435,4.129)(5.965,3.671)
\psline[linestyle=dashed](6.635,2.329)(7.163,1.273)

\end{pspicture*}
} 

\subsection{Case (5)} We give the quantum hasse diagram in case (5).

\centerline{\begin{pspicture*}(0,0)(7.2,4.8)

\psellipse[fillstyle=solid,fillcolor=black](1.2,3.6)(0.084,0.084)
\psellipse[fillstyle=solid,fillcolor=black](1.8,3.6)(0.084,0.084)
\psellipse[fillstyle=solid,fillcolor=black](2.4,3.6)(0.084,0.084)
\psellipse[linecolor=blue,fillstyle=solid,fillcolor=blue](2.4,3.0)(0.084,0.084)
\psellipse[fillstyle=solid,fillcolor=black](3.0,3.6)(0.084,0.084)
\psellipse[linecolor=blue,fillstyle=solid,fillcolor=blue](3.0,3.0)(0.084,0.084)
\psellipse[fillstyle=solid,fillcolor=black](3.6,3.6)(0.084,0.084)
\psellipse[linecolor=blue,fillstyle=solid,fillcolor=blue](3.6,3.0)(0.084,0.084)
\psellipse[fillstyle=solid,fillcolor=black](3.6,2.4)(0.084,0.084)
\psellipse[fillstyle=solid,fillcolor=black](4.2,3.6)(0.084,0.084)
\psellipse[linecolor=blue,fillstyle=solid,fillcolor=blue](4.2,3.0)(0.084,0.084)
\psellipse[fillstyle=solid,fillcolor=black](4.2,2.4)(0.084,0.084)
\psellipse[linecolor=blue,fillstyle=solid,fillcolor=blue](4.8,3.0)(0.084,0.084)
\psellipse[fillstyle=solid,fillcolor=black](4.8,2.4)(0.084,0.084)
\psellipse[linecolor=blue,fillstyle=solid,fillcolor=blue](4.8,1.8)(0.084,0.084)
\psellipse[linecolor=blue,fillstyle=solid,fillcolor=blue](5.4,3.0)(0.084,0.084)
\psellipse[fillstyle=solid,fillcolor=black](5.4,2.4)(0.084,0.084)
\psellipse[linecolor=blue,fillstyle=solid,fillcolor=blue](5.4,1.8)(0.084,0.084)
\psellipse[fillstyle=solid,fillcolor=black](6.0,2.4)(0.084,0.084)
\psellipse[linecolor=blue,fillstyle=solid,fillcolor=blue](6.0,1.8)(0.084,0.084)
\psellipse[fillstyle=solid,fillcolor=black](6.0,1.2)(0.084,0.084)
\psline(1.284,3.6)(1.716,3.6)
\psline(3.684,3.6)(4.116,3.6)
\psline(3.684,2.4)(4.116,2.4)
\psline(2.476,3.564)(2.924,3.564)
\psline(2.476,3.636)(2.924,3.636)
\psline(4.876,2.364)(5.324,2.364)
\psline(4.876,2.436)(5.324,2.436)
\psline(1.843,3.528)(2.357,3.528)
\psline(1.884,3.6)(2.316,3.6)
\psline(1.843,3.672)(2.357,3.672)
\psline(3.043,3.528)(3.557,3.528)
\psline(3.084,3.6)(3.516,3.6)
\psline(3.043,3.672)(3.557,3.672)
\psline(4.243,2.328)(4.757,2.328)
\psline(4.284,2.4)(4.716,2.4)
\psline(4.243,2.472)(4.757,2.472)
\psline(5.443,2.328)(5.957,2.328)
\psline(5.484,2.4)(5.916,2.4)
\psline(5.443,2.472)(5.957,2.472)
\psline[linecolor=blue](2.484,3.0)(2.916,3.0)
\psline[linecolor=blue](3.084,3.0)(3.516,3.0)
\psline[linecolor=blue](4.284,3.0)(4.716,3.0)
\psline[linecolor=blue](4.884,3.0)(5.316,3.0)
\psline[linecolor=blue](4.884,1.8)(5.316,1.8)
\psline[linecolor=blue](5.484,1.8)(5.916,1.8)
\psline[linecolor=blue](3.676,2.964)(4.124,2.964)
\psline[linecolor=blue](3.676,3.036)(4.124,3.036)
\psline[linecolor=gray](1.859,3.541)(2.341,3.059)
\psline[linecolor=gray](2.459,3.541)(2.941,3.059)
\psline[linecolor=gray](3.059,3.541)(3.541,3.059)
\psline[linecolor=gray](3.659,3.541)(4.141,3.059)
\psline[linecolor=gray](4.259,3.541)(4.741,3.059)
\psline[linecolor=gray](4.259,2.341)(4.741,1.859)
\psline[linecolor=gray](4.859,2.341)(5.341,1.859)
\psline[linecolor=gray](5.459,2.341)(5.941,1.859)
\psline[linestyle=dashed](3.059,2.941)(3.541,2.459)
\psline[linestyle=dashed](3.659,2.941)(4.141,2.459)
\psline[linestyle=dashed](4.259,2.941)(4.741,2.459)
\psline[linestyle=dashed](4.859,2.941)(5.341,2.459)
\psline[linestyle=dashed](5.459,2.941)(5.941,2.459)
\psline[linestyle=dashed](4.227,3.52)(4.773,1.88)
\psline[linestyle=dashed](4.838,2.925)(5.362,1.875)
\psline[linestyle=dashed](5.438,2.925)(5.962,1.875)
\psline[linestyle=dashed](5.39,2.917)(5.942,1.261)
\psline[linestyle=dashed](5.458,2.939)(6.01,1.283)

\end{pspicture*}
}
}


\section{Appendix}

In this appendix, we give the beginning (these are infinite graphs) of the quantum Hasse diagrams in case (1), $n = 3$, case (2), case (3) $n = 3 = m$ and case (5). Each column corresponds to a degree starting from the left with degree zero. The black vertices correspond to the classes $\s'_u$ and where there are two classes of the same degree, the classes are $\s'_u$ and $\s'_{u'}$ and the top class is always $\s'_u$ while the bottom one is $\s'_{u'}$. We use the same convention for blue vertices which correspond to $\tau_v$ classes. Black and blue edges are from the Hasse diagram of $Y$ and $Z$. Grey edges are from the Hasse diagram of $Y$ to the one of $Z$ while dotted edges correspond to quantum multiplication.

\subsection{Case (1), $n = 3$} We give the quantum hasse diagram in case (1) for $n = 3$.

\centerline{}

\subsection{Case (2)} We give the quantum hasse diagram in case (2).

\centerline{}

\subsection{Case (3), $n = 3 = m$} We give the quantum hasse diagram in case (3) for $n = 3 = m$.

\centerline{} 

\subsection{Case (5)} We give the quantum hasse diagram in case (5).

\centerline{}

\bibliographystyle{plain}
\bibliography{biblio-revision}

\end{document}